\documentclass[10pt,twoside,reqno]{amsart}
 \usepackage[top=1.4in, bottom=1.4in, left=1in, right=1in]{geometry}
\usepackage{graphicx}
\usepackage{amsmath,latexsym}
\usepackage{amsfonts,amssymb}
\usepackage{mathtools}
\usepackage{amsmath}
\usepackage{amscd, amsthm}
\usepackage{stmaryrd}
\usepackage{fancyhdr}
\usepackage{commath,enumerate}
\usepackage{indentfirst, hyperref}
\usepackage{caption, subcaption}

\allowdisplaybreaks[4]

\usepackage{pgf,tikz}
\usepackage{mathrsfs}
\usetikzlibrary{arrows,calc,positioning}
\usepackage{float}
\numberwithin{figure}{section}
\usepackage[bottom]{footmisc}

\numberwithin{equation}{section}
\newtheorem{theorem}{Theorem}[section]
\newtheorem{lemma}[theorem]{Lemma}
\newtheorem{proposition}[theorem]{Proposition}

\newtheorem{definition}[theorem]{Definition}
\newtheorem{remark}[theorem]{Remark}

\newcommand{\Rm}[1]{
  \textup{\uppercase\expandafter{\romannumeral#1}}
}
\let\olddefinition\definition
\renewcommand{\definition}{\olddefinition\normalfont}
\let\oldremark\remark
\renewcommand{\remark}{\oldremark\normalfont}

\newcommand{\C}{\mathbb{C}}

\newcommand{\F}{\mathcal{F}}

\newcommand{\M}{\mathcal{M}}
\newcommand{\med}{\mathrm{med}}
\newcommand{\n}{\mathbf{n}}
\newcommand{\N}{\mathbb{N}}
\newcommand{\Nc}{\mathcal{N}}
\renewcommand{\O}{\mathcal{O}}

\newcommand{\R}{\mathbb{R}}
\newcommand{\Rc}{\mathcal{R}}

\newcommand{\Tb}{\mathbf{T}}
\newcommand{\ub}{\mathbf{u}}
\newcommand{\x}{\mathbf{x}}
\newcommand{\s}{\mathbf{s}}
\renewcommand{\S}{\mathcal{S}}

\newcommand{\Z}{\mathbb{Z}}

\newcommand{\etab}{\boldsymbol{\eta}}
\renewcommand{\Mc}{\mathcal{M}}
\newcommand{\cutoff}{\mathfrak{b}}
\newcommand{\cutoffxi}{\mathfrak{d}}

\def\vp{\varphi}
\def\ve{\varepsilon}
\def\px{\partial_x}

\newcommand{\diff}{\,\mathrm{d}}

\DeclareMathOperator{\sgn}{\mathrm{sgn}\,}

\def\bel{\begin{equation}\label}
\def\beq{\begin{equation}}
\def\eeq{\end{equation}}
\def\bega{\begin{array}}
\def\enda{\end{array}}

\renewcommand{\vec}[1]{\mathbf{#1}}

\usepackage{scalerel,stackengine}
\stackMath
\newcommand\reallywidehat[1]{%
\savestack{\tmpbox}{\stretchto{%
  \scaleto{%
    \scalerel*[\widthof{\ensuremath{#1}}]{\kern-.6pt\bigwedge\kern-.6pt}%
    {\rule[-\textheight/2]{1ex}{\textheight}}
  }{\textheight}%
}{0.5ex}}%
\stackon[1pt]{#1}{\tmpbox}%
}
\author{John K. Hunter}
\address{Department of Mathematics, University of California at Davis}
\email{jkhunter@ucdavis.edu}
\thanks{JKH was supported by the NSF under grant number DMS-1616988 and DMS-1908947}
\author{Jingyang Shu}
\address{Department of Mathematics, University of California at Davis}
\email{jyshu@ucdavis.edu}
\author{Qingtian Zhang}
\address{Department of Mathematics, West Virginia University}
\email{qingtian.zhang@mail.wvu.edu}
\title[GSQG front equations]{
Global solutions for a family of GSQG front equations}
\date{\today}

\begin{document}

\begin{abstract}
We prove the global existence of solutions with small and smooth initial data of a nonlinear dispersive equation for the motion of generalized surface quasi-geostrophic (GSQG) fronts in a parameter regime $1<\alpha<2$, where $\alpha=1$ corresponds to the SQG equation and $\alpha=2$ corresponds to the incompressible Euler equations. This result completes previous global well-posedness results for $0<\alpha \le 1$. We also use contour dynamics to derive the GSQG front equations for $1<\alpha<2$.
\end{abstract}

\maketitle

\tableofcontents

\section{Introduction}

The inviscid generalized surface quasi-geostrophic (GSQG) equation is a two-dimensional transport equation for an active scalar
\begin{align}
& \theta_t(\x, t) + \ub(\x, t) \cdot \nabla \theta(\x, t) = 0,\label{gsqg1}\\
& (- \Delta)^{\alpha / 2}\ub(\x, t) = \nabla^\perp \theta(\x, t),\label{gsqg2}
\end{align}
where $0 < \alpha \leq 2$ is a parameter. Here, the scalar field $\theta \colon \R^2 \times \R_+ \to \R$ is transported by the divergence-free  velocity $\ub \colon \R^2 \times \R_+ \to \R^2$, which is determined nonlocally from $\theta$ by \eqref{gsqg2}, $\x = (x, y)$ is the spatial variable, $\nabla^\perp = (- \partial_y, \partial_x)$, and $(-\Delta)^{\alpha / 2}$ is a
fractional Laplacian.
When $\alpha = 2$, equations \eqref{gsqg1}--\eqref{gsqg2} correspond to the streamfunction-vorticity formulation of the two-dimensional incompressible Euler equations \cite{MB02}.
When $\alpha = 1$, these equations are the surface quasi-geostrophic (SQG) equation, which arises from oceanic and atmospheric science \cite{Lap17, Ped87} and has mathematical similarities with the three-dimensional incompressible Euler equations \cite{CMT94a, CMT94b}.

The transport equation \eqref{gsqg1} is compatible with piecewise-constant solutions for $\theta$, and the simplest class of such solutions is
obtained when $\theta$ takes on two distinct values. As in \cite{HSZ19pa}, we distinguish between different geometries. We refer to patch solutions when one of the values is $0$ and the support of $\theta$ is a simply connected, bounded  set whose boundary is a simple closed curve; we refer to front solutions when the two different values of $\theta$ are taken on in half-spaces whose common boundary is a curve, or front, with infinite length.
In this paper we consider front solutions. The advantage of these solutions over patch solutions is that their boundary geometry is simpler,
especially when the front can be represented as a graph, although the lack of compact support in $\theta$ introduces additional complications in the formulation of front equations.

Zabusky et.~al.~\cite{Zab} introduced contour dynamics  for Euler patches, which leads to a closed equation for the evolution of the patch boundary, and smooth patch solutions of the Euler equations exist globally in time \cite{BeCo93, Che93, Che98}. Similar methods lead to contour dynamics equations for the boundary of SQG and GSQG patches.
Local-in-time patch solutions of the SQG and GSQG equations are analyzed in \cite{CCGSMZ14, CCCGW12, CCG18, Gan08, GP18p, KR20pa, KR20pb, KYZ17, MB02},
but a proof of whether the patch boundaries form singularities in finite time or stay globally smooth is an open question. There is, however, numerical evidence suggesting finite-time singularity formation in SQG and GSQG patches \cite{CFMR05, SD14, SD19} and a proof of finite-time singularity formation for GSQG patches in the presence of a boundary for $\alpha$ strictly less than $2$ and sufficiently close to $2$ \cite{GP18p, KRYZ16}. In addition, some particular global smooth solutions for rotating patches are constructed and studied in \cite{CCGS16a, CCGS16b, dlHHH16, GS18, GSPSY19p}.

Smooth and analytic solutions for spatially periodic SQG fronts are proved to be locally well-posed in \cite{FR11, Ro05}, and almost-sharp SQG fronts are studied in \cite{CFR04, FLR12, FR12, FR15}. Contour dynamics equations for GSQG fronts with $0 < \alpha < 1$ are straightforward to derive because the standard potential representation for $\vec{u}$ converges even though $\theta$ does not have compact support,
and Cordoba et.~al. \cite{CGI19} prove that flat planar fronts are asymptotically stable in that case. However, the same derivation does not work when $1 \leq \alpha \le 2$ because the Riesz potential \cite{Ste70, Ste93} for $(-\Delta)^{-\alpha / 2}$ decays too slowly at infinity for the straightforward potential representation of $\vec{u}$ to converge.

A derivation of front contour dynamics equations for $1\le \alpha \le 2$ by a regularization procedure is given in \cite{HS18}, and the same equations are derived in \cite{HSZ19pb} for SQG fronts with $\alpha = 1$ by decomposing the velocity field into a planar shear flow and a perturbation due to the front motion with an absolutely convergent potential representation. We provide a similar derivation of GSQG front equations with $1<\alpha < 2$ in the appendix of the present paper,  and a derivation for Euler fronts with $\alpha = 2$ is given in \cite{HMSZ}. The local well-posedness of  a cubically nonlinear approximation for SQG fronts is proved in \cite{HSZ18}, and flat planar SQG fronts are shown to be globally asymptotically stable in \cite{HSZ18p}. A related two-front GSQG problem is studied in \cite{HSZ19pa}.

This paper is concerned with the regime $1 < \alpha < 2$. We assume that the front is a graph $y=\vp(x,t)$ and study piecewise-constant distributional solutions of the GSQG equations \eqref{gsqg1}--\eqref{gsqg2} with
\begin{align}
\label{frontsol}
\theta(\x, t) = \left\{\begin{array}{ll}\theta_+ & \quad \text{if}\ y > \vp(x, t),\\ \theta_- & \quad \text{if}\ y < \vp(x, t).\end{array}\right.
\end{align}
The graph assumption greatly simplifies the evolution equation for the front, but it does not describe wave-breaking
or filamentation of the front. However, for the small-slope fronts we study in this paper, we will show that wave-breaking never occurs.

In Appendix~\ref{app:contour}, we show that for $1<\alpha < 2$ the front location satisfies the evolution equation
\begin{align}
\label{GSQGfront0}
\begin{split}
&\vp_t(x, t) + \Theta A |\px|^{1 - \alpha} \vp_x(x, t)
\\
&\qquad + \Theta \int_\R [\vp_x(x, t) - \vp_x(x + \zeta, t)] \bigg\{\frac{1}{|\zeta|^{2 - \alpha}} - \frac{1}{(\zeta^2 + [\vp(x, t) - \vp(x + \zeta, t)]^2)^{(2 - \alpha) / 2}} \bigg\} \diff{\zeta} = 0,
\end{split}
\end{align}
where
\begin{equation}
\Theta = g_\alpha (\theta_+ - \theta_-),\qquad
A=2 \sin\left(\frac{\pi \alpha}{2}\right) \Gamma(\alpha - 1),
\qquad g_\alpha = \frac{\Gamma(1 - \alpha / 2)}{2^\alpha \pi \Gamma(\alpha / 2)}.
\label{defA}
\end{equation}
Our main result, stated in Theorem~\ref{main}, is that the Cauchy problem on $\R$ for the GSQG front equation \eqref{GSQGfront0} with sufficiently small and smooth initial data has smooth solutions globally in time. Together with \cite{CGI19} for $0<\alpha <1$ and \cite{HSZ18p}
for $\alpha =1$, this completes the proof of asymptotic stability of planar GSQG fronts in the entire range $0 < \alpha < 2$. Our proof follows the ones in \cite{CGI19,HSZ18p}.

We remark that the Euler front equations with $\alpha =2$ are nondispersive \cite{BH10, HS18}, and --- in the absence of dispersive decay --- one cannot expect to get global smooth solutions. In that case, numerical solutions of the full contour dynamics equations \cite{BHnum} indicate that the graphical description of the front may fail in finite time, after which the front breaks and forms extremely thin filaments, similar to the ones that are observed in patches \cite{Dr1, Dr2}.

The rest of the paper is organized as follows. In Section \ref{sec:prelim}, we review results from para-differential calculus and state some estimates for multilinear Fourier multipliers. In Section \ref{sec:para}, we carry out a multilinear expansion of the nonlinearity in \eqref{GSQGfront0} and para-linearize the equation, which enables us to derive improved energy estimates in Section \ref{sec:local}. In Section \ref{sec:global}, we state the main global theorem and outline the steps in the proof of the theorem. Finally, in Sections \ref{sec:sharp}--\ref{sec:nonlindisp} we provide the proofs of each step.

\section{Preliminaries}
\label{sec:prelim}

In this section, we summarize some notation and lemmas that we will use below.

\subsection{Para-differential calculus}
We denote the Fourier transform of $f \colon \R\to \C$ by $\hat f \colon \R\to \C$, where $\hat f= \F f$ is given by
\[
f(x)=\int_{\R} \hat f(\xi) e^{i\xi x} \diff\xi,  \qquad \hat f(\xi)=\frac1{2\pi} \int_{\R}f(x) e^{-i\xi x}\diff{x}.
\]
For $s\in \R$, we denote by $H^s(\R)$ the space of Schwartz distributions $f$ with $\|f\|_{H^s} < \infty$, where
\[
\|f\|_{H^s} = \left[\int_\R (1+\xi^2)^s |\hat{f}(\xi)|^2\, \diff \xi\right]^{1/2}.
\]

Throughout this paper, we use $A\lesssim B$ to mean there is a constant $C$ such that $A\leq C B$, and $A\gtrsim B$ to mean there is a constant $C$ such that $A\geq C B$. We remark that the constant $C$ may depend on $\alpha$. We use $A\approx B$ to mean that $A\lesssim B$ and $B\lesssim A$.
The notation $\O(f)$ denotes a term satisfying
\[
\|\O(f)\|_{{H}^s}\lesssim \|f\|_{{H}^s}
\]
whenever there exists  $s\in \R$ such that $f\in {H}^s$. We also use $O(f)$ to denote a term satisfying $|O(f)|\lesssim|f|$ pointwise.

Let $\chi \colon \R \to \R$ be a smooth function supported in the interval $\{\xi\in \R \mid |\xi|\leq 1/10\}$
and equal to $1$  on $\{\xi\in \R \mid |\xi|\leq 3/40\}$.
If $f$ is a Schwartz distribution on $\R$ and $a \colon \R \times \R \to \C$ is a symbol, then
we define the Weyl para-product $T_a f$ by
\begin{equation}
\label{weyldef}
\F \left[T_a f\right](\xi)= \int_{\R} \chi\left(\frac{|\xi-\eta|}{|\xi+\eta|}\right) \tilde{a}\Big(\xi-\eta, \frac{\xi + \eta}{2}\Big)\hat f(\eta)\diff\eta,
\end{equation}
where $\tilde{a}(\xi,\eta)$ denotes the partial Fourier transform of $a(x, \eta) $ with respect to $x$.

For $r_1, r_2 \in \N_0$, we define a normed symbol space by
\begin{align*}
\Mc_{(r_1, r_2)} &= \{a \colon \R \times \R \to \C : \|a\|_{\Mc_{(r_1, r_2)}} < \infty\},
\\
\|a\|_{\Mc_{(r_1, r_2)}} &= \sup_{(x,\eta) \in \R^2} \left\{\sum_{\alpha=0}^ {r_1} \sum_{\beta=0}^{r_2}|\eta|^{\beta} \big|\partial_\eta^\beta \px^\alpha a(x, \eta)\big|\right\}.
\end{align*}
If $a \in \Mc_{(0, 0)}$ and $f \in {L}^p$, with $1\le p \le \infty$, then $T_af\in L^p$ and
\[
\|T_a f\|_{L^p} \lesssim \|a\|_{\Mc_{(0,0)}} \|f\|_{L^p}.
\]
In particular, if $a\in\Mc_{(0, 0)}$ is real-valued, then $T_a$ is a self-adjoint, bounded linear operator on $L^2$.

\subsection{Fourier multipliers}

Let $\psi \colon\R\to [0,1]$ be a smooth function supported in $[-8/5, 8/5]$ and equal to $1$ in $[-5/4, 5/4]$.
For any $k\in \mathbb Z$, we define
\begin{align}
\label{defpsik}
\begin{split}
\psi_k(\xi)&=\psi(\xi/2^k)-\psi(\xi/2^{k-1}), \qquad \psi_{\leq k}(\xi)=\psi(\xi/2^k),\qquad \psi_{\geq k}(\xi)=1-\psi(\xi/2^{k-1}),\\
\tilde\psi_k(\xi)&=\psi_{k-1}(\xi)+\psi_k(\xi)+\psi_{k+1}(\xi),
\end{split}
\end{align}
and denote by $P_k$, $P_{\leq k}$, $P_{\geq k}$, and $\tilde{P}_k$  the Fourier multiplier operators with symbols $\psi_k, \psi_{\leq k}, \psi_{\geq k}$, and $\tilde{\psi}_k$, respectively. Notice that $\psi_k(\xi)=\psi_0(\xi/2^k)$, $\tilde\psi_k(\xi)=\tilde\psi_0(\xi/2^k)$,
and
\begin{equation}
\label{psi-L2}
 \|\psi_k\|_{L^2}\approx  2^{k/2}, \qquad  \|\psi_k'\|_{L^2}\approx 2^{-k/2}.
\end{equation}
The proof of the following interpolation lemma can be found in \cite{IPu16}.
\begin{lemma}\label{interpolation}
For any $k\in\mathbb Z$ and $f\in L^2(\R)$, we have
\[
\|\widehat{P_kf}\|_{L^\infty}^2\lesssim \|P_k f\|_{L^1}^2\lesssim 2^{-k}\|\hat f\|_{L^2_\xi}\left[2^k\|\partial_\xi\hat f\|_{L^2_\xi}+\|\hat f\|_{L^2_\xi}\right].
\]
\end{lemma}

Next, we state an estimate for multilinear Fourier multipliers proved in \cite{IP15}.  Define the class of symbols
\begin{align}
\label{Sinf}
S^\infty:=\{\kappa \colon \R^d\to \C,\quad \kappa \text{ continuous and } \|\kappa\|_{S^\infty}:=\|\F^{-1}(\kappa)\|_{L^1}<\infty\},
\end{align}
and given $\kappa \in S^\infty$, define a multilinear operator $M_\kappa$ acting on Schwartz functions  $f_1,\dotsc, f_m \in \mathcal{S}(\R)$ by
\begin{align*}
M_{\kappa}(f_1,\dotsc,f_m)(x)=\int_{\R^{m}} e^{ix(\xi_1+\dotsb+\xi_m)}\kappa(\xi_1, \dotsc, \xi_m)\hat f_1(\xi_1) \dotsm \hat f_m(\xi_m)\diff{\xi_1} \dotsm \diff{\xi_m}.
\end{align*}

\begin{lemma}\label{multilinear}
(i)~If $\kappa_1, \kappa_2\in S^\infty$, then $\kappa_1\kappa_2\in S^\infty$.

(ii)~Suppose that $1 \le p_1, \dotsc, p_m\leq \infty$, $1 \le p \le \infty$, satisfy
\[
\frac1{p_1}+\frac1{p_2}+ \dotsb +\frac1{p_m}=\frac1p.
\]
 If $\kappa \in S^\infty$, then
\[
\|M_{\kappa}\|_{L^{p_1}\times \dotsb \times L^{p_m}\to L^p}\lesssim \|\kappa\|_{S^\infty}.
\]

(iii)~ Assume $p, q, r\in [1, \infty]$ satisfy $1/p+1/q + 1/r=1$, and $m \in S^\infty_{\eta_1, \eta_2} L^\infty_\xi$. Then, for any $f \in L^p(\R)$, $g \in L^q(\R)$, and $h \in L^r(\R)$,
\[
\left\|\int_{\R^2} m(\eta_1,\eta_2, \xi)\hat f(\eta_1)\hat g(\eta_2)\hat h(\xi-\eta_1-\eta_2) \diff\eta_1\diff\eta_2 \right\|_{L^\infty_\xi}\lesssim \|m\|_{S^\infty_{\eta_1,\eta_2}L^\infty_\xi}\|f\|_{L^p}\|g\|_{L^q}\|h\|_{L^r}.
\]
\end{lemma}

By interpolation, we can estimate the $S^{\infty}$-norm of a symbol $m(\eta_1,\eta_2)\in C_c^\infty$ by
\begin{equation}\label{SymEst}
\begin{aligned}
\|m\|_{S^\infty}\lesssim  \|m\|_{L^1}^{1/4} \|\partial_{\eta_2}^2m\|_{L^1}^{1/2}\|\partial_{\eta_1}^2\partial_{\eta_2}^2m\|_{L^1}^{1/4},\\
\|m\|_{S^\infty}\lesssim  \|m\|_{L^1}^{1/4}\|\partial_{\eta_1}^2m\|_{L^1}^{1/2}\|\partial_{\eta_1}^2\partial_{\eta_2}^2m\|_{L^1}^{1/4}.
\end{aligned}
\end{equation}

\section{Para-linearization of the equation}
\label{sec:para}
In this section, we para-linearize the front equation \eqref{GSQGfront0}; the final result is given in \eqref{Para-eq}.
Without loss of generality, we fix $\Theta = -1$ in the following.

Assuming that $|\vp_x| \ll 1$, we first carry out a multilinear expansion of the nonlinear term,
in a similar way to \cite{CGI19, HSZ18p}. Omitting the details, we find that  \eqref{GSQGfront0} can be written as
\begin{equation}\label{GSQGfront}
\begin{aligned}
& \varphi_t(x, t)  - A |\px|^{1 - \alpha} \varphi_x(x, t) \\
& \quad + \sum_{n = 1}^\infty \frac{c_n}{2n + 1} \px \int_{\R^{2n + 1}} \Tb_n(\etab_n) \hat{\varphi}(\eta_1, t) \hat{\varphi}(\eta_2, t) \dotsm \hat{\varphi}(\eta_{2n + 1}, t) e^{i (\eta_1 + \eta_2 + \dotsb + \eta_{2n + 1}) x} \diff{\etab_n}
= 0,
\end{aligned}
\end{equation}
where ${\etab}_n = (\eta_1, \eta_2, \dotsc, \eta_{2n})$, and the symbol $\Tb_n$ is given by
\begin{align}
\label{defTn}
\Tb_n(\etab_n) = \int_{\R} \frac{\prod_{j = 1}^{2n + 1} \left(1 - e^{i \eta_j \zeta}\right)}{\zeta^{2n + 1}} |\zeta|^{\alpha - 1} \sgn{\zeta} \diff{\zeta}, \qquad c_n = \frac{\Gamma\left(\frac{\alpha}{2}\right)}{\Gamma(n + 1) \Gamma\left(\frac{\alpha}{2} - n\right)}.
\end{align}
We then adapt Proposition 5.2 of \cite{HSZ19pa}, with improved remainder estimates, to the case $1 < \alpha < 2$.
Fixing $\alpha$, we use $C(n,s)$ to denote a generic constant depending on $n,s\in \N$.

\begin{proposition}
\label{varphinonlindecomp01}
Let $1 < \alpha < 2$. Suppose that $\varphi(\cdot, t) \in H^s(\R)$ where $s \geq 4$ is an integer, and $\|\varphi\|_{W^{3, \infty}} $ is sufficiently small. Then
\begin{align*}
& \sum_{n = 1}^\infty \frac{c_n}{2n + 1} \px \int_{\R^{2n + 1}} \Tb_n(\etab_n) \hat{\varphi}(\eta_1, t) \hat{\varphi}(\eta_2, t) \dotsm \hat{\varphi}(\eta_{2n + 1}, t) e^{i (\eta_1 + \eta_2 + \dotsb + \eta_{2n + 1}) x} \diff{\etab_n}\\
&\qquad = \px T_{B^0[\varphi]} \varphi(x, t)  + \px |\px|^{1 - \alpha} T_{B^{1 - \alpha}[\varphi]} \varphi(x, t) + \Rc,
\end{align*}
where the symbols $B^0[\varphi]$ and $B^{1 - \alpha}[\varphi]$are defined by
\begin{align*}
\begin{split}
B^0[\varphi](\cdot, \xi) &=  - \sum_{n = 1}^\infty B^0_n[\varphi](\cdot, \xi),\qquad
B^{1 - \alpha}[\varphi](\cdot, \xi) =  - \sum_{n = 1}^\infty B^{1 - \alpha}_n[\varphi](\cdot, \xi),\\
B^0_n[\varphi](\cdot, \xi) &= A c_n \F_{\zeta}^{-1}\Bigg[\int_{\R^{2n}} \delta\bigg(\zeta - \sum_{j = 1}^{2n} \eta_j\bigg) \prod_{j = 1}^{2n} \bigg(i \eta_j \hat{\varphi}(\eta_j) \chi\Big(\frac{(2n + 1) \eta_j}{\xi}\Big)\bigg) \int_{[0, 1]^{2n}} \bigg|\sum_{j = 1}^{2n} \eta_j s_j\bigg|^{1 - \alpha} \diff{\hat{\s}_n} \diff{\hat{\etab}_n}\Bigg],
\\
B^{1 - \alpha}_n[\varphi](\cdot, \xi) &= - A c_n \F_{\zeta}^{-1}\Bigg[\int_{\R^{2n}} \delta\bigg(\zeta - \sum_{j = 1}^{2n} \eta_j\bigg) \prod_{j = 1}^{2n} \bigg(i \eta_j \hat{\varphi}(\eta_j) \chi\Big(\frac{(2n + 1) \eta_j}{\xi}\Big)\bigg) \diff{\hat{\etab}_n}\Bigg],
\end{split}
\end{align*}
where $A$ is defined in \eqref{defA}, and and the remainder term $\Rc$ satisfies
\begin{align}
\label{R1est01}
\|\Rc\|_{H^s} \lesssim \|\vp\|_{H^s} \sum_{n = 1}^\infty C(n, s) |c_n| \|\vp_x\|_{W^{1, \infty}}^{2n}.
\end{align}
Here, $\chi$ is the cutoff function in \eqref{weyldef}, $\hat{\etab}_n = (\eta_1, \eta_2, \dotsc, \eta_{2n})$, and $\hat{\s}_n = (s_1, s_2, \dotsc, s_{2n})$. The operators $T_{B^{1 - \alpha}[\varphi]}$ and $T_{B^0[\varphi]}$ are self-adjoint and they satisfy the estimates
\begin{align}
\label{B-est01}
\begin{split}
\|B^{1 - \alpha}[\vp]\|_{\M_{(1, 1)}} \lesssim \sum_{n = 1}^\infty C(n, s) |c_n| \|\vp_x\|_{W^{1, \infty}}^{2n},\\
\|B^0[\vp]\|_{\M_{(1, 1)}} \lesssim \sum_{n = 1}^\infty C(n, s) |c_n| \|\vp_x\|_{W^{1, \infty}}^{2n},
\end{split}
\end{align}
\end{proposition}

\begin{proof}
Let
\[
f_n(x)= \int_{\R^{2n + 1}} \Tb_n(\etab_n) \hat{\varphi}(\eta_1, t) \hat{\varphi}(\eta_2, t) \dotsm \hat{\varphi}(\eta_{2n + 1}, t) e^{i (\eta_1 + \eta_2 + \dotsb + \eta_{2n + 1}) x} \diff{\etab_n}.
\]
In view of the commutator estimate
\[
\left\|\px \left[T_{B^{1-\alpha}[\vp]},|\px|^{1-\alpha}\right]\vp\right\|_{H^s}\lesssim \|B^{1-\alpha}[\vp]\|_{\M_{(1,1)}}\|\vp\|_{H^s},
\]
we only need to prove that
\[
 \sum_{n = 1}^\infty \frac{c_n}{2n + 1} \px f_n(x)=\partial_x T_{B^0[\vp]}\vp+\px\left(T_{B^{1-\alpha}[\vp]}|\px|^{1-\alpha}\vp\right)+\Rc,
\]
where $\Rc$ satisfies \eqref{R1est01}, and to do this it suffices to prove for each $n \in \N$ that
\begin{align*}
 &\frac{c_n}{2n + 1} \px f_n(x)= \partial_x T_{B^0_n[\vp]}\vp + \px \left(T_{B^{1-\alpha}_n[\vp]} |\px|^{1 - \alpha} \vp\right)+\Rc_n,
\\
&\|\Rc_n\|_{H^s}\lesssim C(n,s)|c_n| \|\vp_x\|_{W^{1, \infty}}^{2n - 2} \|\vp\|_{H^s}.
\end{align*}

By symmetry, we can assume that $|\eta_{2n+1}|$ is the largest frequency in the expression of $f_n$. Then
\begin{align}
\label{intprod}
\begin{split}
 &\frac{c_n}{2n + 1} \px f_n(x)=c_n\px \int\limits_{\substack{|\eta_{2n+1}|\geq |\eta_j|\\ \text{ for all } j=1, \dotsc, 2n}} \Tb_n(\etab_n) \hat\varphi(\eta_1) \hat\varphi(\eta_2) \dotsm \hat\varphi(\eta_{2n+1}) e^{i (\eta_1 + \eta_2 + \dotsb + \eta_{2n + 1}) x}\diff{\etab_n}\\
 &\qquad=c_n\px \int_{\R}\int\limits_{\substack{ |\eta_j|\leq |\eta_{2n+1}|\\ \text{ for all } j=1, \dotsc, 2n}} \Tb_n(\etab_n) \hat\varphi(\eta_1) \hat\varphi(\eta_2) \dotsm \hat\varphi(\eta_{2n}) e^{i (\eta_1 + \eta_2 + \dotsb + \eta_{2n }) x}\diff{\hat\etab_n}
 \hat\varphi(\eta_{2n+1})e^{ix\eta_{2n+1 } } \diff{\eta_{2n+1}}\\
&\qquad=c_n\px \int_{\R}\int\limits_{\substack{ |\eta_j|\leq |\eta_{2n+1}|\\ \text{ for all } j=1, \dotsc, 2n}} \Tb_n(\etab_n) \prod\limits_{j=1}^{2n}\left[\chi\left(\frac{(2n+1)\eta_j}{\eta_{2n+1}}\right)+1-\chi\left(\frac{(2n+1)\eta_j}{\eta_{2n+1}}\right)\right]
\hat\varphi(\eta_j)
\\&\hskip2.5in
\cdot e^{i (\eta_1 + \eta_2 + \dotsb + \eta_{2n }) x}\diff{\hat\etab_n}
\hat\varphi(\eta_{2n+1})e^{ix\eta_{2n+1 } } \diff{\eta_{2n+1}}.
\end{split}
\end{align}
Next, we expand the product in the above integral, and consider two cases depending on whether a term in the expansion
contains only factors of $\chi$ or contains at least one factor $1-\chi$.

{\bf Case I.} When we take only factors of $\chi$ in the expansion of the product, we get the integral
\begin{align}\label{expandfn1}
c_n\px \int_{\R}\int\limits_{\substack{ |\eta_j|\leq |\eta_{2n+1}|\\ \text{ for all } j=1, \dotsc, 2n}} \Tb_n(\etab_n) \prod\limits_{j=1}^{2n}\chi\left(\frac{(2n+1)\eta_j}{\eta_{2n+1}}\right)\hat\varphi(\eta_j) e^{i (\eta_1 + \eta_2 + \dotsb + \eta_{2n }) x}\diff{\hat\etab_n}\hat\varphi(\eta_{2n+1})e^{ix\eta_{2n+1 } } \diff{\eta_{2n+1}}.
\end{align}

From \eqref{defTn}, we can write $\Tb_n$ as an integral with respect to
$\s_n = (s_1, s_2, \dotsc, s_{2n + 1})$ by first adding a convergent factor $e^{-\epsilon \zeta^2}$ and then take limits $\epsilon \to 0+$,
\begin{align*}
\Tb_n(\etab_n) &= - \int_\R |\zeta|^{\alpha - 1} \sgn{\zeta} \int_{[0, 1]^{2n + 1}} \prod_{j = 1}^{2n + 1} i \eta_j e^{i \eta_j s_j \zeta} \diff{\s_n} \diff{\zeta}\\
&= - i A (1 - \alpha) \prod_{j = 1}^{2n + 1} (i \eta_j) \int_{[0, 1]^{2n + 1}} \bigg|\sum_{j = 1}^{2n + 1} \eta_j s_j\bigg|^{- \alpha} \sgn\bigg(\sum_{j = 1}^{2n + 1} \eta_j s_j\bigg) \diff{\s_n}\\
&= - A |\eta_{2n + 1}|^{1 - \alpha} \prod_{j = 1}^{2n} (i \eta_j) \int_{[0, 1]^{2n}} \bigg|1 + \sum_{j = 1}^{2n} \frac{\eta_j}{\eta_{2n + 1}} s_j\bigg|^{1 - \alpha} - \bigg|\sum_{j = 1}^{2n} \frac{\eta_j}{\eta_{2n + 1}} s_j\bigg|^{1 - \alpha} \diff{\hat{\s}_n}.
\end{align*}
Substitution of this expression into \eqref{expandfn1} gives the following three terms
\begin{align}
\label{T0}
&c_n\px \int_{\R}\int\limits_{\substack{ |\eta_j|\leq |\eta_{2n+1}|\\ \text{ for all } j=1, \dotsc, 2n}} \Tb_n^{0}(\etab_n) \prod\limits_{j=1}^{2n}\chi\left(\frac{(2n+1)\eta_j}{\eta_{2n+1}}\right)\hat\varphi(\eta_j) e^{i (\eta_1 + \eta_2 + \dotsb + \eta_{2n }) x}\diff{\hat\etab_n}\hat\varphi(\eta_{2n+1})e^{ix\eta_{2n+1 } } \diff{\eta_{2n+1}},
\\
\label{T2-alpha}
&c_n\px \int_{\R}\int\limits_{\substack{ |\eta_j|\leq |\eta_{2n+1}|\\ \text{ for all } j=1, \dotsc, 2n}} \Tb_n^{1-\alpha}(\etab_n) \prod\limits_{j=1}^{2n}\chi\left(\frac{(2n+1)\eta_j}{\eta_{2n+1}}\right)\hat\varphi(\eta_j) e^{i (\eta_1 + \eta_2 + \dotsb + \eta_{2n }) x}\diff{\hat\etab_n}\hat\varphi(\eta_{2n+1})e^{ix\eta_{2n+1 } } \diff{\eta_{2n+1}},
\\
\label{remainder1}
&c_n\px \int_{\R}\int\limits_{\substack{ |\eta_j|\leq |\eta_{2n+1}|\\ \text{ for all } j=1, \dotsc, 2n}} \Tb_n^{\leq -1}(\etab_n)(\etab_n) \prod\limits_{j=1}^{2n}\chi\left(\frac{(2n+1)\eta_j}{\eta_{2n+1}}\right)\hat\varphi(\eta_j) e^{i (\eta_1 + \eta_2 + \dotsb + \eta_{2n }) x}\diff{\hat\etab_n}\hat\varphi(\eta_{2n+1})e^{ix\eta_{2n+1 } } \diff{\eta_{2n+1}},
\end{align}
where
\begin{align*}
\Tb_n^{0}(\etab_n) &= A\prod_{j = 1}^{2n} (i \eta_j) \int_{[0, 1]^{2n}} \bigg|\sum_{j = 1}^{2n} \eta_j s_j\bigg|^{1 - \alpha} \diff{\hat{\s}_n},\\
\Tb_n^{1 - \alpha}(\etab_n) &= -A |\eta_{2n + 1}|^{1 - \alpha} \prod_{j = 1}^{2n} (i \eta_j),\\
\Tb_n^{\leq -1}(\etab_n) &= - A |\eta_{2n + 1}|^{1 - \alpha} \prod_{j = 1}^{2n} (i \eta_j) \cdot \int_{[0, 1]^{2n}} \bigg\{\bigg|1 + \sum_{j = 1}^{2n} \frac{\eta_j}{\eta_{2n + 1}} s_j\bigg|^{1 - \alpha} - 1\bigg\} \diff{\hat{\s}_n}.
\end{align*}

We claim that \eqref{T0}, \eqref{T2-alpha}, and \eqref{remainder1} can be written as
\begin{align}
\label{TBdecomp01}
  \px T_{B^0[\varphi]} \varphi + \Rc_{1, 2},\quad \px T_{B^{1 - \alpha}[\varphi]} |\px|^{1 - \alpha} \varphi + \Rc_{1, 1}, \quad \text{and} \quad  \Rc_{1, 3},
\end{align}
where $\Rc_{1, 1}$, $\Rc_{1, 2}$, and $\Rc_{1, 3}$ satisfy the estimate \eqref{R1est01}. Indeed,
\begin{align*}
\F\left[\px T_{B^{1 - \alpha}_n[\varphi]} \big(|\px|^{1 - \alpha} \varphi\big)\right](\xi) &= - \frac{c_n}{\pi} \Gamma(\alpha - 1) \sin\bigg(\frac{\pi \alpha}{2}\bigg) i \xi \int_\R \chi\bigg(\frac{|\xi - \eta|}{|\xi + \eta|}\bigg) |\eta|^{1 - \alpha}\\
& \qquad \int_{\R^{2n}} \delta\bigg(\xi - \eta - \sum_{j = 1}^{2n} \eta_j\bigg) \prod_{j = 1}^{2n} \bigg(i \eta_j \hat{\varphi}(\eta_j) \chi\Big(\frac{2 (2n + 1) \eta_j}{\xi + \eta}\Big)\bigg) \diff{\hat{\etab}_n} \hat{\varphi}(\eta) \diff{\eta},
\end{align*}
while the Fourier transform of \eqref{T2-alpha} is
\begin{align*}
& -\frac{c_n}{\pi} \Gamma(\alpha - 1) \sin\bigg(\frac{\pi \alpha}{2}\bigg) i \xi \int_\R \int\limits_{\substack{|\eta_j| \leq |\eta_{2n + 1}|\\ \text{for all}\ j = 1, 2, \dotsc, 2n}} \delta\bigg(\xi - \sum_{j = 1}^{2n + 1} \eta_j\bigg) |\eta_{2n + 1}|^{1 - \alpha}\\
& \hspace{1in} \cdot \prod_{j = 1}^{2n} \bigg(\chi\Big(\frac{(2n + 1) \eta_j}{\eta_{2n + 1}}\Big) (i \eta_j) \hat{\varphi}(\eta_j)\bigg) \diff{\hat{\etab}_n} \hat{\varphi}(\eta_{2n + 1}) \diff{\eta_{2n + 1}}.
\end{align*}
The difference of the above two integrals is
\begin{align}
\label{error01}
\begin{split}
- \frac{c_n}{\pi} & \Gamma(\alpha - 1) \sin\bigg(\frac{\pi \alpha}{2}\bigg) i \xi \int_{\R^{2n + 1}} \delta\bigg(\xi - \sum_{j = 1}^{2n + 1} \eta_j\bigg) |\eta_{2n + 1}|^{1 - \alpha} \\
& \cdot \Bigg[\chi\bigg(\frac{|\xi - \eta_{2n + 1}|}{|\xi + \eta_{2n + 1}|}\bigg) \prod_{j = 1}^{2n} \bigg(i \eta_j \hat{\varphi}(\eta_j) \chi\Big(\frac{2 (2n + 1) \eta_j}{\xi + \eta_{2n + 1}}\Big)\bigg)\\
& \quad - \mathbb{I}_n \prod_{j = 1}^{2n} \bigg(i \eta_j \hat{\varphi}(\eta_j) \chi\Big(\frac{(2n + 1) \eta_j}{\eta_{2n + 1}}\Big)\bigg)\Bigg] \diff{\hat{\etab}_n} \hat{\varphi}(\eta_{2n + 1}) \diff{\eta_{2n + 1}},
\end{split}
\end{align}
where $\mathbb{I}_n$ is the function which is equal to $1$ on $\{|\eta_j| \leq |\eta_{2n + 1}|,\ \text{for all}\ 1, \dotsc, 2n\}$ and equal to zero otherwise.

When $\etab_n$ satisfies
\begin{align}
\label{etancon01}
|\eta_j| \leq \frac{1}{40} \frac{1}{2 n + 1} |\eta_{2 n + 1}| \qquad \text{for all}\ j = 1, 2, \dotsc, 2n,
\end{align}
we have $\mathbb{I}_n = 1$ and $\chi \left(\frac{(2n + 1) \eta_j}{\eta_{2n + 1}}\right) = 1$. In addition, since $\xi = \sum_{j = 1}^{2n + 1} \eta_j$, we have
\begin{align*}
& \frac{|\xi - \eta_{2n + 1}|}{|\xi + \eta_{2n + 1}|} = \frac{\left|\sum_{j = 1}^{2n} \eta_j\right|}{\left|\sum_{j = 1}^{2n} \eta_j + 2 \eta_{2n + 1}\right|} \leq \frac{\frac{1}{40} |\eta_{2n + 1}|}{(2 - \frac{1}{40}) |\eta_{2n + 1}|} = \frac{1}{79} < \frac{3}{40},\\
& \frac{2 (2n + 1) |\eta_j|}{|\xi + \eta_{2n + 1}|} \leq \frac{\frac{1}{20} |\eta_{2n + 1}|}{(2 - \frac{1}{40}) |\eta_{2n + 1}|} = \frac{2}{79} < \frac{3}{40}.
\end{align*}

Therefore, the integrand of \eqref{error01} is supported outside of the set \eqref{etancon01}, and there exists $j_1 \in \{1, \dotsc, 2n\}$, such that $|\eta_{2n + 1}| \geq |\eta_{j_ 1}| > \frac{1}{40} \frac{1}{2n + 1} |\eta_{2n + 1}|$. It follows from this comparability of $|\eta_{j_1}|$ and $|\eta_{2n + 1}|$ that the $H^s$-norm of \eqref{error01} can be bounded by
\[
C(n, s) |c_n| \|\vp\|_{H^s} \|\vp_x\|_{W^{1, \infty}}^{2n}.
\]
It follows that \eqref{T2-alpha} can be written as in \eqref{TBdecomp01}. Similar calculations apply to \eqref{T0}.

The symbols $B^{1 - \alpha}_n[\varphi]$ and $B^0_n[\varphi]$ are real, so $T_{B^{1 - \alpha}_n[\varphi]}$ and $T_{B^0_n[\varphi]}$ are self-adjoint. Again, without loss of generality, we assume that $|\eta_{2n}| = \max_{1 \leq j \leq 2n} |\eta_j|$ and observe that
\begin{align*}
\int_{[0, 1]^{2n}} \bigg|\sum_{j = 1}^{2n} \eta_j s_j\bigg|^{1 - \alpha} \diff{\hat{\s}_n} = |\eta_{2n}|^{1 - \alpha} + O(1).
\end{align*}
Thus, using Young's inequality, we obtain the symbol estimates \eqref{B-est01}.

To estimate \eqref{remainder1}, we observe that on the support of the functions $\chi\left(\frac{(2n + 1) \eta_j}{\eta_{2n + 1}}\right)$, we have
\[
\frac{|\eta_j|}{|\eta_{2n + 1}|} \leq \frac{1}{10 (2n + 1)}.
\]
Since $s_j \in [0, 1]$, a Taylor expansion gives
\[
\left|\Tb_n^{\leq -1}(\etab_n)\right| \lesssim \frac{\left[\prod_{j = 1}^{2n} |\eta_j|\right] \left[\sum_{j = 1}^{2n} |\eta_j|\right]}{|\eta_{2n + 1}|^\alpha}.
\]
Therefore, the $H^s$-norm of \eqref{remainder1} is bounded by $C(n, s) |c_n| \|\vp\|_{H^s} \|\vp_x\|_{W^{1,\infty}}^{2n}$.

\textbf{Case II.} When there is at least one factor of the form $1 - \chi$ in the expansion of the product in the integral \eqref{intprod}, we get a term like
\begin{align}
\label{error'01}
\begin{split}
& c_n \px \int_\R \int\limits_{\substack{|\eta_j| \leq |\eta_{2n + 1}|\\ \text{for all}\ j = 1, 2, \dotsc, 2n}} \Tb_n(\etab_n) \prod_{k = 1}^\ell \chi\bigg(\frac{(2n + 1) \eta_{j_k}}{\eta_{2n + 1}}\bigg) \cdot \prod_{k = \ell + 1}^{2n} \bigg[1 - \chi\bigg(\frac{(2n + 1) \eta_{j_k}}{\eta_{2n + 1}}\bigg)\bigg]\\
& \hspace{2.5in} \cdot \prod_{j = 1}^{2n} \hat{\varphi}(\eta_j) e^{i (\eta_1 + \eta_2 + \dotsb + \eta_{2n}) x} \diff{\hat{\etab}_n} \hat{\varphi}(\eta_{2n + 1}) e^{i \eta_{2n + 1} x} \diff{\eta_{2n + 1}},
\end{split}
\end{align}
where $0 \leq \ell \leq 2n - 1$ is an integer, and $\{j_k : k = 1, \dotsc, 2n\}$ is a permutation of $\{1, \dotsc, 2n\}$.

$1 - \chi\left(\frac{(2n + 1) \eta_{j_{2n}}}{\eta_{2n + 1}}\right)$ is compactly supported on
\[
\frac{|\eta_{j_{2n}}|}{|\eta_{2n + 1}|}\geq \frac{3}{40(2n + 1)}.
\]
By assumption, $\eta_{2n + 1}$ has the largest absolute value, so
\[
\frac{3}{40(2n + 1)} |\eta_{2n + 1}| \leq |\eta_{j_{2n}}|\leq |\eta_{2n + 1}|,
\]
meaning that the frequencies $|\eta_{j_{2n}}|$ and $|\eta_{2n + 1}|$ are comparable.

Without loss of generality, we assume that $|\eta_{j_1}| \leq |\eta_{j_2}| \leq \dotsb \leq |\eta_{j_{2n}}| \leq |\eta_{2n + 1}|$, define $\eta_{j_{2n + 1}} = \eta_{2n + 1}$, and split the integral of $\Tb_n$ \eqref{defTn} into three parts.
\begin{align}
\Tb_n(\etab_n) & = \Tb_n^{low}(\etab_n) + \sum_{k = 1}^{2n - 1} \Tb_n^{med, (k)}(\etab_n) + \Tb_n^{high}(\etab_n),\notag\\
\Tb_n^{low}(\etab_n) & = \int_{|\eta_{j_{2n}} \zeta| < 2} \frac{\prod_{j = 1}^{2n + 1} \left(1 - e^{i \eta_j \zeta}\right)}{\zeta^{2 n + 1}} |\zeta|^{\alpha - 1} \sgn{\zeta} \diff{\zeta}, \label{Tn-low}\\
\Tb_n^{med, (k)}(\etab_n) & = \int_{\frac{2}{\left|\eta_{j_{k + 1}}\right|} \le |\zeta| \le \frac{2}{\left|\eta_{j_k}\right|}} \frac{\prod_{j = 1}^{2n + 1} \left(1 - e^{i \eta_j \zeta}\right)}{\zeta^{2 n + 1}} |\zeta|^{\alpha - 1} \sgn{\zeta} \diff{\zeta}, \label{Tn-med}\\
\Tb_n^{high}(\etab_n) & = \int_{|\eta_{j_1} \zeta| > 2} \frac{\prod_{j = 1}^{2n + 1} \left(1 - e^{i \eta_j \zeta}\right)}{\zeta^{2 n + 1}} |\zeta|^{\alpha - 1} \sgn{\zeta} \diff{\zeta}. \label{Tn-high}
\end{align}

To estimate \eqref{Tn-low}, we observe that
\begin{align*}
|\Tb_n^{low}(\etab_n)| & \le \prod_{k = 1}^{2n} |\eta_{j_k}| \cdot \int_{|\eta_{j_{2n}} \zeta| < 2} \bigg(\prod_{k = 1}^{2n} \frac{|1 - e^{i \eta_{j_k} \zeta}|}{|\eta_{j_k} \zeta|}\bigg) |1 - e^{i \eta_{2n + 1} \zeta}| \cdot |\zeta|^{\alpha - 2} \diff{\zeta}\\
& \leq \prod_{k = 1}^{2n} |\eta_{j_k}| \cdot \int_{|\eta_{j_{2n}} \zeta| < 2} \frac{2} {|\zeta|^{2 - \alpha}} \diff{\zeta}\\
& \leq \frac{2^{\alpha + 1}}{\alpha - 1} \bigg(\prod_{k = 1}^{2n} |\eta_{j_k}|\bigg) |\eta_{j_{2n}}|^{1 - \alpha}.
\end{align*}

For each $1 \le k \le 2n - 1$, we estimate \eqref{Tn-med} as
\begin{align*}
|\Tb_n^{med, (k)}(\etab_n)| & \leq \prod_{\ell = 1}^{k} |\eta_{j_\ell}| \cdot \int_{\frac{2}{\left|\eta_{j_{k + 1}}\right|} \le |\zeta| \le \frac{2}{\left|\eta_{j_k}\right|}} \bigg(\prod_{\ell = 1}^k \frac{|1 - e^{i \eta_{j_\ell} \zeta}|}{|\eta_{j_\ell} \zeta|}\bigg) \cdot \frac{\prod_{\ell = k + 1}^{2n + 1} |1 - e^{i \eta_{j_\ell} \zeta}|}{|\zeta|^{2n + 1 - k}} |\zeta|^{\alpha - 1} \diff{\zeta}\\
& \leq 2^{2n + 1 - k} \prod_{\ell = 1}^{k} |\eta_{j_\ell}| \cdot \int_{\frac{2}{\left|\eta_{j_{k + 1}}\right|} \le |\zeta| \le \frac{2}{\left|\eta_{j_k}\right|}} |\zeta|^{-2n - 2 + k + \alpha} \diff{\zeta}\\
& \leq \frac{2^\alpha}{2n +1 - k - \alpha} \left(|\eta_{j_k}|^{2n + 1 - k - \alpha} + |\eta_{j_{k + 1}}|^{2n + 1 - k - \alpha}\right) \prod_{\ell = 1}^{k} |\eta_{j_\ell}|\\
& \leq \frac{2^\alpha}{2 - \alpha} \bigg(\prod_{k = 1}^{2n} |\eta_{j_k}|\bigg) |\eta_{j_{2n}}|^{1 - \alpha}.
\end{align*}

As for \eqref{Tn-high}, we have
\begin{align*}
|\Tb_n^{high}(\etab_n)| & \leq |\eta_{j_1}| \int_{|\eta_{j_1} \zeta| > 2} \bigg(\prod_{k = 2}^{2n + 1} \frac{|1 - e^{i \eta_{j_k} \zeta}|}{|\zeta|} \bigg) \cdot \frac{|1 - e^{i \eta_{j_1} \zeta}|}{|\eta_{j_1} \zeta|} |\zeta|^{\alpha - 1} \diff{\zeta}\\
& \le 2^{2n} |\eta_{j_1}| \int_{|\eta_{j_1} \zeta| > 2} \frac{\diff{\zeta}}{|\zeta|^{2n + 1 - \alpha}}\\
& \le \frac{2^{\alpha + 1}}{2n - \alpha} \bigg(\prod_{k = 1}^{2n} |\eta_{j_k}|\bigg) |\eta_{j_{2n}}|^{1 - \alpha}.
\end{align*}
Collecting these estimates, we get that
\[
\Tb_n(\etab_n) \leq C(n, s) \biggl(\prod_{k = 1}^{2n} |\eta_{j_k}|\biggr) |\eta_{j_{2n}}|^{1 - \alpha},
\]
and we can use this inequality to bound the $H^s$-norm of the expression in \eqref{error'01} by
\[
\|\vp\|_{H^s} \sum_{n = 1}^\infty  C(n, s) |c_n| \|\varphi_x\|_{W^{1, \infty}}^{2n}.
\]

The proposition then follows by combining the previous estimates.
\end{proof}

Using the results of Proposition~\ref{varphinonlindecomp01} in \eqref{GSQGfront}, we find that the front equation \eqref{GSQGfront0} with $\Theta = -1$ can be written as
\begin{equation}\label{Para-eq}
\vp_t-A\partial_x|\partial_x|^{1-\alpha}\vp+\partial_x T_{B^{0}[\vp]}\vp + \partial_x|\partial_x|^{1-\alpha} T_{B^{1-\alpha}[\vp]}\vp+\Rc=0,
\end{equation}
where the remainder term $\Rc$ satisfies \eqref{R1est01}.

\section{Improved energy estimates}
\label{sec:local}

When $1< \alpha \le 2$, the GSQG front equation has standard short-time energy estimates for unweighted $H^s$-norms without any smallness assumption on $\vp_x$ \cite{HS18,HSZ19pa}. (By contrast, when $0<\alpha \le 1$, one has to use appropriately weighted $H^s$-norms \cite{CGI19, HSZ18, HSZ18p}.) In this section, at the expense of a smallness assumption on $\vp_x$, we use the para-linearized equation to derive more precise \textit{a priori} energy estimates that involve the $W^{1, \infty}$-norm of $\vp_x$. These estimates will allow us to analyze the decay of solutions in the global existence proof.

We first observe that the $L^2$-norm of the solution is conserved, so we only need to estimate the higher homogeneous Sobolev norms. Taking $k\in \N$ derivatives of equation \eqref{Para-eq}, we get that
\begin{equation}\label{ktheq}
\px^k\vp_t-A\partial_x^{k+1}|\partial_x|^{1-\alpha}\vp+ \partial_x^{k+1}|\partial_x|^{1-\alpha} T_{B^{1-\alpha}[\vp]}\vp+\partial_x ^{k+1}T_{B^{0}[\vp]}\vp+\px^k\Rc=0,
\end{equation}
where
\begin{align*}
\partial_x^{k+1}|\partial_x|^{1-\alpha} T_{B^{1-\alpha}[\vp]}\vp &= \px|\partial_x|^{1-\alpha} T_{B^{1-\alpha}[\vp]}\partial_x^{k}\vp-\px|\partial_x|^{1-\alpha}\left[T_{B^{1-\alpha}[\vp]},\partial_x^{k}\right]\vp,\\
\partial_x ^{k+1}T_{B^{0}[\vp]}\vp &= \px\left[\px^{k}, T_{B^{0}[\vp]}\right] \vp+\px T_{B^{0}[\vp]}\px^{k}\vp.
\end{align*}
We multiply \eqref{ktheq} by $\px^k\vp$ and integrate over $\R$ to get
\begin{align}
\label{energy_eq}
\begin{split}
&\frac12\frac{\diff}{\diff t} \|\px^k\vp\|_{L^2}^2 -\int_{\R} A\partial_x^{k+1}|\partial_x|^{1-\alpha}\vp\cdot \px^k\vp\diff x+\int_{\R} \px|\partial_x|^{1-\alpha} T_{B^{1-\alpha}[\vp]}\partial_x^{k}\vp \cdot \px^k\vp
\\
&\qquad -\px|\partial_x|^{1-\alpha} \left[T_{B^{1-\alpha}[\vp]},\partial_x^{k}\right]\vp\cdot \px^k\vp \diff x
+\int_{\R} \px\left[\px^{k}, T_{B^{0}[\vp]}\right] \vp  \cdot\px^k\vp
\\
&\qquad+\px T_{B^{0}[\vp]}\px^{k}\vp\cdot \px^k\vp  +\px^k\Rc\cdot \px^k\vp  \diff x=0.
\end{split}
\end{align}
Since $\partial_x|\partial_x|^{1-\alpha}$ is skew-symmetric and $\px|\partial_x|^{1-\alpha} T_{B^{1-\alpha}[\vp]}$, $\px T_{B^{0}[\vp]}$
are skew-symmetric up to commutators, we can write
\begin{align*}
&-\int_{\R} A\partial_x^{k+1}|\partial_x|^{1-\alpha}\vp \cdot\px^k\vp\diff x+\int_{\R} \px|\partial_x|^{1-\alpha} T_{B^{1-\alpha}[\vp]}\partial_x^{k}\vp\cdot  \px^k\vp +\px T_{B^{0}[\vp]}\px^{k}\vp \cdot\px^k\vp \diff{x}
\\
&\qquad\qquad = \frac{1}{2} \int_{\R}  \left[\px|\partial_x|^{1-\alpha}, T_{B^{1-\alpha}[\vp]}\right]\partial_x^{k}\vp\cdot  \px^k\vp
+\left[\px, T_{B^{0}[\vp]}\right]\px^{k}\vp \cdot\px^k\vp \diff{x}.
\end{align*}

Using the commutator estimates
\begin{align*}
\left\|\px|\partial_x|^{1-\alpha} \left[T_{B^{1-\alpha}[\vp]},\partial_x^{k}\right]\vp\right\|_{L^2}
+ \left\|\left[\px|\partial_x|^{1-\alpha}, T_{B^{1-\alpha}[\vp]}\right]\partial_x^{k}\vp\right\|_{L^2}
&\lesssim \|\vp\|_{H^{k+1-\alpha}} \bigg(\sum_{n = 1}^\infty  C(n, s) |c_n| \|\varphi_x\|_{W^{1, \infty}}^{2n}\bigg),\\
\left\|\px\left[\px^{k}, T_{B^{0}[\vp]}\right] \vp\right\|_{L^2}
+ \left\|\left[\px, T_{B^{0}[\vp]}\right]\px^{k}\vp\right\|_{L^2}
&\lesssim \|\vp\|_{H^{k}} \bigg(\sum_{n = 1}^\infty  C(n, s) |c_n| \|\varphi_x\|_{W^{1, \infty}}^{2n}\bigg),
\end{align*}
and the estimate \eqref{R1est01} of the remainder term $\Rc$ in \eqref{energy_eq}, we get that
\[
\frac{\diff}{\diff t} \|\partial_x^k\vp\|_{L^2}^2\lesssim \|\vp\|_{H^{k}}^2 \bigg(\sum_{n = 1}^\infty  C(n, s) |c_n| \|\varphi_x\|_{W^{1, \infty}}^{2n}\bigg).
\]
Summing this estimate over $1\le k \le s$, and including the $L^2$-norm, we obtain that
\[
\frac{\diff}{\diff t} \|\vp\|_{H^s}^2\lesssim \|\vp\|_{H^s}^2 \bigg(\sum_{n = 1}^\infty  C(n, s) |c_n| \|\varphi_x\|_{W^{1, \infty}}^{2n}\bigg),
\]
and Gr\"{o}nwall's inequality implies that
\begin{align}
\label{exp-ineq}
\|\vp\|_{H^s}^2 \lesssim \|\vp_0\|^2 \exp\left[\int_0^t \|\vp_x(\tau)\|_{W^{1, \infty}}^2 F(\|\vp_x(\tau)\|_{W^{1, \infty}}) \diff{\tau}\right],
\end{align}
where $F \colon \R \to \R_+$ is a continuous, increasing function.

Using this energy estimate, we can obtain local solutions $\vp\in C([0,T]; H^s)$ for $s \ge 3$ by standard methods.

\section{Global solution for small initial data}
\label{sec:global}

From now on, we fix the following parameter values
\begin{equation}
s = 1200, \qquad r = 8, \qquad p_0 = 10^{-4}.
\label{defsrp}
\end{equation}
We denote by
\begin{align}
\label{def-S}
\S = (2 - \alpha) t \partial_t + x \partial_x
\end{align}
the scaling vector field that commutes with the linearization of \eqref{Para-eq},
$\vp_t = A\partial_x|\partial_x|^{1-\alpha}\vp$.
We also introduce the notation
\begin{align}
\label{def-h}
h(x, t) = e^{- A t \px |\px|^{1 - \alpha}} \vp(x, t), \qquad \hat{h}(\xi, t) = e^{- i A t \xi |\xi|^{1 - \alpha}} \hat{\vp}(\xi, t)
\end{align}
for the function $h$ obtained by removing the action of the linearized evolution group on $\vp$.

\begin{theorem}[Main theorem]\label{main}
Let $s$, $r$, $p$ be defined as in \eqref{defsrp}. There exists $\epsilon > 0$ such that if $0 < \ve_0 \le \ve$
and $\vp_0\in H^s(\R)$ satisfies
\[
\|\vp_0\|_{H^s} + \|x \px \vp_0\|_{H^r} \le \ve_0,
\]
then there exists a unique global solution $\vp\in C([0,\infty); H^s(\R))$ of \eqref{GSQGfront0} with initial condition
$\left.\vp\right|_{t=0} = \vp_0$. Moreover, this solution satisfies
\[
\|\vp(t)\|_{H^s} + \|\S \vp(t)\|_{H^r} \lesssim \ve_0 (t + 1)^{p_0},
\]
where $\S$ is the vector field in \eqref{def-S}.
\end{theorem}

This theorem is a consequence of local existence and the following bootstrap result involving the $Z$-norm of the solution,
which we define for a function $f$ by
\begin{align}
\label{def-Z}
\|f\|_Z = \left\|(|\xi| + |\xi|^{r + 3}) \hat{f}(\xi)\right\|_{L_\xi^\infty}.
\end{align}

\begin{proposition}[Bootstrap]
\label{bootstrap}
Let $T > 1$ and suppose that $\vp \in C([0, T]; H^s(\R))$ is a solution of \eqref{GSQGfront0}, where the initial data satisfies
\[
\|\vp_0\|_{H^s} + \|x \px \vp_0\|_{H^r} \le \ve_0
\]
for some $0 < \ve_0 \ll 1$. If there exists $\ve_0 \ll \ve_1 \lesssim \ve_0^{1 / 3}$ such that the solution satisfies
\[
(t + 1)^{- p_0} \left(\|\vp(t)\|_{H^s} + \|\S \vp(t)\|_{H^r}\right) + \|\vp(t)\|_Z \le \ve_1
\]
for every $t \in [0, T]$, then the solution satisfies an improved bound
\[
(t + 1)^{- p_0} \left(\|\vp(t)\|_{H^s} + \|\S \vp(t)\|_{H^r}\right) + \|\vp(t)\|_Z \le \ve_0.
\]
\end{proposition}

We call the assumptions in Proposition \ref{bootstrap} the \emph{bootstrap assumptions}. To prove Proposition \ref{bootstrap}, we need the following lemmas, most of whose proofs are deferred to the next sections.

\begin{lemma}[Sharp pointwise decay]\label{sharp}
Under the bootstrap assumptions,
\[
\|\vp_x(t)\|_{W^{r, \infty}}\lesssim \ve_1(t+1)^{-1/2} \qquad \text{for $0\le t \le T$}.
\]
\end{lemma}

\begin{lemma}[Scaling vector field estimate]
\label{weightedE}
Under the bootstrap assumptions,
\[
(t + 1)^{- p_0} \|\S \vp(t)\|_{H^r} \lesssim \ve_0\qquad \text{for $0\le t \le T$}.
\]
\end{lemma}

\begin{lemma}
Under the bootstrap assumptions,
\[
(t + 1)^{- p_0} \left(\|\vp(t)\|_{H^s} + \|x \px h(t)\|_{H^r} \right) \lesssim \ve_0\qquad \text{for $0\le t \le T$}.
\]
\end{lemma}

\begin{proof}
Recalling the energy estimate \eqref{exp-ineq},
we have from Lemma \ref{sharp} that $F(\|\vp_x\|_{W^{1, \infty}}) \lesssim 1$,
which implies that
\[
\|\vp(t)\|_{H^s}^2 \lesssim \ve_0^2 (t + 1)^{C \ve_1^2},
\]
so once $\ve_1^2 \ll p_0$, we have
\[
(t + 1)^{- p_0} \|\vp(t)\|_{H^s} \lesssim \ve_0.
\]

It follows from \eqref{def-h}, the definition of $\S$ in \eqref{def-S}, and \eqref{GSQGfront} that
\begin{align}
\label{xipartialh}
\begin{split}
\F[x \px h](\xi, t) & = - \hat{h}(\xi, t) - \xi \partial_\xi \hat{h}(\xi, t),\\
\xi \partial_\xi \hat{h}(\xi, t) & = \xi e^{- i A t \xi |\xi|^{1 - \alpha}} \left(- i (2 - \alpha) A t |\xi|^{1 - \alpha} \hat{\vp}(\xi, t) + \partial_\xi \hat{\vp}(\xi, t)\right)\\
& = e^{- i A t \xi |\xi|^{1 - \alpha}} \left(- (2 - \alpha) t \hat{\vp}_t(\xi, t) - (2 - \alpha) t \hat{\Nc}(xi, t) + \xi \partial_\xi \hat{\vp}(\xi, t)\right)\\
& = e^{- i A t \xi |\xi|^{1 - \alpha}} \left(- \widehat{\S \vp}(\xi, t) - \hat{\vp}(\xi, t) - (2 - \alpha) t \hat{\Nc}(\xi, t)\right),
\end{split}
\end{align}
where
\begin{align*}
\Nc(x, t) & = \sum_{n = 1}^\infty \frac{c_n}{2n + 1} \px \int_{\R^{2n + 1}} \Tb_n(\etab_n) \hat{\vp}(\eta_1, t) \hat{\vp}(\eta_2, t) \dotsm \hat{\vp}(\eta_{2n + 1}, t) e^{i (\eta_1 + \eta_2 + \dotsb + \eta_{2n + 1}) x} \diff{\etab_n}
\end{align*}
denotes the nonlinear term of \eqref{GSQGfront}. It follows from \eqref{xipartialh} that
\begin{align*}
\F_x[x \px h](\xi, t) & = - e^{- i A t \xi |\xi|^{1 - \alpha}} \left(\widehat{\S \vp}(\xi, t) + (2 - \alpha) t \Nc(\hat{\vp}(\xi, t))\right).
\end{align*}
We observe that in view of Proposition \ref{varphinonlindecomp01}, the nonlinear term $\Nc$ satisfies
\begin{align}
\label{estimate_N}
\left\||\px|^j \Nc\right\|_{L^2} \lesssim \sum_{n = 1}^\infty C(n,s) \|\vp_x\|_{W^{2, \infty}}^{2n} \|\vp\|_{H^{j + 2}} \qquad \text{for  $j = 0, \dots, r$}.
\end{align}
By the bootstrap assumptions and Lemmas \ref{sharp}--\ref{weightedE}, we find that
$(t + 1)^{- p_0} \|x \px h(t)\|_{H^r} \lesssim \ve_0$.
\end{proof}

\begin{lemma}[Nonlinear dispersive estimate]
\label{nonlindisp}
Under the bootstrap assumptions, the solutions of \eqref{GSQGfront} satisfies
\[
\|\vp(t)\|_Z \lesssim \ve_0\qquad \text{for $0\le t \le T$}.
\]
\end{lemma}

Proposition \ref{bootstrap} then follows by combining Lemmas \ref{sharp}--\ref{nonlindisp}.

\section{Sharp dispersive estimate}
\label{sec:sharp}

In this section, we prove Lemma \ref{sharp}. We first state a dispersive estimate for the linearized evolution operator $e^{At\partial_x|\partial_x|^{1-\alpha}}$. This estimate is similar to ones in \cite{CGI19, HSZ18p} and we omit the proof.
We recall that $P_k$ is the frequency-localization operator with symbol $\psi_k$ defined in \eqref{defpsik}.

\begin{lemma}\label{disp}
For $t> 0$ and $f \in L^2$, we have the linear dispersive estimate
\begin{align}
\label{LocDis}
\begin{split}
\|e^{At\partial_x|\partial_x|^{1-\alpha}}P_k f\|_{L^\infty} &\lesssim (t+1)^{-1/2} 2^{\alpha k / 2}\|\widehat{P_k f}\|_{L^\infty_\xi}
\\
&+ (t + 1)^{-3 / 4} 2^{(3 \alpha / 4 - 1) k} \left[\|P_k (x \px f)\|_{L^2} + \|\tilde{P}_k f\|_{L^2}\right].
\end{split}
\end{align}
\end{lemma}

Using this lemma, we have the following.

\begin{proof} [Proof of Lemma \ref{sharp}]
Take the function $f$ in Lemma \ref{disp} to be $\px^j h$ for $j=1, 2, \dotsc, r+1$. Since $e^{A t \px |\px|^{1 - \alpha}}$ and $P_k$ commute, and
\[
x\px^{j+1}h=\px^j(x\px h)-j\px^j h,
\]
we have that
\begin{align*}
\|P_k \partial_x^j \vp\|_{L^\infty}&\lesssim (t+1)^{-1/2}\|\F(P_k |\partial_x|^{\frac{\alpha}{2}+j}\vp)\|_{L^\infty_\xi}
\\
&+(t+1)^{-3/4} \left[2^{\frac{3}{4} \alpha k}\||\px|^{j-1}P_k(x\px h) \|_{L^2}
+\|\tilde P_k (|\px|^{j - 1 + \frac{3 \alpha}{4}}\vp)\|_{L^2} \right].
\end{align*}

It follows from \eqref{xipartialh} that
\[
\||\px|^{j-1}P_k(x\px h) \|_{L^2}\lesssim \| |\px|^{j-1} P_k \vp\|_{L^2}+\| |\px|^{j-1} P_k \S\vp\|_{L^2}+t\|P_k|\px|^{j-1}\mathcal{N}\|_{L^2}.
\]
We first note that if $k\in \Z_-$, then $(t+1)^{-1/4+p_0}2^{\frac34 \alpha k}\lesssim 1$, so
\begin{align}\label{sharp1-}
\begin{split}
\|P_k\px^j\vp\|_{L^\infty} & \lesssim (t+1)^{-1/2} 2^{\alpha k/2}\|\psi_k(\xi)|\xi|^{j}\hat \vp(\xi)\|_{L^\infty_\xi} \\
&\qquad +(t+1)^{-1/2-p_0}[\||\px|^{j-1}\tilde P_k\vp\|_{L^2}+\||\px|^{j-1}P_k\mathcal S\vp\|_{L^2}+t\||\px|^{j-1}P_k\mathcal N\|_{L^2}].
\end{split}
\end{align}
If $k\in \Z_+$ and $(t+1)^{-1/4+p_0}2^{\frac34 \alpha k}\lesssim 1$, then
\begin{align}
\begin{split}
\|P_k\px^j\vp\|_{L^\infty} & \lesssim (t+1)^{-1/2} 2^{(\alpha / 2 - 2) k}\|\psi_k(\xi)|\xi|^{j+2}\hat \vp(\xi)\|_{L^\infty_\xi}\\
& \qquad +(t+1)^{-1/2-p_0}[\||\px|^{j-1}\tilde P_k\vp\|_{L^2}+\||\px|^{j-1}P_k\mathcal S\vp\|_{L^2}+t\||\px|^{j-1}P_k\mathcal N\|_{L^2}].
\end{split}
\end{align}
Finally, if $k\in \Z_+$ and $(t+1)^{-1/4+p_0}2^{\frac34 \alpha k}\gtrsim 1$, then $2^{-\alpha k}\lesssim (t+1)^{-\frac13+\frac43p_0}$, and
\begin{align}
\label{sharp1+}
\begin{split}
\|P_k\px^j\vp\|_{L^\infty}&\lesssim \||\xi|^j\psi_k(\xi)\hat\vp(\xi)\|_{L^1_\xi}\lesssim \||\xi|^{j-s}\psi_k(\xi)\|_{L^2}\|\tilde P_k\vp\|_{H^s}
\\
&\lesssim 2^{(j-s-\frac12)k}\|\tilde P_k\vp\|_{H^s}\lesssim (t+1)^{-1 - p_0}\|\tilde P_k\vp\|_{H^s}.
\end{split}
\end{align}

Using the bootstrap assumptions, the estimate \eqref{estimate_N}, and \eqref{sharp1-}--\eqref{sharp1+} in the corresponding ranges of $k$, then summing the results over $k \in \Z$ and $0\le j \le r+1$, we obtain that $\|\vp_x(t)\|_{W^{r + 1, \infty}} \lesssim \ve_1 (t + 1)^{- 1 / 2}$.
\end{proof}

\section{Scaling vector field estimate}
In this section, we prove the scaling vector field estimate in Lemma \ref{weightedE}. A direct calculation gives the following commutators.
\begin{lemma}
\label{S-comm}
Let $\varphi(x,t)$ be a Schwartz distribution on $\R^2$ such that $|\px|^{1 - \alpha} \varphi(x,t)$ is a Schwartz distribution
and $\S$ the vector field \eqref{def-S}. Then
\begin{align*}
&[\S,\partial_x]\vp = -\partial_x \vp,\qquad
[\S,|\partial_x|^{1-\alpha}\partial_x]\vp = - (2 - \alpha) |\partial_x|^{1-\alpha} \px \vp,
\\
&[\S,\partial_t]\vp = -(2-\alpha)\partial_t\vp,\qquad
[\S, \partial_t - A|\partial_x|^{1 - \alpha} \px]\vp=-(2-\alpha)(\partial_t-A|\partial_x|^{1 - \alpha} \px) \vp.
\end{align*}
\end{lemma}

\begin{proof}[Proof of Lemma \ref{weightedE}]
Applying $\S$ to equation \eqref{Para-eq} and using Lemma \ref{S-comm}, we get that
\[
(\S \vp)_t-A|\partial_x|^{1-\alpha}\partial_x(\S\vp)+\partial_x  T_{B^0[\vp]} \S\vp + |\partial_x|^{1-\alpha}[ T_{B^{1-\alpha}[\vp]}\S\vp]_x+\S\Rc=\text{commutators},
\]
where the commutators include
\[
\partial_x [\S, T_{B^0[\vp]}]\vp,\quad [\S, \partial_x] T_{B^0[\vp]}\vp,\quad [\S, |\partial_x|^{1-\alpha}\partial_x]\big(T_{B^{1-\alpha}[\vp]}\vp\big),\quad |\partial_x|^{1-\alpha}\partial_x\Big([\S, T_{B^{1-\alpha}[\vp]}]\vp\Big).
\]
By Kato-Ponce type commutator estimates and \eqref{B-est01}, when $k\leq r$ we obtain that
\begin{align*}
\|\partial_x [\S, T_{B^0[\vp]}]\vp\|_{H^k}&\lesssim F(\|\vp_x\|_{W^{2,\infty}}) (\|\vp_x\|_{W^{2,\infty}}) \|\vp_x\|_{W^{r, \infty}}\|\S\vp\|_{H^{r}}\\
&\lesssim F(\|\vp_x\|_{W^{2,\infty}}) (\|\vp_x\|_{W^{r,\infty}})^2\|\S\vp\|_{H^{r}},
\end{align*}
\begin{align*}
\|[\S, \partial_x] T_{B^0[\vp]}\vp\|_{H^k}=\|- \px T_{B^0[\vp]}\vp\|_{H^k}\lesssim F(\|\vp_x\|_{W^{2,\infty}}) (\|\vp_x\|_{W^{r,\infty}})^2\|\vp_x\|_{H^{k + 1}},
\end{align*}
\begin{align*}
\big\|[\S, |\partial_x|^{1-\alpha}\partial_x]\big(T_{B^{1-\alpha}[\vp]}\vp\big)\big\|_{H^k}&=\big\|\big(T_{B^{1-\alpha}[\vp]}\vp\big)_x+|\partial_x|^{1-\alpha}\partial_x\big(T_{B^{1-\alpha}[\vp]}\vp\big)\big\|_{H^k}\\
&\lesssim F(\|\vp_x\|_{W^{2,\infty}}) (\|\vp_x\|_{W^{r,\infty}})^2\|\vp\|_{H^{k+1}},
\end{align*}
\begin{align*}
\bigg\| |\partial_x|^{1-\alpha}\partial_x\Big([\S, T_{B^{1-\alpha}[\vp]}]\vp\Big)\bigg\|_{H^k}\lesssim F(\||\partial_x|^{1-\alpha}\vp\|_{W^{2,\infty}}+\|\vp\|_{W^{2,\infty}}) (\|\vp_x\|_{W^{1,\infty}})^2\|\S\vp\|_{H^{r}}.
\end{align*}

Thus, straightforward energy estimates give for $j = 0, \dotsc, r$ that
\[
\frac {\diff}{\diff{t}} \|\S \vp(t)\|_{\dot{H}^j}^2 \lesssim (\|\vp_x\|_{W^{r , \infty}} )^2 F(\|\S \vp\|_{H^r} + \|\vp\|_{H^s}) \|\S\vp\|_{H^j}^2.
\]
Then using Lemma \ref{sharp} and integrating in time $t$, we find that
\[
\|\S \vp(t)\|_{H^r}^2 \lesssim \ve_0^2 (t + 1)^{2 p_0},
\]
which proves the lemma.
\end{proof}

\section{Nonlinear dispersive estimate}
\label{sec:nonlindisp}

In this section, we prove the dispersive estimate in Lemma \ref{nonlindisp}, where the $Z$-norm is defined in \eqref{def-Z}.

When $|\xi|<(t + 1)^{- p_0}$, Lemma \ref{interpolation} and the conservation of the $L^2$-norm of $\vp$ gives
\begin{align*}
|(|\xi| + |\xi|^{r + 3})\hat\vp(\xi,t)|^2&\lesssim (|\xi| + |\xi|^{r + 3})^2|\xi|^{-1}\|\hat\vp\|_{L^2_\xi}(|\xi|\|\partial_\xi\hat\vp\|_{L^2_\xi}+\|\hat\vp\|_{L^2_\xi})\\
&\lesssim (|\xi| + |\xi|^{r + 3}) \|\vp\|_{L^2}(\|\S\vp\|_{L^2}+\|\vp\|_{L^2})\\
&\lesssim \ve_0^2.
\end{align*}
Let $p_1 = 10^{-6}$. When $|\xi| > s (t + 1)^{p_1}$, Lemma \ref{interpolation} gives
\begin{align*}
|(|\xi| + |\xi|^{r + 3}) \hat\vp(\xi,t)|^2 &\lesssim \frac{(|\xi| + |\xi|^{r + 3})^2}{|\xi|^{s + 1}} \|\vp\|_{H^s} (\|\S\vp\|_{L^2} + \|\vp\|_{L^2})\\
&\lesssim |\xi|^{2r + 5 - s} \ve_0^2 (t + 1)^{2p_0} \\
&\lesssim \ve_0^2 .
\end{align*}
Thus, we only need to consider the frequency range
\begin{align}
(t + 1)^{- p_0}\leq |\xi|\leq (t + 1)^{p_1}.
\label{freq_range}
\end{align}
In the following, we fix $\xi$ in this range and denote by $\cutoffxi(\xi, t)$ a smooth cut-off function that is compactly supported in a small neighborhood of
\begin{equation}
\left\{(\xi, t)\in \R^2: \text{$(t+1)^{-p_0}\le|\xi|\le(t+1)^{p_1}$, $0 \le t \le T$}\right\}.
\label{cutoff}
\end{equation}

Taking the Fourier transform of \eqref{GSQGfront}, we obtain that
\begin{align}
\begin{split}
&\hat \vp_t(\xi,t) + i A' \xi \iint_{\R^2} \Tb_1'(\eta_1,\eta_2, \xi - \eta_1 - \eta_2) \hat\vp(\xi-\eta_1-\eta_2,t)\hat\vp(\eta_1,t)\hat\vp(\eta_2,t) \diff{\eta_1} \diff{\eta_2}
\\
&\qquad\qquad\qquad\qquad +\widehat{\Nc_{\geq5}(\vp)}(\xi,t)
 = i A \xi |\xi|^{1 - \alpha} \hat{\vp}(\xi,t),
\end{split}
\label{phihateq}
\end{align}
where $A' = - {A}/[6 (3 - \alpha)]$ and
\begin{align*}
\Tb_1'(\eta_1,\eta_2, \eta_3) &= \frac{(2 - \alpha) (3 - \alpha)}{A} \Tb_1(\eta_1, \eta_2, \eta_3)\\
& = |\eta_1|^{3 - \alpha} + |\eta_2|^{3 - \alpha} + |\eta_3|^{3 - \alpha} + |\eta_1 + \eta_2 + \eta_3|^{3 - \alpha} - |\eta_1 + \eta_2|^{3 - \alpha} - |\eta_1 + \eta_3|^{3 - \alpha} - |\eta_2 + \eta_3|^{3 - \alpha},\\
\Nc_{\geq 5}(\vp)(x,t) &= \sum_{n = 2}^\infty \frac{c_n}{2 n + 1} \px \int_{\R^{2 n + 1}} \Tb_n(\etab_n) \hat{\vp}(\eta_1, t) \hat{\vp}(\eta_2, t) \dotsm \hat{\vp}(\eta_{2n + 1}, t) e^{i (\eta_1 + \eta_2 + \dotsb + \eta_{2n + 1}) x} \diff{\etab_n}.
\end{align*}

The main difficulty is in the estimate of the cubic term in \eqref{phihateq}. In Section~\ref{8.1}, we introduce a phase shift
to account for the modified scattering of the solution and carry out a dyadic decomposition of the cubic term. We then consider high frequencies in Section~\ref{8.2}, nonresonant frequencies in Section~\ref{8.3}, near-resonant frequencies in Section~\ref{8.4}, and resonant frequencies in Section~\ref{8.5}.
Finally, in Section \ref{higherorder} we give the easier estimate of the higher-degree remainder term in \eqref{phihateq}.

\subsection{Modified scattering}
\label{8.1}

Define
\[
\Theta(\xi, t) = - A t \xi |\xi|^{1 - \alpha} + \xi \int_0^t \left[\beta_1(t) \Tb_1'(\xi, \xi, - \xi) + \beta_2(t) \Tb_1'(\xi, -\xi, \xi) + \beta_3(t) \Tb_1'(-\xi, \xi, \xi)\right] |\hat{\vp}(\xi, \tau)|^2 \diff{\tau},
\]
where the real-valued functions $\beta_1(t)$, $\beta_2(t)$, and $\beta_3(t)$ will be chosen in Section \ref{sec:spacetime}, and
let
\[
\hat{v}(\xi, t) = e^{i \Theta(\xi, t)} \hat{\vp}(\xi, t).
\]

Then, using \eqref{phihateq}, we obtain that
\begin{equation}
\hat v_t(\xi, t) = e^{i \Theta(\xi, t)} [\hat\vp_t(\xi, t) + i \Theta_t(\xi, t) \hat\vp(\xi, t)] = U_1(\xi, t) + U_2(\xi, t) - e^{i \Theta(\xi, t)} \widehat{\Nc_{\geq 5}(\vp)}(\xi, t),
\label{vteq}
\end{equation}
where
\begin{align}
\label{defU1U2}
\begin{split}
U_1(\xi, t) &=e^{i\Theta(\xi, t)}\bigg\{- i A' \xi \iint_{\R^2}  \Tb_1'(\eta_1,\eta_2, \xi - \eta_1 - \eta_2) \hat\vp(\xi-\eta_1-\eta_2, t)\hat\vp(\eta_1, t)\hat\vp(\eta_2, t) \diff{\eta_1} \diff{\eta_2}\\
&\qquad +i\xi \left[\beta_1(t)\Tb_1'(\xi, \xi, -\xi)+\beta_2(t)\Tb_1'(\xi, -\xi, \xi)+\beta_3(t)\Tb_1'(-\xi,\xi,\xi)\right]|\hat \vp(\xi, t)|^2\hat \vp(\xi, t) \bigg\},\\
U_2(\xi, t) &= \hat v(\xi,t)\bigg[  i\xi \int_0^t \left[\beta_1'(t)\Tb_1'(\xi, \xi, -\xi)+\beta'_2(t)\Tb_1'(\xi, -\xi, \xi)+\beta'_3(t)\Tb_1'(-\xi,\xi,\xi)\right]|\hat \vp(\xi, \tau)|^2 \diff{\tau}\bigg].
\end{split}
\end{align}
The coefficient of $\hat v$ in the bracket in the term $U_2$ is purely imaginary, so it leads to a phase shift in $\hat{v}$ that does not affect its $Z$-norm, and we get from \eqref{vteq} that
\begin{align}
\label{vpZ}
\begin{split}
\|\vp(t)\|_{Z} &= \|(|\xi|+|\xi|^{r+3}) \hat\vp(\xi, t)\|_{L^\infty_\xi}=\|(|\xi|+|\xi|^{r+3}) \hat v(\xi, t)\|_{L^\infty_\xi}\\
&\lesssim \int_0^t \|(|\xi|+|\xi|^{r+3})U_1(\xi, \tau)\|_{L^\infty_\xi}+\|(|\xi|+|\xi|^{r+3})\widehat{\mathcal{N}_{\geq5}(\vp)}(\xi, \tau)\|_{L^\infty_\xi} \diff \tau.
\end{split}
\end{align}
We estimate $U_1$ in Sections~\ref{8.2}--\ref{8.5} and take care of the term $e^{i\Theta(\xi, t)}\widehat{\Nc_{\geq5}(\vp)}$ in Section \ref{higherorder}.

To start with, we recall that $h(x,t) =e^{- A t \partial_x |\partial_x|^{1 - \alpha}}\vp(x,t)$ is defined in \eqref{def-h}.
Then, suppressing the dependence of $\vp$ and $h$ on the time variable $t$, the first integral in $U_1$ can be written in terms of $h$ as
\begin{multline*}
\iint_{\R^2}  \Tb_1'(\eta_1,\eta_2, \xi - \eta_1 - \eta_2) \hat\vp(\xi-\eta_1-\eta_2)\hat\vp(\eta_1)\hat\vp(\eta_2) \diff{\eta_1} \diff{\eta_2}\\
=\iint_{\R^2} \Tb_1'(\eta_1, \eta_2, \xi - \eta_1 - \eta_2) e^{i A t \Phi(\xi,\eta_1,\eta_2)} \hat h(\xi-\eta_1-\eta_2) \hat h(\eta_1)\hat h(\eta_2)  \diff{\eta_1} \diff{\eta_2},
\end{multline*}
where
\begin{equation}
\Phi(\xi,\eta_1,\eta_2) = (\xi-\eta_1-\eta_2) |\xi-\eta_1-\eta_2|^{1 - \alpha} + \eta_1 |\eta_1|^{1 - \alpha} + \eta_2 |\eta_2|^{1 - \alpha} - \xi |\xi|^{1 - \alpha}.
\label{defPhi}
\end{equation}
We will not need to consider the second integral in $U_1$ until we use it to account for modified scattering in Section~\ref{sec:spacetime}.

To carry out the dyadic decomposition, we let $ h_j=P_j h$ and $ \vp_j=P_j\vp$, where $P_j$ is the projection onto frequencies of order $2^j$ with symbol $\psi_j$ defined in \eqref{defpsik}, and rewrite the integral in each dyadic block as
\begin{align}
\label{cubdyadic}
 \iint_{\R^2}   \Tb_1'(\eta_1, \eta_2, \xi - \eta_1 - \eta_2) e^{i A t \Phi(\xi,\eta_1,\eta_2)} \hat h_{j_1}(\eta_1)\hat h_{j_2}(\eta_2) \hat h_{j_3}(\xi-\eta_1-\eta_2) \diff{\eta_1} \diff{\eta_2} .
\end{align}
In the following, we fix $\xi\in \R$ satisfying \eqref{freq_range} and estimate \eqref{cubdyadic} in different regimes of $(j_1,j_2,j_3)$.

\subsection{High frequencies}
\label{8.2}

When
\[
\max\{j_1,j_2,j_3\} \gtrsim 10^{-3}\log_2 |t+1|>0,
\]
we can estimate the cubic terms \eqref{cubdyadic} by using Lemma \ref{multilinear},  with the $L^\infty$-norm placed on the lowest derivative term.
There are, in total, $r+6=13$ derivatives shared by three factors of $\vp$. Thus, we can ensure that the term with least derivatives has at most four derivatives.

To be more specific, introducing the cutoff function $\cutoffxi(\xi, t)$ to restrict attention to the frequency range \eqref{cutoff}, using H\"{o}lder's inequality, Sobolev embedding, and the bootstrap assumptions, we obtain the estimate
\begin{align*}
 & \left\| \xi (|\xi|+|\xi|^{r + 3})\cutoffxi(\xi, t)  \iint_{\R^2}   \Tb_1'(\eta_1, \eta_2, \xi - \eta_1 - \eta_2) e^{i A t \Phi(\xi,\eta_1,\eta_2)} \hat h_{j_1}(\eta_1)\hat h_{j_2}(\eta_2) \hat h_{j_3}(\xi-\eta_1-\eta_2)  \diff{\eta_1} \diff{\eta_2} \right\|_{L^\infty_\xi}\\
 \lesssim ~~& (t + 1)^{(r+7-s)10^{-3}}\|\vp_{\min}\|_{L^2}\|\vp_{\med}\|_{L^{\infty}}\|\vp_{\max}\|_{H^s}\\
\lesssim ~~& (t + 1)^{(r+7-s)10^{-3}} \|\vp_{j_1}\|_{H^s}\|\vp_{j_2}\|_{H^s}\|\vp_{j_3}\|_{H^s},
\end{align*}
where $\max$, $\med$, $\min$ represent the maximum, median, and the minimum of  $j_1$, $j_2$, $j_3$. From \eqref{defsrp},
\[
(r+7-s)10^{-3}<-1.1,
\]
so the right-hand-side of this equation is summable over $j_1$, $j_2$, $j_3$ and the sum is integrable in $t$ on $(0,\infty)$.

\subsection{Nonresonant frequencies}
\label{8.3}

We now only need to consider the case when
\[
\max\{j_1,j_2,j_3\}<10^{-3}\log_2(t+1).
\]
The frequencies are nonresonant if $|j_1-j_3|>1$ or $|j_2-j_3|>1$ and, without loss of generality, we assume that $|j_1-j_3|>1$.
We recall that $h$ is defined in \eqref{def-h} and the symbol $\tilde{\psi}_k$ is defined in \eqref{defpsik}.

From \eqref{defPhi}, we have
\begin{equation}\label{denom1}
\partial_{\eta_1}\Phi(\xi,\eta_1,\eta_2) = (2 - \alpha) \left[|\eta_1|^{1 - \alpha} - |\xi - \eta_1 - \eta_2|^{1 - \alpha}\right].
\end{equation}
Since $|\eta_1|$ and $|\xi-\eta_1-\eta_2|$ are in different dyadic blocks, we have
\[
\big||\eta_1|-|\xi-\eta_1-\eta_2|\big|\gtrsim \max\{|\eta_1|, |\xi-\eta_1-\eta_2|\},
\]
so $|\partial_{\eta_1}\Phi| \gtrsim 2^{(1 - \alpha) \min\{j_1, j_3\}}$.

Using an integration by parts, we have
\begin{align*}
&  \iint_{\R^2}    \Tb_1'(\eta_1, \eta_2, \xi - \eta_1 - \eta_2) e^{i A t \Phi(\xi,\eta_1,\eta_2)} \hat h_{j_1}(\eta_1)\hat h_{j_2}(\eta_2) \hat h_{j_3}(\xi-\eta_1-\eta_2) \diff{\eta_1} \diff{\eta_2}\\
 =~&   \iint_{\R^2}    \frac{\Tb_1'(\eta_1, \eta_2, \xi - \eta_1 - \eta_2)}{i A t\partial_{\eta_1}\Phi(\xi,\eta_1,\eta_2)}\partial_{\eta_1} e^{i A t \Phi(\xi,\eta_1,\eta_2)} \hat h_{j_1}(\eta_1)\hat h_{j_2}(\eta_2) \hat h_{j_3}(\xi-\eta_1-\eta_2) \diff{\eta_1} \diff{\eta_2}\\
 =~& - \frac{W_1 + W_2 + W_3}{A},
\end{align*}
where
\[
W_1(\xi, t)=  \iint_{\R^2} \partial_{\eta_1}\left[ \frac{\Tb_1'(\eta_1, \eta_2, \xi - \eta_1 - \eta_2)}{it\partial_{\eta_1}\Phi(\xi,\eta_1,\eta_2)}\right] e^{i A t \Phi(\xi,\eta_1,\eta_2)} \hat h_{j_1}(\eta_1)\hat h_{j_2}(\eta_2) \hat h_{j_3}(\xi-\eta_1-\eta_2) \diff{\eta_1} \diff{\eta_2},
\]
\[
W_2(\xi, t)=  \iint_{\R^2} \left[ \frac{\Tb_1'(\eta_1, \eta_2, \xi - \eta_1 - \eta_2)}{it\partial_{\eta_1}\Phi(\xi,\eta_1,\eta_2)}\right] e^{i A t \Phi(\xi,\eta_1,\eta_2)} \hat h_{j_1}(\eta_1)\hat h_{j_2}(\eta_2) \partial_{\eta_1}\hat h_{j_3}(\xi-\eta_1-\eta_2) \diff{\eta_1} \diff{\eta_2},
\]\[
W_3(\xi, t)=  \iint_{\R^2} \left[ \frac{\Tb_1'(\eta_1, \eta_2, \xi - \eta_1 - \eta_2)}{it\partial_{\eta_1}\Phi(\xi,\eta_1,\eta_2)}\right] e^{i A t \Phi(\xi,\eta_1,\eta_2)} \partial_{\eta_1} \hat h_{j_1}(\eta_1)\hat h_{j_2}(\eta_2) \hat h_{j_3}(\xi-\eta_1-\eta_2) \diff{\eta_1} \diff{\eta_2}.
\]

{\bf Estimate of $W_1$.}
Since
\begin{align*}
\|W_1\|_{L^\infty_\xi}\lesssim \|\F^{-1}({W}_1)\|_{L^1},
\end{align*}
it suffices to estimate the $L^1_x$ norm of
\[
\iiint_{\R^3} e^{i\xi x}   \partial_{\eta_1}\left[ \frac{\Tb_1'(\eta_1, \eta_2, \xi - \eta_1 - \eta_2)}{it\partial_{\eta_1}\Phi(\xi,\eta_1,\eta_2)}\right] e^{i A t \Phi(\xi,\eta_1,\eta_2)} \hat h_{j_1}(\eta_1)\hat h_{j_2}(\eta_2) \hat h_{j_3}(\xi-\eta_1-\eta_2) \diff{\eta_1} \diff{\eta_2}\diff \xi.
\]
By direct calculation,
\[
\begin{aligned}
& \partial_{\eta_1}\frac{\Tb_1'(\eta_1, \eta_2, \xi - \eta_1 - \eta_2)}{\partial_{\eta_1}\Phi(\xi, \eta_1, \eta_2)}\\
=~& \frac{3 - \alpha}{2 - \alpha} \cdot \frac{\eta_1 |\eta_1|^{1 - \alpha} + (\xi - \eta_1) |\xi - \eta_1|^{1 - \alpha} - (\xi - \eta_1 - \eta_2) |\xi - \eta_1 - \eta_2|^{1 - \alpha} - (\eta_1 + \eta_2) |\eta_1 + \eta_2|^{1 - \alpha}}{|\eta_1|^{1 - \alpha} - |\xi - \eta_1 - \eta_2|^{1 - \alpha}}\\
&  + \frac{\alpha - 1}{2 - \alpha} \cdot \frac{|\eta_1|^{3 - \alpha} + |\eta_2|^{3 - \alpha} + |\xi - \eta_1 - \eta_2|^{3 - \alpha} + |\xi|^{3 - \alpha} - |\eta_1 + \eta_2|^{3 - \alpha} - |\xi - \eta_1|^{3 - \alpha} - |\xi - \eta_2|^{3 - \alpha}}{(|\eta_1|^{1 - \alpha} - |\xi - \eta_1 - \eta_2|^{1 - \alpha})^2}\\
& \qquad \cdot \bigg[\frac{\eta_1}{|\eta_1|^{\alpha + 1}} + \frac{\xi - \eta_1 - \eta_2}{|\xi - \eta_1 - \eta_2|^{\alpha + 1}}\bigg]\\
=~& \frac{\kappa_1(\eta_1,\eta_2,\xi - \eta_1 - \eta_2)}{2 - \alpha} + \frac{\alpha - 1}{2 - \alpha} \cdot \kappa_2(\eta_1, \eta_2, \xi - \eta_1 - \eta_2),
\end{aligned}
\]
where
\begin{align*}
\kappa_1(\eta_1,\eta_2,\eta_3)&=\frac{\partial_{\eta_1}\Tb_1'(\eta_1,\eta_2,\eta_3)-\partial_{\eta_3}\Tb_1'(\eta_1,\eta_2,\eta_3)}{|\eta_1|^{1 - \alpha} - |\eta_3|^{1 - \alpha}},
\\
\kappa_2(\eta_1,\eta_2,\eta_3)&=\Tb_1'(\eta_1,\eta_2,\eta_3)  \frac{\frac{\eta_1}{|\eta_1|^{\alpha + 1}} + \frac{\eta_3}{|\eta_3|^{\alpha + 1}}}{(|\eta_1|^{1 - \alpha} - |\eta_3|^{1 - \alpha})^2}.
\end{align*}

Making a change of variable $\eta_3=\xi-\eta_1-\eta_2$, we see that it suffices to estimate the trilinear operator
\begin{align*}
\frac1{it}\iiint_{\R^3}e^{i(\eta_1+\eta_2+\eta_3) x}   \left[\kappa_1(\eta_1,\eta_2,\eta_3)+\kappa_2(\eta_1,\eta_2,\eta_3)\right] \hat \vp_{j_1}(\eta_1)\hat \vp_{j_2}(\eta_2) \hat \vp_{j_3}(\eta_3) \diff{\eta_1} \diff{\eta_2}\diff\eta_3,
\end{align*}
with symbol
\[
   \left[\kappa_1(\eta_1,\eta_2,\eta_3)+\kappa_2(\eta_1,\eta_2,\eta_3)\right]\psi_{j_1}(\eta_1)\psi_{j_2}(\eta_2)\psi_{j_3}(\eta_3).
\]
From Lemma~\ref{multilinear}, this trilinear operator is bounded on $L^2\times L^2 \times L^\infty \to L^1$ by
\begin{align}
\label{kappaest}
\begin{split}
&\left\|\left[\kappa_1(\eta_1,\eta_2,\eta_3)+\kappa_2(\eta_1,\eta_2,\eta_3)\right]\psi_{j_1}(\eta_1)\psi_{j_2}(\eta_2)\psi_{j_3}(\eta_3)\right\|_{S^\infty}\\
\lesssim~&\Big(\|  \partial_{\eta_1}\Tb_1'(\eta_1,\eta_2,\eta_3)\tilde\psi_{j_1}(\eta_1)\tilde\psi_{j_2}(\eta_2)\tilde\psi_{j_3}(\eta_3)\|_{S^\infty}\\
&\qquad +\|\partial_{\eta_3}\Tb_1'(\eta_1,\eta_2,\eta_3)\tilde\psi_{j_1}(\eta_1)\tilde\psi_{j_2}(\eta_2)\tilde\psi_{j_3}(\eta_3)\|_{S^\infty}\Big) \cdot\left\|\frac{\psi_{j_1}(\eta_1)\psi_{j_2}(\eta_2)\psi_{j_3}(\eta_3)}{|\eta_1|^{1 - \alpha} - |\eta_3|^{1 - \alpha}}\right\|_{S^\infty}\\
&+\bigg(\left\|\frac{\eta_1 \cdot \Tb_1'(\eta_1,\eta_2,\eta_3)}{|\eta_1|^{\alpha + 1}}\tilde\psi_{j_1}(\eta_1)\tilde\psi_{j_2}(\eta_2)\tilde\psi_{j_3}(\eta_3)\right\|_{S^\infty}\\
& \qquad +\left\|  \frac{\eta_3 \cdot \Tb_1'(\eta_1, \eta_2, \eta_3)}{|\eta_3|^{\alpha + 1}}\tilde\psi_{j_1}(\eta_1)\tilde\psi_{j_2}(\eta_2)\tilde\psi_{j_3}(\eta_3)\right\|_{S^\infty}\bigg)\cdot\left\|\frac{\psi_{j_1}(\eta_1)\psi_{j_2}(\eta_2)\psi_{j_3}(\eta_3)}{(|\eta_1|^{1 - \alpha} - |\eta_3|^{1 - \alpha})^2}\right\|_{S^\infty},
\end{split}
\end{align}
where $\tilde{\psi}_k$ is defined in \eqref{defpsik}.

To estimate the terms on the right-hand side of \eqref{kappaest}, we prove the following lemmas.

\begin{lemma}\label{log-log}
If $|j_1-j_3|>1$ and $m \in \Z_+$, then
\[
\left\|\frac{1}{(|\eta_1|^{1 - \alpha} - |\eta_3|^{1 - \alpha})^m}\psi_{j_1}(\eta_1)\psi_{j_2}(\eta_2)\psi_{j_3}(\eta_3)\right\|_{S^\infty} \lesssim 2^{m (\alpha - 1) \min\{j_1, j_3\}}.
\]
\end{lemma}
\begin{proof}
Since
\[
\big||\eta_1|^{1 - \alpha} - |\eta_3|^{1 - \alpha}\big|\gtrsim 2^{1-\alpha}
\]

By the definition of the $S^\infty$-norm \eqref{Sinf} and the definition of $\psi_k$  \eqref{defpsik}, we have that
\begin{align*}
&\left\|\frac{\psi_{j_1}(\eta_1)\psi_{j_2}(\eta_2)\psi_{j_3}(\eta_3)}{(|\eta_1|^{1 - \alpha} - |\eta_3|^{1 - \alpha})^m}\right\|_{S^\infty}\\
=~&\left\|\iiint_{\R^3} \frac{\psi_{j_1}(\eta_1)\psi_{j_2}(\eta_2)\psi_{j_3}(\eta_3)}{(|\eta_1|^{1 - \alpha} - |\eta_3|^{1 - \alpha})^m} e^{i(y_1\eta_1+y_2\eta_2+y_3\eta_3)}\diff \eta_1\diff\eta_2\diff\eta_3\right\|_{L^1}\\
=~&\iiint_{\R^3}\left|\iiint_{\R^3} \frac{\psi_0(2^{-j_1}\eta_1)\psi_0(2^{-j_2}\eta_2)\psi_0(2^{-j_3}\eta_3)}{(|\eta_1|^{1 - \alpha} - |\eta_3|^{1 - \alpha})^m} e^{i(y_1\eta_1+y_2\eta_2+y_3\eta_3)}\diff \eta_1\diff\eta_2\diff\eta_3\right|\diff y_1\diff y_2\diff y_3\\
\lesssim~& 2^{m (\alpha - 1) \min\{j_1, j_3\}},
\end{align*}
where the last inequality comes from using oscillatory integral estimates, together with the facts that the support of $\psi_0$ is $(-8/5, -5/8)\cup(5/8,8/5)$ and $|j_1-j_3|>1$.
\end{proof}

For the other symbols in \eqref{kappaest}, we have the following estimates.

\begin{lemma} \label{estT1}
For any $j_1, j_2, j_3 \in \Z$,
\begin{align}
\|  \partial_{\eta_1} \Tb_1'(\eta_1, \eta_2, \eta_3)\tilde\psi_{j_1}(\eta_1) \tilde\psi_{j_2}(\eta_2) \tilde\psi_{j_3}(\eta_3)\|_{S^\infty} &\lesssim 2^{(2 - \alpha) \min\{j_2, j_3\}},\label{est_symb1}\\
\|\Tb_1'(\eta_1, \eta_2, \eta_3) \tilde\psi_{j_1}(\eta_1) \tilde\psi_{j_2}(\eta_2) \tilde\psi_{j_3}(\eta_3)\|_{S^\infty} &\lesssim 2^{\min\{j_1,j_2,j_3\} + (2 - \alpha) \med\{j_1,j_2,j_3\}},\label{est_symb1'}
\\
\left\|  \frac{\eta_1 \cdot \Tb_1'(\eta_1, \eta_2, \eta_3)}{|\eta_1|^{\alpha + 1}} \tilde\psi_{j_1}(\eta_1) \tilde\psi_{j_2}(\eta_2) \tilde\psi_{j_3}(\eta_3)\right\|_{S^\infty} &\lesssim 2^{\min\{j_1, j_2, j_3\} + (2 - \alpha) \med\{j_1, j_2, j_3\} - \alpha j_1},\label{est_symb2}\\
\left\|\frac{\Tb_1'(\eta_1, \eta_2, \eta_3)}{\eta_1} \tilde\psi_{j_1}(\eta_1) \tilde\psi_{j_2}(\eta_2) \tilde\psi_{j_3}(\eta_3)\right\|_{S^\infty} &\lesssim 2^{\min\{j_2, j_3\}}.\label{est_symb3}
\end{align}
Furthermore, since $\Tb_1$ is symmetric,
\begin{align*}
\|  \partial_{\eta_3}\Tb_1'(\eta_1, \eta_2, \eta_3) \tilde\psi_{j_1}(\eta_1) \tilde\psi_{j_2}(\eta_2) \tilde\psi_{j_3}(\eta_3)\|_{S^\infty} &\lesssim 2^{(2 - \alpha) \min\{j_1, j_2\}},\\
\left\|  \frac{\eta_3 \cdot \Tb_1'(\eta_1,\eta_2,\eta_3)}{|\eta_3|^{\alpha + 1}} \tilde\psi_{j_1}(\eta_1) \tilde\psi_{j_2}(\eta_2) \tilde\psi_{j_3}(\eta_3)\right\|_{S^\infty} &\lesssim 2^{\min\{j_1, j_2, j_3\} + (2 - \alpha) \med\{j_1, j_2, j_3\} - \alpha j_3},\\
\left\|\frac{\Tb_1'(\eta_1, \eta_2, \eta_3)}{\eta_3} \tilde\psi_{j_1}(\eta_1) \tilde\psi_{j_2}(\eta_2) \tilde\psi_{j_3}(\eta_3)\right\|_{S^\infty} &\lesssim 2^{\min\{j_1, j_2\}}.
\end{align*}
\end{lemma}

\begin{proof} {\bf 1.} We prove \eqref{est_symb1} first. Using the inverse Fourier transform in $(\eta_1, \eta_2, \eta_3)$, we obtain that
\begin{align*}
& \frac{A}{(2 - \alpha) (3 - \alpha)} \F^{-1}[\partial_{\eta_1}\Tb_1'(\eta_1,\eta_2,\eta_3) \tilde\psi_{j_1}(\eta_1) \tilde\psi_{j_2}(\eta_2) \tilde\psi_{j_3}(\eta_3)]\\
=~&\iiint_{\R^3}e^{i(y_1 \eta_1 + y_2 \eta_2 + y_3 \eta_3)} \partial_{\eta_1} \left[\int_{\R}\frac{\prod_{j=1}^3(1-e^{i\eta_j\zeta})}{\zeta^3} |\zeta|^{\alpha - 1} \sgn{\zeta} \diff \zeta\right] \tilde\psi_{j_1}(\eta_1) \tilde\psi_{j_2}(\eta_2) \tilde\psi_{j_3}(\eta_3) \diff \eta_1\diff \eta_2\diff \eta_3\\
=~&\iiint_{\R^3}  \left[\int_{\R}\frac{-i |\zeta|^\alpha e^{i\eta_1(\zeta+y_1)}(e^{iy_2\eta_2}-e^{i\eta_2(\zeta+y_2)})(e^{iy_3\eta_3}-e^{i\eta_3(\zeta+y_3)})}{\zeta^{3}}\diff \zeta\right] \tilde\psi_{j_1}(\eta_1) \tilde\psi_{j_2}(\eta_2) \tilde\psi_{j_3}(\eta_3) \diff \eta_1\diff \eta_2\diff \eta_3\\
=~& \int_{\R}\frac{-i |\zeta|^\alpha}{\zeta^3} \cdot\left[\F^{-1}[\tilde\psi_{j_1}](y_1+\zeta)\right]\cdot \left[\F^{-1}[\tilde\psi_{j_2}](y_2)-\F^{-1}[\tilde\psi_{j_2}](\zeta+y_2)\right] \left[\F^{-1}[\tilde\psi_{j_3}](y_3)-\F^{-1}[\tilde\psi_{j_3}](\zeta+y_3)\right] \diff \zeta.
\end{align*}

We observe that
\begin{align*}
\left|\F^{-1}[\tilde\psi_{j_1}](y_1+\zeta)\right| &= 2^{j_1}\left|\F^{-1}[\tilde\psi_0](2^{j_1}(y_1+\zeta))\right|,\\
\left|\F^{-1}[\tilde\psi_{j_2}](y_2)-\F^{-1}[\tilde\psi_{j_2}](\zeta+y_2)\right| &= 2^{j_2} \left|\F^{-1}[\tilde\psi_0](2^{j_2} y_2)-\F^{-1}[\tilde\psi_0](2^{j_2}(\zeta+y_2))\right|,\\
\left|\F^{-1}[\tilde\psi_{j_3}](y_3)-\F^{-1}[\tilde\psi_{j_3}](\zeta+y_3)\right| &= 2^{j_3} \left|\F^{-1}[\tilde\psi_0](2^{j_3} y_3)-\F^{-1}[\tilde\psi_0](2^{j_3}(\zeta+y_3))\right|,
\end{align*}
and
\begin{align*}
\int_{\R} \left|\F^{-1}[\tilde\psi_0](2^{j_1}(y_1+\zeta))\right|\diff y_1&\lesssim 2^{-j_1},\\
\int_{\R} \left|\F^{-1}[\tilde\psi_{j_2}](2^{j_2} y_2)-\F^{-1}[\tilde\psi_{j_2}](2^{j_2}(\zeta+y_2))\right| \diff y_2 &\lesssim \min\{2^{-j_2}, |\zeta|\},\\
\int_{\R} \left|\F^{-1}[\tilde\psi_{j_3}](2^{j_3} y_3)-\F^{-1}[\tilde\psi_{j_3}](2^{j_3}(\zeta+y_3))\right|\diff y_3 &\lesssim \min\{2^{-j_3}, |\zeta|\}.
\end{align*}
Therefore, we have
\begin{align*}
&\left\|\F^{-1}[  \partial_{\eta_1}\Tb_1'(\eta_1,\eta_2,\eta_3)\tilde\psi_{j_1}(\eta_1)\tilde\psi_{j_2}(\eta_2)\tilde\psi_{j_3}(\eta_3)]\right\|_{L^1}\\
\lesssim~&\int_{\R}\frac{1}{|\zeta|^{3 - \alpha}} 2^{j_2+j_3}\min\{2^{-j_2}, |\zeta|\} \min\{2^{-j_3}, |\zeta|\} \diff\zeta\\
=~&2^{j_2+j_3}\bigg(\int_{|\zeta|>\max\{2^{-j_2}, 2^{-j_3}\}}\frac1{|\zeta|^{3 - \alpha}}2^{-j_2-j_3}\diff\zeta+\int_{|\zeta|<\max\{2^{-j_2}, 2^{-j_3}\}}\frac1{|\zeta|^{2 - \alpha}}\min\{2^{-j_2}, 2^{-j_3}\}\diff\zeta\bigg)\\
\lesssim~& 2^{(2 - \alpha) \min\{j_2, j_3\}}.
\end{align*}

{\bf 2.} Next, we prove \eqref{est_symb1'} and \eqref{est_symb2}. The proof of \eqref{est_symb1'} is similar to that of \eqref{est_symb1}. We first use the inverse Fourier transform and write
\begin{align*}
& \frac{A}{(2 - \alpha) (3 - \alpha)} \F^{-1}\left[  \Tb_1'(\eta_1,\eta_2,\eta_3)\tilde\psi_{j_1}(\eta_1)\tilde\psi_{j_2}(\eta_2)\tilde\psi_{j_3}(\eta_3)\right]\\
=~&\iiint_{\R^3}e^{i(y_1\eta_1+y_2\eta_2+y_3\eta_3)}   \left[\int_{\R}\frac{\prod_{j=1}^3(1-e^{i\eta_j\zeta})}{\zeta^{3}} |\zeta|^{\alpha - 1} \sgn{\zeta} \diff \zeta\right]\tilde\psi_{j_1}(\eta_1)\tilde\psi_{j_2}(\eta_2)\tilde\psi_{j_3}(\eta_3) \diff \eta_1\diff \eta_2\diff \eta_3\\
=~&\iiint_{\R^3}  \left[\int_{\R}\frac{ (e^{iy_1\eta_1}-e^{i\eta_1(\zeta+y_1)})(e^{iy_2\eta_2}-e^{i\eta_2(\zeta+y_2)})(e^{iy_3\eta_3}-e^{i\eta_3(\zeta+y_3)})}{ |\zeta|^{4 - \alpha}}\diff \zeta\right]\tilde\psi_{j_1}(\eta_1)\tilde\psi_{j_2}(\eta_2) \tilde\psi_{j_3}(\eta_3) \diff \eta_1\diff \eta_2\diff \eta_3\\
=~& \int_{\R}\frac{1 }{|\zeta|^{4 - \alpha}} \left[\F^{-1}[\tilde\psi_{j_1}](y_1)-\F^{-1}[\tilde\psi_{j_1}](\zeta+y_1)\right]\\
& \qquad \cdot \left[\F^{-1}[\tilde\psi_{j_2}](y_2)-\F^{-1}[\tilde\psi_{j_2}](\zeta+y_2)\right] \cdot \left[\F^{-1}[\tilde\psi_{j_3}](y_3)-\F^{-1}[\tilde\psi_{j_3}](\zeta+y_3)\right] \diff \zeta.
\end{align*}
Taking the $L^1$-norm, we obtain
\begin{align*}
&\left\|\F^{-1}[\Tb_1'(\eta_1,\eta_2,\eta_3)\tilde\psi_{j_1}(\eta_1)\tilde\psi_{j_2}(\eta_2)\tilde\psi_{j_3}(\eta_3)]\right\|_{L^1}\\
\lesssim~&\int_{\R} 2^{j_1+j_2+j_3}\frac1{|\zeta|^{4 - \alpha}}\min\{2^{-j_1},|\zeta|\}\min\{2^{-j_2},|\zeta|\}\min\{2^{-j_3},|\zeta|\}\diff\zeta\\
\lesssim~& \int_{|\zeta|>\max\{2^{-j_1},2^{-j_2},2^{-j_3}\}}\frac1{|\zeta|^{4 - \alpha}}\diff\zeta\\*
&+\int_{\med\{2^{-j_1},2^{-j_2},2^{-j_3}\}<|\zeta|<\max\{2^{-j_1},2^{-j_2},2^{-j_3}\}}2^{\min\{j_1,j_2,j_3\}}\frac1{|\zeta|^{3 - \alpha}}\diff \zeta\\*
&+\int_{|\zeta|<\med\{2^{-j_1},2^{-j_2},2^{-j_3}\}}2^{\med\{j_1,j_2,j_3\}+\min\{j_1,j_2,j_3\}}\frac1{|\zeta|^{2 - \alpha}}\diff \zeta\\
\lesssim~&2^{(3 - \alpha) \min\{j_1,j_2,j_3\}} + 2^{\min\{j_1,j_2,j_3\} + (2 - \alpha) \med\{j_1,j_2,j_3\}}\\
\lesssim~& 2^{\min\{j_1,j_2,j_3\} + (2 - \alpha) \med\{j_1,j_2,j_3\}},
\end{align*}
which proves \eqref{est_symb1'}.

As for \eqref{est_symb2}, we define
\[
\tilde{\tilde\psi}_k(\eta) =\sum\limits_{j=k-3}^{k+3}\psi_j(\eta).
\]
It then follows from the support of $\psi_k$ and the fact that the $\psi_k$ form a partition of unity that
\begin{align*}
 &\frac{\eta_1 \cdot \Tb_1'(\eta_1,\eta_2,\eta_3)}{|\eta_1|^{\alpha + 1}}\tilde\psi_{j_1}(\eta_1)\tilde\psi_{j_2}(\eta_2)\tilde\psi_{j_3}(\eta_3)
 \\
 &\qquad = [\Tb_1'(\eta_1,\eta_2,\eta_3)\tilde\psi_{j_1}(\eta_1)\tilde\psi_{j_2}(\eta_2)\tilde\psi_{j_3}(\eta_3)]\cdot\left[ \frac{\eta_1}{|\eta_1|^{\alpha + 1}}\tilde{\tilde\psi}_{j_1}(\eta_1)\tilde{\tilde\psi}_{j_2}(\eta_2)\tilde{\tilde\psi}_{j_3}(\eta_3)\right].
\end{align*}
By Lemma \ref{multilinear}, we have
\begin{align}\label{lm81eq1}
\begin{split}
& \left\|\frac{\eta_1 \cdot \Tb_1'(\eta_1,\eta_2,\eta_3)}{|\eta_1|^{\alpha + 1}}\tilde\psi_{j_1}(\eta_1)\tilde\psi_{j_2}(\eta_2)\tilde\psi_{j_3}(\eta_3)\right\|_{S^\infty}\\
\lesssim~ & \left\|\Tb_1'(\eta_1,\eta_2,\eta_3)\tilde\psi_{j_1}(\eta_1)\tilde\psi_{j_2}(\eta_2)\tilde\psi_{j_3}(\eta_3)\right\|_{S^\infty}\bigg\| \frac{\eta_1 \tilde{\tilde\psi}_{j_1}(\eta_1)\tilde{\tilde\psi}_{j_2}(\eta_2) \tilde{\tilde\psi}_{j_3}(\eta_3)}{|\eta_1|^{\alpha + 1}}\bigg\|_{S^\infty}.
\end{split}
\end{align}
In view of \eqref{est_symb1'}, we only need to estimate the second term. To this end, we have
\begin{align*}
&\left\|\frac{\eta_1}{|\eta_1|^{\alpha + 1}}\tilde{\tilde\psi}_{j_1}(\eta_1)\tilde{\tilde\psi}_{j_2}(\eta_2)\tilde{\tilde\psi}_{j_3}(\eta_3)\right\|_{S^\infty}=\left\|\int_{\R} \frac{\eta_1 \cdot \tilde{\tilde\psi}_{j_1}(\eta_1)}{|\eta_1|^{\alpha + 1}} e^{i\eta_1 y_1}\diff \eta_1 \F^{-1}[\tilde{\tilde\psi}_{j_2}](y_2) \F^{-1}[\tilde{\tilde\psi}_{j_3}](y_3) \right\|_{L^1}\lesssim 2^{- \alpha j_1}.
\end{align*}
Using \eqref{lm81eq1}, we then get \eqref{est_symb2}. The proof of \eqref{est_symb3} is similar, and we omit it.
\end{proof}

Applying the above lemmas to \eqref{kappaest}, we obtain that
\begin{align*}
\|W_1\|_{L^\infty_\xi}\lesssim (t+1)^{-1} \|\px\vp_{\med\{j_1, j_2, j_3\}}\|_{L^\infty}\|\vp_{\min\{j_1, j_2, j_3\}}\|_{L^2}\|\vp_{\max\{j_1, j_2, j_3\}}\|_{L^2}.
\end{align*}
Using \eqref{def-h} and  Lemma \ref{disp}, we have
\begin{multline*}
\|\px\vp_{\med\{j_1, j_2, j_3\}}\|_{L^\infty}\lesssim
(t+1)^{-1/2}\||\xi|^{1 + \alpha / 2}\hat{h}_{\med\{j_1, j_2, j_3\}}\|_{L^\infty_\xi}\\
+(t+1)^{-3/4}\left[\||\px|^{3 \alpha / 4}P_{\med\{j_1, j_2, j_3\}}(x\px h)\|_{L^2}+\||\px|^{3 \alpha / 4} h_{\med\{j_1, j_2, j_3\}}\|_{L^2}  \right].
\end{multline*}
Therefore,
\begin{align*}
\|W_1\|_{L^\infty_\xi}&\lesssim (t+1)^{-1.5}\Bigl[\mathbf1_{\med\{j_1, j_2, j_3\}\leq 0}2^{(\alpha / 2) \med\{j_1, j_2, j_3\}}\||\xi|\hat{h}_{\med\{j_1, j_2, j_3\}}\|_{L^\infty_\xi} \\
&\quad+\mathbf1_{\med\{j_1, j_2, j_3\}> 0}2^{(- r - 2 + \alpha / 2)\med\{j_1, j_2, j_3\}}\||\xi|^{r+3}\hat{h}_{\med\{j_1, j_2, j_3\}}\|_{L^\infty_\xi}\Bigr]
\\
&\hspace{2.5in}\cdot\|\vp_{\min\{j_1, j_2, j_3\}}\|_{L^2}\|\vp_{\max\{j_1, j_2, j_3\}}\|_{L^2}\\
&\quad+(t+1)^{-1.75}\Bigl[\||\px|^{3 \alpha / 4}P_{\med\{j_1, j_2, j_3\}}(x\px h)\|_{L^2}+\||\px|^{3 \alpha / 4} h_{\med\{j_1, j_2, j_3\}}\|_{L^2}  \Bigr]\\*
& \hspace{2.5in} \cdot \|\vp_{\min\{j_1, j_2, j_3\}}\|_{L^2}\|\vp_{\max\{j_1, j_2, j_3\}}\|_{L^2}\\
&\lesssim (t+1)^{-1.5}\Bigl[\mathbf1_{\med\{j_1, j_2, j_3\} \leq 0}2^{(\alpha / 2) \med\{j_1, j_2, j_3\}}+\mathbf1_{\med\{j_1, j_2, j_3\} > 0}2^{(- r - 2 + \alpha / 2) \med\{j_1, j_2, j_3\}} \Bigr]\\
& \hspace{2.5in} \cdot \|h_{\med\{j_1, j_2, j_3\}}\|_Z\|\vp_{\min\{j_1, j_2, j_3\}}\|_{L^2}\|\vp_{\max\{j_1, j_2, j_3\}}\|_{L^2}\\
&\quad+(t+1)^{-1.75}\left[\||\px|^{3 \alpha / 4}P_{\med\{j_1, j_2, j_3\}}(x\px h)\|_{L^2}+\||\px|^{3 \alpha / 4} h_{\med\{j_1, j_2, j_3\}}\|_{L^2}  \right]\\
& \hspace{2.5in} \cdot \|\vp_{\min\{j_1, j_2, j_3\}}\|_{L^2}\|\vp_{\max\{j_1, j_2, j_3\}}\|_{L^2}.
\end{align*}

{\bf Estimate of $W_2$ and $W_3$.}
We rewrite $W_2$ as
\begin{align*}
&\iint_{\R^2} \left[ \frac{\Tb_1'(\eta_1, \eta_2, \xi - \eta_1 - \eta_2)}{it\partial_{\eta_1}\Phi(\xi,\eta_1,\eta_2) (\xi - \eta_1 - \eta_2)}\right]
\\
&\qquad\qquad e^{i A t \Phi(\xi,\eta_1,\eta_2)} \hat h_{j_1}(\eta_1)\hat h_{j_2}(\eta_2) \left[(\xi - \eta_1 - \eta_2)\partial_{\eta_1}\hat h_{j_3}(\xi-\eta_1-\eta_2)\right] \diff{\eta_1} \diff{\eta_2}.
\end{align*}
Using \eqref{denom1} and making the change of variable $\eta_3 = \xi - \eta_1-\eta_2$, we see that,
in view of the multilinear estimate Lemma \ref{multilinear}, we need to estimate the $S^\infty$-norm of the symbol
\[
\frac{\Tb_1'(\eta_1,\eta_2,\eta_3)}{(|\eta_1|^{1 - \alpha} - |\eta_3|^{1 - \alpha})\eta_3}\psi_{j_1}(\eta_1)\psi_{j_2}(\eta_2)\tilde\psi_{j_3}(\eta_3).
\]
Using Lemma \ref{log-log} and Lemma \ref{estT1}, as in the estimate of $W_1$, we obtain that
\begin{align*}
\|W_2\|_{L^\infty_\xi}\lesssim (t+1)^{-1}\|\px\vp_{\min\{j_1,j_2\}}\|_{L^\infty}\|\xi\partial_\xi\hat{h}_{j_3}\|_{L^2_\xi}\||\px|^{\alpha - 1} \vp_{\max\{j_1,j_2\}}\|_{L^2}.
\end{align*}
From \eqref{def-h} and Lemma \ref{disp}, we then have
\begin{align*}
\|W_2\|_{L^\infty_\xi}
& \lesssim (t+1)^{-1.5}\Big(\mathbf1_{\min\{j_1,j_2\}\leq 0}2^{(\alpha / 2) \min\{j_1,j_2\}}+\mathbf1_{\min\{j_1,j_2\}> 0}2^{(-r - 2 + \alpha / 2)\min\{j_1,j_2\}} \Big)\\
& \hspace{2.5in} \cdot\|h_{\min\{j_1,j_2\}}\|_Z  \|\xi\partial_\xi\hat{h}_{j_3}\|_{L^2_\xi}\||\px|^{\alpha - 1} \vp_{\max\{j_1,j_2\}}\|_{L^2}\\
& \qquad +(t+1)^{-1.75}\left[\||\px|^{3 \alpha / 4}P_{\min\{j_1,j_2\}}(x\px h)\|_{L^2}+\||\px|^{3 \alpha / 4} h_{\min\{j_1,j_2\}}\|_{L^2}  \right]\\
& \hspace{2.5in}\cdot \|\xi\partial_\xi\hat{h}_{j_3}\|_{L^2_\xi}\||\px|^{\alpha - 1} \vp_{\max\{j_1,j_2\}}\|_{L^2}.
\end{align*}
A similar argument also gives that
\begin{align*}
\|W_3\|_{L^\infty_\xi}
&\lesssim (t+1)^{-1.5}\Big(\mathbf1_{\min\{j_2,j_3\}\leq 0}2^{(\alpha / 2) \min\{j_2,j_3\}}+\mathbf1_{\min\{j_2,j_3\}> 0}2^{(-r - 2 + \alpha / 2)\min\{j_2,j_3\}} \Big)\\
& \hspace{2.5in}\cdot \|{h}_{\min\{j_2,j_3\}}\|_{Z}  \|\xi\partial_\xi\hat{h}_{j_1}\|_{L^2_\xi}\||\px|^{\alpha - 1} \vp_{\max\{j_2,j_3\}}\|_{L^2}\\
& \qquad +(t+1)^{-1.75}\left[\||\px|^{3 \alpha / 4}P_{\min\{j_2,j_3\}}(x\px h)\|_{L^2}+\||\px|^{3 \alpha / 4} h_{\min\{j_2,j_3\}}\|_{L^2}  \right]\\
& \hspace{2.5in}\cdot \|\xi\partial_\xi\hat{h}_{j_1}\|_{L^2_\xi}\||\px|^{\alpha - 1} \vp_{\max\{j_2,j_3\}}\|_{L^2}.
\end{align*}

In conclusion, introducing the cutoff function $\cutoffxi$, we have for nonresonant frequencies that
\begin{align*}
&\bigg\|\xi(|\xi|+|\xi|^{r+3})\cutoffxi(\xi,t)\iint_{\R^2} \Tb_1'(\eta_1,\eta_2,\xi-\eta_1-\eta_2)e^{i A t \Phi(\xi,\eta_1,\eta_2)}\hat h_{j_1}(\eta_1)\hat h_{j_2}(\eta_2)\hat h_{j_3}(\xi-\eta_1-\eta_2)\diff\eta_1\diff\eta_2\bigg\|_{L^\infty_\xi}\\
\lesssim~ &(t+1)^{(r+4)p_1}\left(\|W_1\|_{L^\infty_\xi}+\|W_2\|_{L^\infty_\xi}+\|W_3\|_{L^\infty_\xi}\right)\\
\lesssim~ & (t+1)^{-1.5 + (r + 4) p_1}\Big(\mathbf1_{\med\{j_1, j_2, j_3\} \leq 0}2^{(\alpha / 2) \med\{j_1, j_2, j_3\}}+\mathbf1_{\med\{j_1, j_2, j_3\} > 0}2^{(- r - 2 + \alpha / 2) \med\{j_1, j_2, j_3\}} \Big)\\
& \hspace{2.5in} \cdot \|h_{\med\{j_1, j_2, j_3\}}\|_Z\|\vp_{\min\{j_1, j_2, j_3\}}\|_{L^2}\|\vp_{\max\{j_1, j_2, j_3\}}\|_{L^2}\\
&\quad+(t+1)^{-1.75 + (r + 4) p_1}\left[\||\px|^{3 \alpha / 4}P_{\med\{j_1, j_2, j_3\}}(x\px h)\|_{L^2}+\||\px|^{3 \alpha / 4} h_{\med\{j_1, j_2, j_3\}}\|_{L^2}  \right]\\
& \hspace{2.5in} \cdot \|\vp_{\min\{j_1, j_2, j_3\}}\|_{L^2}\|\vp_{\max\{j_1, j_2, j_3\}}\|_{L^2}\\
& \quad + (t+1)^{-1.5 + (r + 4) p_1}\Big(\mathbf1_{\min\{j_1,j_2\}\leq 0}2^{(\alpha / 2) \min\{j_1,j_2\}}+\mathbf1_{\min\{j_1,j_2\}> 0}2^{(- r - 2 + \alpha / 2)\min\{j_1,j_2\}} \Big)\\
& \hspace{2.5in} \cdot\|h_{\min\{j_1,j_2\}}\|_Z  \|\xi\partial_\xi\hat{h}_{j_3}\|_{L^2_\xi}\||\px|^{\alpha - 1} \vp_{\max\{j_1,j_2\}}\|_{L^2}\\
& \qquad +(t+1)^{-1.75 + (r + 4) p_1}\left[\||\px|^{3 \alpha / 4}P_{\min\{j_1,j_2\}}(x\px h)\|_{L^2}+\||\px|^{3 \alpha / 4} h_{\min\{j_1,j_2\}}\|_{L^2}  \right]\\
& \hspace{2.5in}\cdot \|\xi\partial_\xi\hat{h}_{j_3}\|_{L^2_\xi}\||\px|^{\alpha - 1} \vp_{\max\{j_1,j_2\}}\|_{L^2}\\
& \quad +  (t+1)^{-1.5 + (r + 4) p_1}\Big(\mathbf1_{\min\{j_2,j_3\}\leq 0}2^{(\alpha / 2) \min\{j_2,j_3\}}+\mathbf1_{\min\{j_2,j_3\}> 0}2^{(- r - 2 + \alpha / 2)\min\{j_2,j_3\}} \Big)\\
& \hspace{2.5in}\cdot \|{h}_{\min\{j_2,j_3\}}\|_{Z}  \|\xi\partial_\xi\hat{h}_{j_1}\|_{L^2_\xi}\||\px|^{\alpha - 1} \vp_{\max\{j_2,j_3\}}\|_{L^2}\\
& \qquad +(t+1)^{-1.75 + (r + 4) p_1}\left[\||\px|^{3 \alpha / 4}P_{\min\{j_2,j_3\}}(x\px h)\|_{L^2}+\||\px|^{3 \alpha / 4} h_{\min\{j_2,j_3\}}\|_{L^2}  \right]\\
& \hspace{2.5in}\cdot \|\xi\partial_\xi\hat{h}_{j_1}\|_{L^2_\xi}\||\px|^{\alpha - 1} \vp_{\max\{j_2,j_3\}}\|_{L^2}.
\end{align*}
The right-hand-side is summable over $j_1,j_2,j_3$ and the sum is integrable in $t$ on $(0,\infty)$.

\subsection{Near-resonant frequencies}
\label{8.4}

The remaining dyadic blocks to consider are when
\begin{align}
\label{resregion}
\max\{j_1,j_2,j_3\}<10^{-3}\log_2(t+1),\qquad |j_3-j_2|\leq 1,\qquad |j_3-j_1|\leq 1.
\end{align}

To estimate \eqref{cubdyadic} in this region, we consider the following two cases:
\begin{enumerate}[(i)]
\item Frequencies $\eta_1, \eta_2$ and $\xi-\eta_1-\eta_2$ have the same sign.

By the definition of cut-off function $\psi$, we then have
\[
\frac58 2^{j_1} \leq |\eta_1|\leq\frac85 2^{j_1},\quad \frac58 2^{j_2} \leq |\eta_2|\leq\frac85 2^{j_2},\quad \frac58 2^{j_3} \leq |\xi-\eta_1-\eta_2|\leq\frac85 2^{j_3},
\]
and thus,
\[
\frac58 (2^{j_1}+2^{j_2}+ 2^{j_3})\leq |\xi| \leq \frac85 (2^{j_1}+2^{j_2}+ 2^{j_3}).
\]
This corresponds to the region near the space resonance $\eta_1 = \eta_2 = \xi - \eta_1 - \eta_2 = \xi / 3$.

\item Frequencies $\eta_1, \eta_2$ and $\xi-\eta_1-\eta_2$ do not have the same sign.

This corresponds to the region near the space-time resonances $(\eta_1,\eta_2)=(\xi,\xi)$, $(\xi,-\xi)$, or $(-\xi,\xi)$ separately. Since the symbol $\Tb'_1(\eta_1, \eta_2, \eta_3)$ is symmetric in $\eta_1$, $\eta_2$, and $\eta_3$, it suffices to consider \eqref{cubdyadic} in the region near $(\xi, \xi)$.
\end{enumerate}

To estimate the integral \eqref{cubdyadic} under \eqref{resregion}, we make a further dyadic decomposition. Denoting $(\xi,\xi, -\xi)$ or $(\xi/3,\xi/3, \xi/3)$ by $(\xi_1,\xi_2, \xi_3)$, we decompose \eqref{resregion} using the cut-off functions $\psi_{k_1}$ and $\psi_{k_2}$. Since
\[
\sum\limits_{(k_1,k_2)\in\Z^2}\psi_{k_1}(\eta_1-\xi_1)\psi_{k_2}(\eta_2-\xi_2)=1,
\]
we can write the integral \eqref{cubdyadic} as
\begin{multline*}
\iint_{\R^2}   \Tb_1'(\eta_1, \eta_2, \xi - \eta_1 - \eta_2) e^{i A t \Phi(\xi,\eta_1,\eta_2)} \hat h_{j_1}(\eta_1)\hat h_{j_2}(\eta_2) \hat h_{j_3}(\xi-\eta_1-\eta_2) \\
\cdot \bigg[\sum\limits_{k_1=-\infty}^{\max\{j_1,j_3\}+1}\psi_{k_1}(\eta_1-\xi_1) \bigg] \cdot \bigg[\sum\limits_{k_2=-\infty}^{\max\{j_2,j_3\}+1}\psi_{k_2}(\eta_2-\xi_2)\bigg] \diff{\eta_1} \diff{\eta_2},
\end{multline*}
where
\[
\bigg[\sum\limits_{k_1=-\infty}^{\max\{j_1,j_3\}+1}\psi_{k_1}(\eta_1-\xi_1) \bigg] \cdot \bigg[\sum\limits_{k_2=-\infty}^{\max\{j_2,j_3\}+1}\psi_{k_2}(\eta_2-\xi_2)\bigg]=1
\]
on the support of $\hat h_{j_1}(\eta_1)\hat h_{j_2}(\eta_2) \hat h_{j_3}(\xi-\eta_1-\eta_2)$.

In this subsection, we restrict our attention to the terms satisfying
\[
k_1\geq \log_2[\varrho(t)] \qquad \text{or} \qquad k_2\geq \log_2[\varrho(t)],
\]
 where
 \begin{equation}
 \varrho(t)=(t+1)^{-0.49}.
 \label{def_rho}
 \end{equation}
The case of $k_1< \log_2[\varrho(t)]$ and $k_2< \log_2[\varrho(t)]$, which includes the resonant frequencies, will be discussed in Section \ref{8.5}.

Since the integrals are symmetric in $\eta_1$ and $\eta_2$, we can assume without loss of generality that $j_1\geq k_1\geq k_2\ge \log_2[\varrho(t)]$.
Using integration by parts, we write each integral as
\begin{equation}\label{resInt1}
\begin{aligned}
& \iint_{\R^2}   \frac{\Tb_1'(\eta_1, \eta_2, \xi - \eta_1 - \eta_2)}{2 (2 - \alpha) i A t (|\eta_1|^{1 - \alpha} - |\xi-\eta_1-\eta_2|^{1 - \alpha})}\partial_{\eta_1} e^{i A t \Phi(\xi,\eta_1,\eta_2)} \hat h_{j_1}(\eta_1)\hat h_{j_2}(\eta_2) \hat h_{j_3}(\xi-\eta_1-\eta_2)\\
&\qquad \cdot \bigg[\psi_{k_1}(\eta_1-\xi_1)  \cdot \psi_{k_2}(\eta_2-\xi_2)\bigg] \diff{\eta_1} \diff{\eta_2},\\
  =~&- \frac{V_1+V_2+V_3+V_4}{2i (2 - \alpha) At},
\end{aligned}
\end{equation}
where
\begin{align*}
V_1(\xi, t)=&\iint_{\R^2}   \partial_{\eta_1}\left[\frac{\Tb_1'(\eta_1, \eta_2, \xi - \eta_1 - \eta_2)}{|\eta_1|^{1 - \alpha} - |\xi-\eta_1-\eta_2|^{1 - \alpha}}\right] e^{i A t \Phi(\xi,\eta_1,\eta_2)} \hat h_{j_1}(\eta_1)\hat h_{j_2}(\eta_2) \hat h_{j_3}(\xi-\eta_1-\eta_2)\\
& \hspace{2in} \cdot \psi_{k_1}(\eta_1-\xi_1)  \cdot \psi_{k_2}(\eta_2-\xi_2)  \diff{\eta_1} \diff{\eta_2},\\[1ex]
   V_2(\xi, t)=& \iint_{\R^2}   \left[\frac{\Tb_1'(\eta_1, \eta_2, \xi - \eta_1 - \eta_2)}{|\eta_1|^{1 - \alpha} - |\xi-\eta_1-\eta_2|^{1 - \alpha}}\right] e^{i A t \Phi(\xi,\eta_1,\eta_2)} \partial_{\eta_1}\hat h_{j_1}(\eta_1)\hat h_{j_2}(\eta_2) \hat h_{j_3}(\xi-\eta_1-\eta_2)\\
   & \hspace{2in} \cdot \psi_{k_1}(\eta_1-\xi_1)  \cdot \psi_{k_2}(\eta_2-\xi_2)  \diff{\eta_1} \diff{\eta_2},\\[1ex]
  V_3(\xi, t)=& \iint_{\R^2}   \left[\frac{\Tb_1'(\eta_1, \eta_2, \xi - \eta_1 - \eta_2)}{|\eta_1|^{1 - \alpha} - |\xi-\eta_1-\eta_2|^{1 - \alpha}}\right] e^{i A t \Phi(\xi,\eta_1,\eta_2)} \hat h_{j_1}(\eta_1)\hat h_{j_2}(\eta_2) \partial_{\eta_1}\hat h_{j_3}(\xi-\eta_1-\eta_2)\\
  & \hspace{2in} \cdot \psi_{k_1}(\eta_1-\xi_1)  \cdot \psi_{k_2}(\eta_2-\xi_2)  \diff{\eta_1} \diff{\eta_2},\\[1ex]
  V_4(\xi, t)=& \iint_{\R^2}   \left[\frac{\Tb_1'(\eta_1, \eta_2, \xi - \eta_1 - \eta_2)}{|\eta_1|^{1 - \alpha} - |\xi-\eta_1-\eta_2|^{1 - \alpha}}\right] e^{i A t \Phi(\xi,\eta_1,\eta_2)} \hat h_{j_1}(\eta_1)\hat h_{j_2}(\eta_2) \hat h_{j_3}(\xi-\eta_1-\eta_2)\\
  & \hspace{2in} \cdot \partial_{\eta_1}\psi_{k_1}(\eta_1-\xi_1)  \cdot \psi_{k_2}(\eta_2-\xi_2)  \diff{\eta_1} \diff{\eta_2}.
\end{align*}

{\bf Estimate of $V_1$}. We first denote the symbol of $V_1$ by
\begin{align*}
m(\eta_1,\eta_2,\xi)&=\partial_{\eta_1}\left[\frac{\Tb_1'(\eta_1, \eta_2, \xi - \eta_1 - \eta_2)}{|\eta_1|^{1 - \alpha} - |\xi-\eta_1-\eta_2|^{1 - \alpha}}\right] \\
&= \frac{(3 - \alpha)}{|\eta_1|^{1 - \alpha} - |\xi-\eta_1-\eta_2|^{1 - \alpha}}\cdot\Big[|\eta_1|^{1-\alpha}\eta_1-|\xi-\eta_1-\eta_2|^{1-\alpha}(\xi-\eta_1-\eta_2)\\*
&\qquad -|\eta_1+\eta_2|^{1-\alpha}(\eta_1+\eta_2)+ |\xi-\eta_1|^{1-\alpha}(\xi-\eta_1)\Big]\\
&\qquad -(1-\alpha)\frac{|\eta_1|^{-1-\alpha}\eta_1+|\xi-\eta_1-\eta_2|^{-1 - \alpha}(\xi-\eta_1-\eta_2)}{(|\eta_1|^{1 - \alpha} - |\xi-\eta_1-\eta_2|^{1 - \alpha})^2}\\*
&\qquad\cdot\Big[|\eta_1|^{3-\alpha}+|\eta_2|^{3-\alpha}+|\xi-\eta_1-\eta_2|^{3-\alpha}+|\xi|^{3-\alpha}-|\eta_1+\eta_2|^{3-\alpha}-|\xi-\eta_2|^{3-\alpha}-|\xi-\eta_1|^{3-\alpha}\Big].
\end{align*}

Writing $\upsilon_i=\eta_i-\xi_i$, $i=1,2$, we need to estimate
\begin{align*}
\bigg\|\iint_{\R^2} & m(\upsilon_1+\xi_1,\upsilon_2+\xi_2,\xi) e^{i A t \Phi(\xi,\upsilon_1+\xi_1,\upsilon_2+\xi_2)} \hat h_{j_1}(\upsilon_1+\xi_1)\hat h_{j_2}(\upsilon_2+\xi_2) \hat h_{j_3}(\xi_3-\upsilon_1-\upsilon_2) \\*
& \hspace{3.5in} \cdot \psi_{k_1}(\upsilon_1)  \psi_{k_2}(\upsilon_2) \diff{\upsilon_1} \diff{\upsilon_2}\bigg\|_{L^\infty_\xi}.
\end{align*}

Using Lemma \ref{multilinear}, we have
\begin{align*}
\|V_1\|_{L^\infty_\xi}
\lesssim~&\|\chi_{j_1,j_3}^{k_1,k_2}(\upsilon_1, \upsilon_2, \xi) m(\upsilon_1+\xi_1,\upsilon_2+\xi_2,\xi)\|_{S^\infty_{\upsilon_1,\upsilon_2}L^\infty_\xi} \\
&\cdot\| \hat \vp_{j_1}(\upsilon_1+\xi_1)\psi_{k_1}(\upsilon_1)\|_{L^2_{\upsilon_1}L^\infty_\xi} \left\|\hat \vp_{j_2}(\upsilon_2+\xi_2)\psi_{k_2}(\upsilon_2)\right\|_{L^2_{\upsilon_2}L^\infty_\xi} \|\vp_{j_3}\|_{L^\infty}
\end{align*}
where
\[
\chi_{j_1,j_3}^{k_1,k_2}(\upsilon_1,\upsilon_2,\xi)=\tilde\psi_{k_1}(\upsilon_1)\tilde\psi_{k_2}(\upsilon_2)\tilde\psi_{j_1}(\upsilon_1+\xi_1) \tilde\psi_{j_2}(\upsilon_2+\xi_2) \tilde\psi_{j_3}(\xi_3-\upsilon_1-\upsilon_2)\chi(\xi).
\]

\noindent (i) If $(\xi_1,\xi_2,\xi_3)=(\xi/3, \xi/3, \xi/3)$, then setting $w_1=\upsilon_1$, $w_2=-2\upsilon_1-\upsilon_2$ and using \eqref{SymEst} together with the rotational and scale invariance of the $S^\infty$-norm, we get that
\begin{align*}
&\|\chi_{j_1,j_3}^{k_1,k_2}(\upsilon_1,\upsilon_2,\xi) m(\upsilon_1+\xi_1,\upsilon_2+\xi_2,\xi)\|_{S^\infty_{\upsilon_1,\upsilon_2}L^\infty_\xi}\\
 =~&\|\chi_{j_1,j_3}^{k_1,k_2}(w_1, -2w_1-w_2,\xi) m(w_1+\xi_1,-2w_1-w_2+\xi_2,\xi)\|_{S^\infty_{w_1,w_2}L^\infty_\xi}\\
 \lesssim~&\|\chi_{j_1,j_3}^{k_1,k_2}(w_1, -2w_1-w_2,\xi) m(w_1+\xi_1,-2w_1-w_2+\xi_2,\xi)\|_{L^1_{w_1w_2}L^\infty_\xi}^{1/4}\\
 &\cdot\|\partial_{w_1}^2[\chi_{j_1,j_3}^{k_1,k_2}(w_1, -2w_1-w_2,\xi) m(w_1+\xi_1,-2w_1-w_2+\xi_2,\xi)]\|_{L^1_{w_1w_2}L^\infty_\xi}^{1/2}\\
 &\cdot\|\partial_{w_1}^2\partial_{w_2}^2[\chi_{j_1,j_3}^{k_1,k_2}(w_1, -2w_1-w_2,\xi) m(w_1+\xi_1,-2w_1-w_2+\xi_2,\xi)]\|_{L^1_{w_1w_2}L^\infty_\xi}^{1/4}\\
 \lesssim~&[2^{\alpha j_1+k_1}2^{(2-\alpha)j_1}]^{1/4}[2^{(\alpha-2) j_1 +k_1}2^{(2-\alpha) j_1}]^{1/2}[2^{(\alpha-2) j_1 - k_1}2^{(2-\alpha) j_1}]^{1/4}\\
 =~&2^{(j_1+k_1)/2},
\end{align*}
where we have used the estimates
\begin{align*}
\left|\frac{\chi_{j_1,j_3}^{k_1,k_2}}{|w_1+\frac\xi3|^{1 - \alpha} - |\frac\xi3+w_1+w_2|^{1 - \alpha}}\right| &\lesssim 2^{\alpha j_1-k_1},\\
\left|\chi_{j_1,j_3}^{k_1,k_2} \partial_{w_1}^2\frac{1}{|w_1+\frac\xi3|^{1 - \alpha} - |\frac\xi3+w_1+w_2|^{1 - \alpha}}\right|& \lesssim 2^{3(\alpha j_1-k_1)}2^{2[(-\alpha-1) j_1 + k_1]} =2^{(\alpha-2)j_1-k_1},\\
\left|\chi_{j_1,j_3}^{k_1,k_2}  \partial_{w_1}^2\partial_{w_2}^2\frac{1}{|w_1+\frac\xi3|^{1 - \alpha} - |\frac\xi3+w_1+w_2|^{1 - \alpha}}\right| & \lesssim 2^{5(\alpha j_1-k_1)}2^{-2j_1\alpha} 2^{2[(-\alpha-1)j_1+k_1]} =2^{(\alpha-2)j_1-3k_1}.
\end{align*}
Therefore, using \eqref{psi-L2}, \eqref{LocDis}, and \eqref{resregion}, we get that
\begin{align}
\label{est_V1}
\begin{split}
\|V_1\|_{L^\infty_\xi}&\lesssim 2^{ (j_1+k_1)/2 - j_3}(t+1)^{-1}\| \hat \vp_{j_1}(\upsilon_1+\xi_1)\psi_{k_1}(\upsilon_1)\|_{L^2_{\upsilon_1}L^\infty_\xi} \left\|\hat \vp_{j_2}(\upsilon_2+\xi_2)\psi_{k_2}(\upsilon_2)\right\|_{L^2_{\upsilon_2}L^\infty_\xi} \|\px \vp_{j_3}\|_{L^\infty}\\
&\lesssim  2^{- 0.5 j_1+k_1+0.5k_2}\|\psi_{k_1} \hat\vp_{j_1}\|_{L^\infty_{\xi}}\|\psi_{k_2} \hat\vp_{j_2}\|_{L^\infty_\xi} \Big\{(t+1)^{-1.5}\||\xi|^{1 + \alpha / 2}\hat{h}_{j_3}^{k_1,k_2}\|_{L^\infty_\xi} \\
& \hspace{1in} +(t+1)^{-1.75}\big[\||\px|^{3 \alpha / 4} P_{j_3}^{k_1,k_2}(x\px h)\|_{L^2}+\||\px|^{3 \alpha / 4} h_{j_3}^{k_1,k_2}\|_{L^2}  \big] \Big\}.
\end{split}
\end{align}

\noindent (ii) If $(\xi_1,\xi_2,\xi_3)=(\xi,\xi,-\xi)$, we use \eqref{SymEst} to obtain
\begin{align*}
&\|\chi_{j_1,j_3}^{k_1,k_2}(\upsilon_1,\upsilon_2,\xi) m(\upsilon_1+\xi_1,\upsilon_2+\xi_2,\xi)\|_{S^\infty_{\upsilon_1,\upsilon_2}L^\infty_\xi}\\
 \lesssim~& \|\chi_{j_1,j_3}^{k_1,k_2}(\upsilon_1,\upsilon_2,\xi) m(\upsilon_1+\xi_1,\upsilon_2+\xi_2,\xi)\|_{L^1_{\upsilon_1\upsilon_2}}^{1/4}\|\partial_{\upsilon_1}^2[\chi_{j_1,j_3}^{k_1,k_2}(\upsilon_1,\upsilon_2,\xi) m(\upsilon_1+\xi_1,\upsilon_2+\xi_2,\xi)]\|_{L^1_{\upsilon_1\upsilon_2}}^{1/2}\\
 &\cdot\|\partial_{\upsilon_1}^2\partial_{\upsilon_2}^2[\chi_{j_1,j_3}^{k_1,k_2}(\upsilon_1,\upsilon_2,\xi) m(\upsilon_1+\xi_1,\upsilon_2+\xi_2,\xi)]\|_{L^1_{\upsilon_1\upsilon_2}}^{1/4}
\\
 \lesssim~ &[2^{\alpha j_1+k_1}2^{(2-\alpha)j_1}]^{1/4} [2^{(\alpha-2) j_1+k_1}2^{(2-\alpha)j_1}]^{1/2} [2^{(\alpha-2) j_1-2k_2+k_1}2^{(2-\alpha)j_1}]^{1/4} \\
 =~& 2^{ \frac12j_1+k_1-\frac12k_2},
\end{align*}
where we have used the estimates
\begin{align*}
\left|\frac{\chi_{j_1,j_3}^{k_1,k_2}}{|\upsilon_1+\xi|^{1 - \alpha} - |-\xi-\upsilon_1-\upsilon_2|^{1 - \alpha}}\right| &\lesssim 2^{\alpha j_1-k_2},\\
\left|\chi_{j_1,j_3}^{k_1,k_2}\partial_{\upsilon_1}^2\frac{1}{|\upsilon_1+\xi|^{1 - \alpha} - |-\xi-\upsilon_1-\upsilon_2|^{1 - \alpha}}\right| &\lesssim 2^{3(\alpha j_1-k_2)}2^{2[(-\alpha-1) j_1+k_2]}=2^{(\alpha-2) j_1-k_2},\\
 \left|\chi_{j_1,j_3}^{k_1,k_2}\partial_{\upsilon_1}^2\partial_{\upsilon_2}^2\frac{1}{|\upsilon_1+\xi|^{1 - \alpha} - |-\xi-\upsilon_1-\upsilon_2|^{1 - \alpha}}\right| &\lesssim 2^{5(\alpha j_1-k_2)}2^{-2\alpha j_1}2^{2[(-\alpha-1)j_1+k_2]}=2^{(\alpha-2) j_1-3k_2}.
\end{align*}

Therefore, using \eqref{psi-L2}, \eqref{LocDis}, and \eqref{resregion}
\begin{align}
\begin{split}
\|V_1\|_{L^\infty_\xi}
&\lesssim 2^{0.5 j_1 - j_3 + k_1 - 0.5k_2}(t+1)^{-1}\| \hat \vp_{j_1}(\upsilon_1+\xi_1)\psi_{k_1}(\upsilon_1)\|_{L^2_{\upsilon_1}L^\infty_\xi} \left\|\hat \vp_{j_2}(\upsilon_2+\xi_2)\psi_{k_2}(\upsilon_2)\right\|_{L^2_{\upsilon_2}L^\infty_\xi} \|\px \vp_{j_3}\|_{L^\infty}\\
&\lesssim  2^{- 0.5 j_1 + 1.5k_1}\|\psi_{k_1} \hat\vp_{j_1}\|_{L^\infty_{\xi}} \|\psi_{k_2} \hat\vp_{j_2}\|_{L^\infty_\xi} \Big\{(t+1)^{-1.5}\||\xi|^{1 + \alpha / 2}\hat{h}_{j_3}\|_{L^\infty_\xi} \\
& \hspace{1in} +(t+1)^{-1.75}\big[\||\px|^{3 \alpha / 4} P_{j_3}(x\px h)\|_{L^2}+\||\px|^{3 \alpha / 4} h_{j_3}\|_{L^2}  \big] \Big\}.
\end{split}
\end{align}

{\bf Estimates of $V_2$--$V_4$.}
The estimates for $V_2$--$V_4$ are similar to $V_1$, and we omit the details here. The resulting estimates are as follows.

\noindent (i) If $(\xi_1,\xi_2, \xi_3)=(\xi/3,\xi/3, \xi/3)$, the symbol can be estimated by
\begin{align*}
&\left\|\frac{\Tb'_1(\xi_1+\upsilon_1, \xi_2+\upsilon_2, \xi_3-\upsilon_1-\upsilon_2)}{|\xi_1+\upsilon_1|^{1-\alpha}-|\xi_3-\upsilon_1-\upsilon_2|^{1-\alpha}}
\right\|_{S^\infty_{\upsilon_1\upsilon_2}L^\infty_\xi}
 \lesssim 2^{1.5 j_1+0.5k_1}.
\end{align*}

\noindent (ii) If $(\xi_1,\xi_2, \xi_3)=(\xi,\xi, -\xi)$, the symbol can be estimated by
\begin{align*}
&\left\|\chi_{j_1, j_3}^{k_1, k_2}(\upsilon_1, \upsilon_2, \xi) \frac{\Tb'_1(\xi_1+\upsilon_1, \xi_2+\upsilon_2, \xi_3-\upsilon_1-\upsilon_2)}{|\xi_1+\upsilon_1|^{1-\alpha}-|\xi_3-\upsilon_1-\upsilon_2|^{1-\alpha}}\right\|_{S^\infty_{\upsilon_1\upsilon_2}L^\infty_\xi}\\
\lesssim~&[2^{\alpha j_1+k_1}2^{(3-\alpha)j_1}]^{1/4} [2^{(\alpha-2) j_1+k_1}2^{(3-\alpha)j_1}]^{1/2} [2^{(\alpha-2) j_1-2k_2+k_1}2^{(3-\alpha)j_1}]^{1/4} \\
 =~& 2^{1.5 j_1+k_1-0.5k_2}.
\end{align*}

In either case, we have the following estimates
\begin{align}
\begin{split}
 \|V_2\|_{L^\infty_\xi}
&\lesssim 2^{-0.5 j_1+k_1}\|\eta_1 \partial_{\eta_1}\hat\vp_{j_1}(\eta_1)\|_{L^2_{\eta_1}} \|\psi_{k_2} \hat\vp_{j_2}\|_{L^\infty_\xi}\\*
& \quad \cdot  \Big\{(t+1)^{-1.5}\||\xi|^{1 + \alpha / 2}\hat{h}_{j_3}\|_{L^\infty_\xi} +(t+1)^{-1.75}\big[\||\px|^{3 \alpha / 4} P_{j_3}(x\px h)\|_{L^2}+\||\px|^{3 \alpha / 4} h_{j_3}\|_{L^2}  \big] \Big\},
\end{split}\\
\begin{split}
\|V_3\|_{L^\infty_\xi}
&\lesssim 2^{-0.5 j_1+k_1}\|\eta_3 \partial_{\eta_3}\hat\vp_{j_3}(\eta_3)\|_{L^2_{\eta_3}}\|\psi_{k_2} \hat\vp_{j_2}\|_{L^\infty_\xi}\\*
& \quad \cdot \Big\{(t+1)^{-1.5}\||\xi|^{1 + \alpha / 2}\hat{h}_{j_1}\|_{L^\infty_\xi} +(t+1)^{-1.75}\big[\||\px|^{3 \alpha / 4} P_{j_1}(x\px h)\|_{L^2}+\||\px|^{3 \alpha / 4} h_{j_1}\|_{L^2}  \big] \Big\},
\end{split}\\
\begin{split}
\|V_4\|_{L^\infty_\xi}
&\lesssim  2^{0.5 j_1+0.5k_1} \|\psi_{k_1} \hat\vp_{j_1}\|_{L^\infty_{\xi}}\|\psi_{k_2} \hat\vp_{j_2}\|_{L^\infty_\xi}\\
& \quad \cdot \Big\{(t+1)^{-1.5}\||\xi|^{1 + \alpha / 2}\hat{h}_{j_3}\|_{L^\infty_\xi} +(t+1)^{-1.75}\big[\||\px|^{3 \alpha / 4} P_{j_3}(x\px h)\|_{L^2}+\||\px|^{3 \alpha / 4} h_{j_3}\|_{L^2}  \big] \Big\}.\label{est_V4}
\end{split}
\end{align}

Finally, we take the summation over $\log_2[\varrho(t)] \leq k_1, k_2 \leq \max\{j_1, j_3\} + 1$, and combine the estimates \eqref{est_V1}--\eqref{est_V4} to get
\begin{align*}
& \Bigg\|\xi (|\xi| + |\xi|^{r + 3}) \cutoffxi(\xi, t) \iint_{\R^2} \Tb_1'(\eta_1, \eta_2, \xi - \eta_1 - \eta_2) e^{i A t \Phi(\xi,\eta_1,\eta_2)} \hat h_{j_1}(\eta_1)\hat h_{j_2}(\eta_2) \hat h_{j_3}(\xi-\eta_1-\eta_2) \\
& \hspace{2in} \cdot \bigg[\sum\limits_{k_1= \log_2[\varrho(t)]}^{\max\{j_1,j_3\}+1}\psi_{k_1}(\eta_1-\xi_1) \bigg] \cdot \bigg[\sum\limits_{k_2=\log_2[\varrho(t)]}^{\max\{j_2,j_3\}+1}\psi_{k_2}(\eta_2-\xi_2)\bigg] \diff{\eta_1} \diff{\eta_2}\Bigg\|_{L^\infty_\xi}\\
& \lesssim \left[\max\{j_1, j_3\} - \log_2[\varrho(t)]\right]^2 (t + 1)^{(r + 3) p_1}\\
& \qquad \cdot \left[\||\xi| \psi_{k_1} \hat\vp_{j_1}\|_{L^\infty_{\xi}}\||\xi| \psi_{k_2} \hat\vp_{j_2}\|_{L^\infty_\xi} + \|\eta_1 \partial_{\eta_1} \hat{\vp}_{j_1}(\eta_1)\|_{L^2_{\eta_1}} \||\xi| \psi_{k_2} \hat\vp_{j_2}\|_{L^\infty_\xi} + \||\xi| \psi_{k_1} \hat\vp_{j_1}\|_{L^\infty_\xi} \|\eta_2 \partial_{\eta_2} \hat{\vp}_{j_2}(\eta_2)\|_{L^2_{\eta_2}}\right]\\
& \qquad \cdot \Big\{(t+1)^{-1.5}\||\xi|^{1 + \alpha / 2}\hat{h}_{j_3}\|_{L^\infty_\xi} +(t+1)^{-1.75}\big[\||\px|^{3 \alpha / 4} P_{j_3}(x\px h)\|_{L^2}+\||\px|^{3 \alpha / 4} h_{j_3}\|_{L^2}  \big] \Big\}.
\end{align*}

The right-hand-side is summable with respect to $j_1, j_2, j_3$ under $|j_3-j_2|\leq 1$ and $|j_3-j_1|\leq 1$, since we can write
\begin{align*}
\begin{split}
& \||\xi|^{1 + \alpha / 2} \hat h_j\|_{L^\infty_\xi}\lesssim  (\mathbf1_{j\leq 0}2^{\alpha j / 2}+\mathbf1_{j>0}2^{(-r - 2 + \alpha / 2)j})\|h_j\|_{Z},
\end{split}
\end{align*}
and the resulting sum is integrable for $t\in (0,\infty)$.

\subsection{Resonant frequencies}
\label{8.5}

In this section, we estimate \eqref{resInt1} under the conditions that
\[
|j_1-j_3|\leq 1,\qquad |j_2-j_3|\leq 1,\qquad k_1 < \log_2[\varrho(t)], \qquad k_2 < \log_2[\varrho(t)],
\]
where $\varrho(t)$ is defined in \eqref{def_rho}, and sum the result over $k_1, k_2 < \log_2(\varrho(t))$.

The restriction of the sum to these dyadic blocks leads to a cut-off function in the integrand that is given by
\begin{align}
\label{defofcutoff}
\cutoff(\xi,\eta_1,\eta_2,t):=\psi\bigg(\frac{\eta_1 - \xi_1}{\varrho(t)}\bigg) \cdot \psi\bigg(\frac{\eta_2 - \xi_2}{\varrho(t)}\bigg).
\end{align}
The support of this cut-off function is
\[
\left\{(\eta_1,\eta_2)\in \R^2\mid |\eta_1 - \xi_1|<\frac85 \varrho(t), ~|\eta_2 - \xi_2\big|<\frac85 \varrho(t)\right\},
\]
which can be written as the union of four disjoint sets $A_1\cup A_2\cup A_3\cup A_4$, where
\begin{align*}
A_1&=\bigg\{(\eta_1,\eta_2)~~\bigg{|}~~ \bigg|\eta_1-\frac{\xi}3\bigg|<\frac85 \varrho(t),\quad \bigg|\eta_2-\frac{\xi}3\bigg|<\frac85 \varrho(t)\bigg\},\\
A_2&=\bigg\{(\eta_1,\eta_2)~~\bigg{|}~~ \big|\eta_1-\xi\big|<\frac85 \varrho(t),\quad \big|\eta_2-\xi\big|<\frac85 \varrho(t)\bigg\},\\
A_3&=\bigg\{(\eta_1,\eta_2)~~\bigg{|}~~\big|(\eta_1-\xi)\big|<\frac85 \varrho(t),\quad \big|\eta_2-(-\xi)\big|<\frac85 \varrho(t) \bigg\},\\
A_4&=\bigg\{(\eta_1,\eta_2)~~\bigg{|}~~\big|\eta_1-(-\xi)\big|<\frac85 \varrho(t),\quad \big|\eta_2-\xi\big|<\frac85 \varrho(t) \bigg\}.
\end{align*}
The regions $A_1$, $A_2$, $A_3$, $A_4$ are discs centered at $(\xi/3, \xi/3)$, $(\xi, \xi)$, $(\xi, -\xi)$, and $(-\xi, \xi)$, respectively. The region $A_1$ corresponds to space resonances $\xi = \xi/3 + \xi/3 + \xi/3$, while $A_2$, $A_3$, $A_4$ correspond to space-time resonances $\xi = \xi + \xi -\xi$.

\subsubsection{Space resonances}
For space-resonant frequencies with $(\eta_1,\eta_2)\in A_1$, we consider the time-integral of the cubic dyadic term \eqref{cubdyadic}
over $1\le \tau \le t$, write
\begin{align*}
e^{i A \tau \Phi(\xi, \eta_1,\eta_2)}=\frac{1}{i A \Phi(\xi, \eta_1,\eta_2)} \left[\partial_\tau e^{i\tau\Phi(\xi, \eta_1,\eta_2)}\right],
\end{align*}
and integrate by parts with respect to $\tau$, to get that
\[
\begin{aligned}
&\int_1^t i\xi\iint_{\R^2}  \Tb_1'(\eta_1, \eta_2, \xi - \eta_1 - \eta_2)e^{i A \tau \Phi(\xi,\eta_1,\eta_2)} \hat h_{j_1}(\eta_1, \tau)\hat h_{j_2}(\eta_2, \tau) \hat h_{j_3}(\xi-\eta_1-\eta_2, \tau) \cutoff(\xi,\eta_1,\eta_2,\tau)  \diff{\eta_1} \diff{\eta_2} \diff{\tau}\\
&\qquad =  \frac{1}{A} \bigg(J_1-\int_1^t J_2(\tau)+J_3(\tau) \diff{\tau}\bigg),
\end{aligned}
\]
where
\begin{align*}
J_1 &=  \iint_{\R^2}  \xi \frac{\Tb_1'(\eta_1, \eta_2, \xi - \eta_1 - \eta_2)}{\Phi(\xi, \eta_1,\eta_2)} \hat h_{j_1}(\eta_1, \tau)\hat h_{j_2}(\eta_2, \tau) \hat h_{j_3}(\xi-\eta_1-\eta_2, \tau)  e^{i A \tau \Phi(\xi, \eta_1,\eta_2)}\cutoff(\xi,\eta_1,\eta_2,\tau) \diff{\eta_1} \diff{\eta_2}\Big|_{\tau=1}^{\tau=t},\\[1ex]
J_2(\tau) &= \xi \iint_{\R^2} \frac{\Tb_1'(\eta_1, \eta_2, \xi - \eta_1 - \eta_2)}{\Phi(\xi, \eta_1,\eta_2)}  e^{i A \tau \Phi(\xi, \eta_1,\eta_2)} \partial_\tau \left[\hat h_{j_1}(\eta_1, \tau)\hat h_{j_2}(\eta_2, \tau) \hat h_{j_3}(\xi-\eta_1-\eta_2, \tau)\right]\cutoff(\xi,\eta_1,\eta_2,\tau) \diff{\eta_1} \diff{\eta_2},\\[1ex]
J_3(\tau)& = \xi\iint_{\R^2} \partial_\tau\cutoff(\xi,\eta_1,\eta_2,\tau) \frac{\Tb_1'(\eta_1, \eta_2, \xi - \eta_1 - \eta_2)}{\Phi(\xi, \eta_1,\eta_2)} e^{i A \tau \Phi(\xi, \eta_1,\eta_2)} \hat h_{j_1}(\eta_1, \tau)\hat h_{j_2}(\eta_2, \tau) \hat h_{j_3}(\xi-\eta_1-\eta_2, \tau)  \diff{\eta_1} \diff{\eta_2} .
\end{align*}

When $(\eta_1,\eta_2)\in A_1$, we can Taylor expand $\Tb_1' / \Phi$ around $(\xi, \xi/3, \xi/3)$ as
\begin{equation}\label{T1Phi}
\frac{\Tb_1'(\eta_1, \eta_2, \xi - \eta_1 - \eta_2)}{\Phi(\xi,\eta_1,\eta_2) }= \frac{3^{2 - \alpha} - 2^{3 - \alpha} + 1}{3 - 3^{2 - \alpha}} \xi + O\bigg(\left|\eta_1-\frac\xi3\right|^2+\left|\eta_2-\frac\xi3\right|^2\bigg).
\end{equation}
For $J_1$ and $1 \leq \tau \leq t$, we have from \eqref{T1Phi} that
\[
\begin{aligned}
&\bigg|(|\xi|+|\xi|^{r + 3})\iint_{\R^2} \cutoff(\xi,\eta_1,\eta_2,t)  \xi \frac{\Tb_1'(\eta_1, \eta_2, \xi - \eta_1 - \eta_2)}{\Phi(\xi, \eta_1,\eta_2)} \hat h_{j_1}(\eta_1, \tau)\hat h_{j_2}(\eta_2, \tau) \hat h_{j_3}(\xi-\eta_1-\eta_2, \tau)   e^{i A \tau \Phi(\xi, \eta_1,\eta_2)}\diff{\eta_1} \diff{\eta_2}\bigg|\\
&\lesssim \bigg|(|\xi|+|\xi|^{r + 3})\iint_{\R^2} \cutoff(\xi,\eta_1,\eta_2,t)  \xi^2 \hat h_{j_1}(\eta_1, \tau)\hat h_{j_2}(\eta_2, \tau) \hat h_{j_3}(\xi-\eta_1-\eta_2, \tau) e^{i\tau\Phi(\xi, \eta_1,\eta_2)} \diff{\eta_1} \diff{\eta_2}\bigg|\\
&\quad+ (|\xi|+|\xi|^{r + 3})\iint_{\R^2} \cutoff(\xi,\eta_1,\eta_2,t)   [\varrho(\tau)]^2 \left| \hat h_{j_1}(\eta_1, \tau)\hat h_{j_2}(\eta_2, \tau) \hat h_{j_3}(\xi-\eta_1-\eta_2, \tau) \right|\diff{\eta_1} \diff{\eta_2}\\
&\lesssim (\tau+1)^{2p_0+(r+2)p_1}\||\xi| \hat h_{j_1}\|_{L^\infty_\xi}\||\xi| \hat h_{j_2}\|_{L^\infty_\xi}\||\xi| \hat h_{j_3}\|_{L^\infty_\xi}\left([\varrho(\tau)]^{2}+[\varrho(\tau)]^{4}\right).
\end{aligned}
\]
If $(\eta_1,\eta_2) \in A_1$, then the number of terms in the sum over $j_1$, $j_2$, $j_3$ is of the order $\log(t+1)$, so the sum of right-hand side of this inequality over $A_1$ is bounded for all $\tau \ge 1$, and thus, in particular, taking $\tau = t$, and using the bootstrap assumptions we have
\[
|(|\xi| + |\xi|^{r + 3}) J_1| \lesssim \ve_1^3 \lesssim \ve_0,
\]

After taking the time derivative $\partial_\tau$, the term $J_2$ can be written as a sum of three terms:
\begin{align*}
&\xi \iint_{\R^2} \frac{\Tb_1'(\eta_1, \eta_2, \xi - \eta_1 - \eta_2)}{\Phi(\xi, \eta_1,\eta_2)}  e^{i A \tau \Phi(\xi, \eta_1,\eta_2)}  \left[\partial_\tau\hat h_{j_1}(\eta_1, \tau)\hat h_{j_2}(\eta_2, \tau) \hat h_{j_3}(\xi-\eta_1-\eta_2, \tau)\right]\cutoff(\xi,\eta_1,\eta_2,\tau) \diff{\eta_1} \diff{\eta_2},\\[1ex]
&\xi \iint_{\R^2} \frac{\Tb_1'(\eta_1, \eta_2, \xi - \eta_1 - \eta_2)}{\Phi(\xi, \eta_1,\eta_2)}  e^{i A \tau \Phi(\xi, \eta_1,\eta_2)}  \left[\hat h_{j_1}(\eta_1, \tau) \partial_\tau \hat h_{j_2}(\eta_2, \tau) \hat h_{j_3}(\tau, \xi-\eta_1-\eta_2)\right]\cutoff(\xi,\eta_1,\eta_2,\tau) \diff{\eta_1} \diff{\eta_2},
\\
&\xi \iint_{\R^2} \frac{\Tb_1'(\eta_1, \eta_2, \xi - \eta_1 - \eta_2)}{\Phi(\xi, \eta_1,\eta_2)}  e^{i A \tau \Phi(\xi, \eta_1,\eta_2)}  \left[\hat h_{j_1}(\eta_1, \tau)\hat h_{j_2}(\eta_2, \tau) \partial_\tau\hat h_{j_3}(\xi-\eta_1-\eta_2, \tau)\right]\cutoff(\xi,\eta_1,\eta_2,\tau) \diff{\eta_1} \diff{\eta_2}.
\end{align*}

From \eqref{phihateq}, we find that $\hat h(\xi, \tau) = e^{- i A \tau \xi |\xi|^{1 - \alpha}}\hat \vp(\xi, \tau)$ satisfies
\begin{align*}
\begin{split}
&\hat h_\tau(\xi, \tau) + i A' \xi\iint_{\R^2} \Tb_1'(\eta_1, \eta_2, \xi - \eta_1 - \eta_2) e^{i A \tau \Phi(\xi,\eta_1,\eta_2)} \hat h(\xi-\eta_1-\eta_2, \tau) \hat h(\eta_1, \tau)\hat h(\eta_2, \tau)  \diff{\eta_1} \diff{\eta_2}
\\
&\hskip2in+e^{- i A \tau \xi |\xi|^{1 - \alpha}}\widehat{\Nc_{\geq5}(\vp)}(\xi, \tau)=0,
\end{split}
\end{align*}
where $\Phi$ is given in \eqref{defPhi}.
Using this equation, the bootstrap assumptions, and Lemma \ref{sharp}, we have
\begin{align*}
\|\partial_\tau\hat h\|_{L^\infty_\xi}&\lesssim \|\xi\iint_{\R^2}\Tb_1'(\eta_1,\eta_2,\xi-\eta_1-\eta_2) e^{i A \tau \Phi(\xi,\eta_1,\eta_2)}\hat h(\xi-\eta_1-\eta_2)\hat h(\eta_1)\hat h(\eta_2)\diff\eta_1\diff\eta_2\|_{L^\infty_\xi}+\|\widehat{\mathcal N_{\geq 5}(\vp)}\|_{L^\infty_\xi}\\
&\lesssim\left\|\partial_x \bigg\{\varphi^2 |\partial_x|^{3 - \alpha} \varphi
- \varphi |\partial_x|^{3 - \alpha} (\varphi^2) + \frac 13 |\partial_x|^{3 - \alpha} (\varphi^3)\bigg\}\right \|_{L^1}+\|\mathcal N_{\geq 5}(\vp)\|_{L^1}\\
&\lesssim\|\vp\|_{H^s}^2 \cdot \sum\limits_{j=0}^{\infty} \|\vp_x\|_{W^{2,\infty}}^{2 j + 1}\\
&\lesssim \ve_1^3 (\tau + 1)^{2p_0-\frac12}.
\end{align*}
Therefore, using this estimate in the $J_2$-terms and the fact that $\ve_1^3 \lesssim \ve_0$, we get that
\[
\begin{aligned}
\left| (|\xi|+|\xi|^{r + 3}) J_2(\tau)\right|
\lesssim \sum \| h_{\ell_1}\|_{Z}\|\partial_\tau\hat h_{\ell_2}\|_{L^\infty_\xi}\| h_{\ell_3}\|_{Z}[\varrho(\tau)]^{2}
\lesssim \ve_0 (\tau + 1)^{p_0-\frac12}[\varrho(\tau)]^{2}\sum \| h_{\ell_1}\|_{Z}\| h_{\ell_3}\|_{Z},
\end{aligned}
\]
where the summation is taken over permutations $\ell_1$, $\ell_2$, $\ell_3$ of $j_1$, $j_2$, $j_3$ for $(\eta_1,\eta_2)$ in the space-resonance region $A_1$. Again, since the number of summations is of the order
$\log(\tau + 1)$, the resulting sum is integrable over $\tau \in (1, \infty)$.

As for the term $J_3$, by the definition of the function $\varrho$ in \eqref{def_rho}, we have
\[
\left|\partial_\tau\left[\psi_{\leq \log_2(\varrho(\tau))}(|\eta_1|-|\xi-\eta_1-\eta_2|) \cdot \psi_{\leq
\log_2(\varrho(\tau))}(|\eta_2|-|\xi-\eta_1-\eta_2|)\right] \right| \lesssim \varrho'(\tau)(\varrho(\tau))^{-1}\lesssim \frac{1}{\tau+1},
\]
and the area of the support of $\partial_\tau \cutoff$ is of the order of $[\varrho(\tau)]^2$.
Then, using \eqref{T1Phi}, we get that
\begin{align*}
\left|(|\xi|+|\xi|^{r + 3})J_3(\tau)\right|\lesssim (\tau+1)^{2p_0+(r+2)p_1-1} [\varrho(\tau)]^{2} \sum \|\xi\hat h_{\ell_1}\|_{L^\infty_\xi}\|\xi\hat h_{\ell_2}\|_{L^\infty_\xi}\|\xi\hat h_{\ell_3}\|_{L^\infty_\xi} \lesssim \ve_0 (\tau + 1)^{-1.5},
\end{align*}
where the summation is taken over permutations $\ell_1$, $\ell_2$, $\ell_3$ of $j_1$, $j_2$, $j_3$. This sum converges and is integrable for $\tau \in (1, \infty)$.

\subsubsection{Space-time resonances} \label{sec:spacetime}
We now use modified scattering to control the integral in $U_1$ in \eqref{defU1U2}, which corresponds to the regions of space-time resonances (cutoff by $\cutoff$, see \eqref{defofcutoff}). These terms are
\begin{align*}
A' \iint_{A_2\bigcup A_3\bigcup A_4}  &i \xi \cutoff(\xi,\eta_1,\eta_2,t) \Tb_1'(\eta_1, \eta_2, \xi - \eta_1 - \eta_2) \hat\vp_{j_1}(\eta_1)\hat\vp_{j_2}(\eta_2) \hat\vp_{j_3}(\xi-\eta_1-\eta_2) \diff{\eta_1} \diff{\eta_2}\\
& -i\xi \left[\beta_1(t)\Tb_1'(\xi, \xi, -\xi)+\beta_2(t)\Tb_1'(\xi, -\xi, \xi)+\beta_3(t)\Tb_1'(-\xi,\xi,\xi)\right]|\hat \vp(\tau,\xi)|^2\hat \vp(\tau,\xi).
\end{align*}

In the region $A_2$, we take
\[
\beta_1(t)= A' \iint_{A_2}\cutoff(\xi,\eta_1,\eta_2,t)\diff\eta_1\diff\eta_2.
\]
Then, using a Taylor expansion, we obtain that
\begin{align*}
& A' \bigg|(|\xi|+|\xi|^{r+3}) i\xi \iint_{A_2} \cutoff(\xi,\eta_1,\eta_2,t)\\
&\qquad \cdot \Big[\Tb_1'(\eta_1, \eta_2, \xi - \eta_1 - \eta_2) \hat\vp_{j_1}(\eta_1)\hat\vp_{j_2}(\eta_2) \hat\vp_{j_3}(\xi-\eta_1-\eta_2) - \Tb_1'(\xi, \xi, -\xi)|\hat \vp(\tau,\xi)|^2\hat \vp(\tau,\xi) \Big] \diff{\eta_1} \diff{\eta_2}\bigg|\\
\lesssim~& (|\xi|+|\xi|^{r+3}) |\xi| \iint_{A_2} \left|\partial_{\eta_1}\left[\Tb_1'(\eta_1, \eta_2, \xi - \eta_1 - \eta_2) \hat\vp_{j_1}(\eta_1)\hat\vp_{j_2}(\eta_2) \hat\vp_{j_3}(\xi-\eta_1-\eta_2)\right] \Big|_{\eta_1 = \eta_1'}(\xi-\eta_1)\right|\\
& \hspace {.5in}+\left|\partial_{\eta_2}\left[\Tb_1'(\eta_1, \eta_2, \xi - \eta_1 - \eta_2) \hat\vp_{j_1}(\eta_1)\hat\vp_{j_2}(\eta_2) \hat\vp_{j_3}(\xi-\eta_1-\eta_2)\right] \Big|_{\eta_2 = \eta_2'} (\xi-\eta_2)\right| \diff{\eta_1} \diff{\eta_2}\\
\lesssim~&(t+1)^{(r+2)p_1}  \|\xi\hat\vp_{j_1}\|_{L^\infty_\xi}\|\xi\hat\vp_{j_2}\|_{L^\infty_\xi}\|\xi\hat\vp_{j_3}\|_{L^\infty_\xi} [\varrho(t)]^{3}+ \sum\|\xi\hat\vp_{\ell_1}\|_{L^\infty_\xi}\|\xi\hat\vp_{\ell_2}\|_{L^\infty_\xi}\|\S\vp_{\ell_3}\|_{H^r}[\varrho(t)]^{5/2}\\
\lesssim ~& \ve_0 (t + 1)^{-1.1},
\end{align*}
where $\eta_1'$ (or $\eta_2'$) in the first inequality are some numbers between $\xi$ and $\eta_1$ (or $\eta_2$), and the summation in the last inequality is for permutations $\ell_1$, $\ell_2$, $\ell_3$ of $j_1$, $j_2$, and $j_3$. The estimates for the regions $A_3$ and $A_4$ follow by a similar argument.

Combining all the above estimates and taking the sum over $j_1, j_2, j_3$, we get that
\begin{equation}
\int_0^{\infty} \|(|\xi|+|\xi|^{r+3})U_1(\xi, t)\|_{L^\infty_\xi} \diff t\lesssim \ve_0.
\label{U1Z}
\end{equation}

\subsection{Higher-degree terms}
\label{higherorder}

In this subsection, we prove $\|(|\xi|+|\xi|^{r+3})\widehat{\mathcal{N}_{\geq5}(\vp)}\|_{L^\infty_\xi}$ is integrable in time. We start from proving the estimate for the symbol $\Tb_n$.

\begin{align*}
&\F^{-1}\left[  \Tb_n(\eta_1,\eta_2,\dotsc, \eta_{2n+1}) \psi_{j_1}(\eta_1)\psi_{j_2}(\eta_2)\dotsm\psi_{j_{2n+1}}(\eta_{2n+1})\right]\\
=~&\iiint_{\R^{2n+1}}e^{i(y_1\eta_1+y_2\eta_2+\dotsb+y_{2n+1}\eta_{2n+1})}   \bigg[\int_{\R}\frac{\prod_{j=1}^{2n+1}(1-e^{i\eta_j\zeta})}{|\zeta|^{2n + 2 - \alpha}} \diff \zeta\bigg]\psi_{j_1}(\eta_1)\psi_{j_2}(\eta_2) \dotsm \psi_{j_{2n+1}}(\eta_{2n+1}) \diff{\etab_n}\\
=~&\iiint_{\R^{2n+1}}  \bigg[\int_{\R}\frac{ (e^{iy_1\eta_1}-e^{i\eta_1(\zeta+y_1)})\dotsm(e^{iy_{2n+1}\eta_{2n+1}}-e^{i\eta_{2n+1}(\zeta+y_{2n+1})})}{|\zeta|^{2n + 2 - \alpha}}\diff \zeta\bigg]\psi_{j_1}(\eta_1)\dotsm\psi_{j_{2n+1}}(\eta_{2n+1}) \diff{\etab_n}\\
=~& \int_{\R}\frac{1 }{|\zeta|^{2n + 2 - \alpha}} \left[\F^{-1}[\psi_{j_1}](y_1)-\F^{-1}[\psi_{j_1}](\zeta+y_1)\right]\dotsm \left[\F^{-1}[\psi_{j_{2n+1}}](y_{2n+1})-\F^{-1}[\psi_{j_{2n+1}}](\zeta+y_{2n+1})\right] \diff \zeta.
\end{align*}
Then its $L^1$ norm is
\begin{align*}
&\left\|\F^{-1}\left[  \Tb_n(\eta_1,\eta_2,\dotsc, \eta_{2n+1}) \psi_{j_1}(\eta_1)\psi_{j_2}(\eta_2)\dotsm\psi_{j_{2n+1}}(\eta_{2n+1})\right]\right\|_{L^1}\\
\lesssim~&\int_{\R} 2^{j_1+\dotsb+j_{2n+1}}\frac1{|\zeta|^{2n + 2 - \alpha}}\min\{2^{-j_1},|\zeta|\}\min\{2^{-j_2},|\zeta|\}\dotsm\min\{2^{-j_{2n+1}},|\zeta|\}\diff\zeta\\
\end{align*}
Assume $\ell_1, \ell_2, \dotsc, \ell_{2n+1}$ is a permutation of $j_1, j_2, \dotsc, j_{2n+1}$ satisfying $2^{-\ell_1}\leq 2^{-\ell_2}\leq \dotsb \leq 2^{-\ell_{2n+1}}$, then we have
\begin{align*}
&\left\|\F^{-1}\left[  \Tb_n(\eta_1,\eta_2,\dotsc, \eta_{2n+1}) \psi_{j_1}(\eta_1)\psi_{j_2}(\eta_2)\dotsm\psi_{j_{2n+1}}(\eta_{2n+1})\right]\right\|_{L^1}\\
\lesssim~& \int_{|\zeta|> 2^{-\ell_{2n+1}}}\frac1{|\zeta|^{2n + 2 - \alpha}}\diff\zeta+\int_{2^{-\ell_{2n}}<|\zeta|<2^{-\ell_{2n+1}}}\frac{2^{\ell_1}}{|\zeta|^{2n + 1 - \alpha}}\diff\zeta\\
& \qquad +\dotsb+\int_{2^{-\ell_2}<|\zeta|<2^{-\ell_3}}\frac{2^{\ell_1+\dotsb+\ell_{2n - 1}}}{|\zeta|^{3 - \alpha}}\diff \zeta+\int_{|\zeta|<2^{-\ell_2}}\frac{2^{\ell_1+\dotsb+\ell_{2n}}}{|\zeta|^{2 - \alpha}}\diff \zeta\\
\lesssim~&2^{(2 - \alpha) \ell_2 + \ell_3 + \dotsb+\ell_{2n+1}}.
\end{align*}

Therefore, by Lemma \ref{multilinear}, we have
\begin{align*}
\left\|(|\xi|+|\xi|^{r+3})\widehat{\mathcal{N}_{\geq5}(\vp)}\right\|_{L^\infty_\xi}\lesssim& ~(t+1)^{(r+3)p_1}\|\mathcal{N}_{\geq 5}(\vp)\|_{L^1} \lesssim \|\vp\|_{H^1}^2\sum\limits_{n=2}^\infty \|\vp_x\|_{L^\infty}^{2n-1}.
\end{align*}
Using the dispersive estimate (Lemma~\ref{sharp}), we see the right-hand-side is integrable in $t$, and
\begin{equation}
\int_0^\infty\|(|\xi|+|\xi|^{r+3})\widehat{\mathcal N_{\geq5}(\vp)}\|_{L_\xi^\infty}\diff t\lesssim \ve_0.
\label{N5Z}
\end{equation}
Finally, the use of \eqref{U1Z} and \eqref{N5Z} in \eqref{vpZ} completes the proof of Lemma~\ref{nonlindisp}, and therefore the proof of
Theorem~\ref{main}.

\appendix
\section{Contour dynamics derivation for the GSQG front equation ($1 < \alpha < 2$)}
\label{app:contour}

In this appendix, we derive the GSQG front equation \eqref{GSQGfront0} for $1 < \alpha < 2$ by the method used in \cite{HSZ19pb} for SQG fronts.

The front-solutions considered here have unbounded velocity fields as $|y|\to\infty$, and we interpret \eqref{gsqg2} in a distributional sense.
Let $L_\alpha^1(\R^n)$ denote the space of measurable functions $f \colon \R^n \to \R$ such that
\[
\int_{\R^n} \frac{|f(\x)|}{1 + |\x|^{n + \alpha}} \diff{\x} < \infty.
\]
Then the fractional Laplacian $(-\Delta)^{\alpha / 2} \colon L_\alpha^1(\R^n) \to \mathcal{D}'(\R^n)$ can be defined by \cite{BB99}
\[
\left\langle(- \Delta)^{\alpha / 2} f, \phi\right\rangle = \int_{\R^n} f(\x) \cdot (- \Delta)^{\alpha / 2} \phi(\x) \diff{\x} \qquad \text{for all}\ \phi \in C_c^\infty(\R^n),
\]
where $(- \Delta)^{\alpha / 2} \phi$ is defined as a Fourier multiplier or singular integral \cite{Kwa17}. For $1 < \alpha < 2$, the only $\alpha$-harmonic solutions $f \in L_\alpha^1(\R^n)$ of $(- \Delta)^{\alpha / 2} f = 0$ are affine functions \cite{CDL15, Fall16}, so if we require that $\ub(\x)$ has sublinear growth in $\x$, then \eqref{gsqg2} determines $\ub$ uniquely up to a spatially uniform constant.
This constant can be removed by a transformation into a suitable reference frame, and for definiteness we set it equal to zero.

To start with, we consider the shear flow solutions $\bar{\ub}(y) = (U(y), 0)$
associated with a planar front
\begin{equation}
\label{thetabar}
\bar{\theta}(y) = \begin{cases}
\theta_+ \qquad \text{if $y > 0$},\\
\theta_- \qquad \text{if $y < 0$}.
\end{cases}
\end{equation}
In that case, \eqref{gsqg2} reduces to the equation
\begin{align}
\label{Ueqn}
|\partial_y|^\alpha U(y) = - \frac{\Theta}{g_\alpha}\delta(y),
\end{align}
where $\delta$ is the delta-distribution and $\Theta$, $g_\alpha$ are defined in \eqref{defA}.
Equation \eqref{Ueqn} has the sublinear solution
\begin{align}
\label{shear}
U(y) = \Theta C_\alpha' |y|^{\alpha - 1}, \qquad \bar{\ub}(y) = (U(y), 0),
\end{align}
where
\[
C_\alpha' = - \frac{1}{\sqrt{\pi}}\sin \left(\frac{\pi \alpha}{2}\right) \Gamma\left(\frac{1 - a}{2}\right) \Gamma\left(\frac{a}{2}\right).
\]

We now derive contour dynamics equations for GSQG ($1 < \alpha < 2$) front solutions \eqref{frontsol} whose velocity field has the asymptotic
behavior
\[
\ub(\x, t) = (\Theta C_\alpha' |y|^{\alpha - 1}, 0) + o(1) \quad \text{as}\ |y| \to \infty
\]
by decomposing the solutions into a planar shear flow and a perturbation whose velocity field approaches zero as $|y| \to \infty$.

We denote the front $y=\vp(x,t)$ by $\Gamma(t) = \partial\Omega(t)$, and consider its motion on a
time interval $0\le t \le T$ for some $T>0$. We assume that:
\begin{align*}
\begin{split}
&\text{(i) $\varphi(\cdot,t) \in C^{1}(\R)$ and $\varphi(x,t)$ is bounded
on $\R \times [0,T]$};
\\
&\text{(ii) $\varphi_x(x,t) = O(|x|^{- (\alpha - 1 + \beta)})$ as $|x|\to \infty$ for some $\beta > 0$}.
\end{split}
\end{align*}
In that case, all of the integrals in the following converge.

We choose $h > 0$ such that $- h < \inf \{\vp(x,t) : (x,t) \in \R\times[0,T]\}$, and let
\begin{align}
\label{thetatilde}
\tilde{\theta}(y) = \begin{cases} \theta_+ & \text{if $\ y > -h$},\\ \theta_- & \text{if $y < -h$},\end{cases},
\qquad
\tilde{\ub}(y) = \left(\Theta C_\alpha' |y + h|^{\alpha - 1}, 0\right),
\end{align}
be the planar front solution \eqref{thetabar}, \eqref{shear} translated to $y=-h$.

We decompose the front solution \eqref{frontsol} as
\begin{align*}
\theta(\x,t) = \tilde{\theta}(y) + \theta^*(\x,t),
\end{align*}
where $\tilde{\theta}$ is defined in \eqref{thetatilde}, and
\begin{align*}
 \theta^*(\x,t) = \begin{cases} - {\Theta}/{g_\alpha} &\text{if $-h < y < \vp(x,t)$},\\
 0 &\text{otherwise}.\end{cases}
 \end{align*}
We denote the support of $\theta^*(\cdot,t)$ by $\Omega^*(t)$.
The corresponding decomposition of the velocity field is
\begin{align*}
\ub(\x,t) = \tilde{\ub}(y) + \ub^*(\x,t),
\end{align*}
where $\tilde{\ub}$ is defined in \eqref{thetatilde}. By use of a Riesz potential representation, we find that $\ub^* = \nabla^\perp (- \Delta)^{- \alpha / 2} \theta^*$ is given by
\begin{align*}
\ub^*(\x,t) =
\Theta\, \mathrm{p.v.} \int_{\Omega^*(t)} \nabla_{\x'}^\perp \frac{1}{|\x - \x'|^{2 - \alpha}} \diff{\x'}.
\end{align*}
Writing $\x' = (x', y')$, we see that the integrand is $O\left(|x'|^{-(3 - \alpha)}\right)$ as
$|x'|\to \infty$ and compactly supported in $y'$, so this principal value integral converges absolutely at infinity.

Applying Green's theorem on a truncated region with $|x - x'| < \lambda$, and taking the limit $\lambda \to \infty$ (as in \cite{HSZ19pb}), we get that
\begin{align*}
\ub^*(\x, t) &= \left(u^*(x, y, t), v^*(x, y, t)\right),\\
u^*(x, y, t) &= - \Theta \int_\R \frac{1}{\big[(x - x')^2 + (y - \vp(x', t))^2\big]^{1 - \alpha / 2}} - \frac{1}{\big((x - x')^2 + (y + h)^2\big)^{1 - \alpha / 2}} \diff{x'},\\
v^*(x, y, t) &= - \Theta \int_\R \frac{\vp_{x'}(x', t)}{\big[(x - x')^2 + (y - \vp(x', t))^2\big]^{1 - \alpha / 2}} \diff{x'}.
\end{align*}
The integral for $u^*$ converges since the integrand is $O(|x'|^{3 - \alpha})$ as $|x'| \to \infty$, while the integral for $v^*$ converges since we assume that $\vp_{x'}(x', t) = O(|x'|^{- (\alpha - 1 + \beta)})$ as $|x'| \to \infty$ for some $\beta > 0$.

Let $\x = \left(x,\varphi(x,t)\right)$ be a point on the front and denote by
\begin{align*}
\vec{n}(\x,t) = \frac{1}{\sqrt{1+\varphi_x^2(x,t)}}(-\varphi_x(x,t),1)
\end{align*}
the unit upward normal to $\Gamma(t)$ at $\x$. The motion of the front is determined by the normal
velocity $\ub\cdot \n$, so the front $y = \vp(x, t)$ moves with the upward normal velocity
\[
\ub \cdot \n = \frac{\vp_t}{\sqrt{1 + \vp_x^2}}.
\]

Using the previous expressions for $\ub$, we therefore get that
\begin{align*}
\vp_t(x, t) & = \Theta \int_\R \frac{\vp_x(x, t) - \vp_{x'}(x', t)}{\big[(x - x')^2 + (y - \vp(x', t))^2\big]^{1 - \alpha / 2}} - \frac{\vp_{x}(x, t)}{\big[(x - x')^2 + (\vp(x, t) + h)^2\big]^{1 - \alpha / 2}} \diff{x'}\\*
& \qquad - \Theta C_\alpha' \left|\vp(x, t) + h\right|^{\alpha - 1} \vp_x(x, t)\\
& = \Theta(I_1(x, t) + I_2(x, t) + I_3(x, t)),
\end{align*}
where
\begin{align*}
I_1(x, t) &= \int_\R \frac{\vp_x(x, t) - \vp_{x'}(x', t)}{\big[(x - x')^2 + (y - \vp(x', t))^2\big]^{1 - \alpha / 2}} - \frac{\vp_x(x, t) - \vp_{x'}(x', t)}{|x - x'|^{2 - \alpha}} \diff{x'},\\
I_2(x, t) &= \int_\R \frac{\vp_x(x, t) - \vp_{x'}(x', t)}{|x - x'|^{2 - \alpha}} - \frac{\vp_x(x, t)}{[(x')^2 + 1]^{1 - \alpha / 2}} \diff{x'},\\
I_3(x, t) &= \vp_x(x, t) \bigg\{\int_\R \frac{1}{[(x')^2 + 1]^{1 - \alpha / 2}} - \frac{1}{\big[(x - x')^2 + (\vp(x, t) + h)^2\big]^{1 - \alpha / 2}}\diff{x'} - C_\alpha' \left|\vp(x, t) + h\right|^{\alpha - 1}\bigg\}.
\end{align*}

We can express $I_2$ as a Fourier multiplier. Note that $I_2=0$ if $\vp =1$, and
if $\vp(x) = e^{i\xi x}$ with $\xi\ne 0$, then
\begin{align*}
I_2(x) &= i\xi e^{i\xi x} \int_\R \left[ \frac{ 1- e^{i\xi(x'-x)}}{|x-x'|^{2 - \alpha}} - \frac{1}{((x')^2 + 1)^{\frac{2 - \alpha}{2}}} \right] \diff{x'}
\\
& = 2i\xi e^{i\xi x}\int_0^\infty\left[ \frac{ 1- \cos \xi s}{s^{2 - \alpha}} - \frac{ 1}{(s^2 + 1)^{\frac{2 - \alpha}{2}}} \right] \diff{s}.
\end{align*}
Using the identity
\begin{equation}
\int_0^\infty\left[\frac{ 1}{(s^2 + 1)^{\frac{2 - \alpha}{2}}} - \frac{1}{(s^2 + c^2)^{\frac{2 - \alpha}{2}}} \right] \diff{s} = \frac{\sqrt{\pi} \Gamma\left(\frac{1 - \alpha}{2}\right)}{2 \Gamma\left(1 - \frac{\alpha}{2}\right)} \left(1 - |c|^{\alpha - 1}\right),
\label{scaleid}
\end{equation}
with $c= 1/|\xi|$, the change of variable $s' = |\xi|s$, and identities for generalized hypergeometric functions, we get that
\begin{align*}
I_2(x)
& = - \frac{\sqrt{\pi} \Gamma\left(\frac{1 - \alpha}{2}\right)}{\Gamma\left(1 - \frac{\alpha}{2}\right)} i \xi e^{i \xi x} - i A \xi |\xi|^{1 - \alpha} e^{i \xi x},
\end{align*}
where $A$ is given in \eqref{defA}. It follows that
\[
I_2 =  - \frac{\sqrt{\pi} \Gamma\left(\frac{1 - \alpha}{2}\right)}{2 \Gamma\left(1 - \frac{\alpha}{2}\right)} \vp_x - A |\px|^{1 - \alpha} \vp_x.
\]
For $I_3$, we find after some algebra and the use of\eqref{scaleid} that
\[
I_3 = \frac{\sqrt{\pi} \Gamma\left(\frac{1 - \alpha}{2}\right)}{\Gamma\left(1 - \frac{\alpha}{2}\right)} \vp_x.
\]

Putting everything together, we get the contour dynamics equation for GSQG ($1 < \alpha < 2$) fronts
\begin{align*}
& \vp_t(x, t) + \Theta A |\px|^{1 - \alpha} \vp_x(x, t) - \Theta \int_\R \left\{\frac{\vp_x(x, t) - \vp_{x'}(x', t)}{\big[(x - x')^2 + (y - \vp(x', t))^2\big]^{1 - \alpha / 2}} - \frac{\vp_x(x, t) - \vp_{x'}(x', t)}{|x - x'|^{2 - \alpha}}\right\} \diff{x'} = 0,
\end{align*}
which is equivalent to \eqref{GSQGfront0}.
This equation also agrees with the result obtained in \cite{HS18} by a regularization method.


\begin{thebibliography}{99}


\bibitem{BeCo93} \textsc{A. L. Bertozzi and P. Constantin}. Global regularity for vortex patches. \emph{Comm. Math. Phys.}, \textbf{152}(1), 19-28, 1993.

\bibitem{BH10} \textsc{J.~Biello and J.~K.~Hunter}. Nonlinear Hamiltonian waves with constant frequency and surface waves on vorticity discontinuities. \emph{Comm. Pure Appl. Math.}, {\bf 63}(3), 303--336, 2010.

\bibitem{BHnum} \textsc{J.~Biello and J.~K.~Hunter}. Contour dynamics for vorticity discontinuities. In preparation.

\bibitem{BB99} \textsc{K.~Bogdan and T.~Byczkowski}.
Potential theory for the $\alpha$-stable Schr\"odinger operator on bounded Lipschitz domains. \emph{Studia Mathematica},  {\bf 133}, 53--92, 1999.

\bibitem{CCGS16a} \textsc{A.~Castro, D.~C\'{o}rdoba, and J.~G\'{o}mez-Serrano}, Existence and regularity of rotating global solutions for the generalized surface quasi-geostrophic equations. {\it Duke Math. J.}, {\bf 165}(5), 93--984, 2016.

\bibitem{CCGS16b} \textsc{A.~Castro, D.~C\'{o}rdoba, and J.~G\'{o}mez-Serrano}. Uniformly rotating analytic global patch solutions for active scalars. {\it Annals of PDE}, {\bf 2}(1), 1--34, 2016.

\bibitem{CCGSMZ14} \textsc{A.~Castro, D.~C\'{o}rdoba,  J.~G\'{o}mez-Serrano, and A.~Mart\'{\i}n Zamora}. Remarks on geometric properties of {SQG} sharp fronts and {$\alpha$}-patches. {\it Discrete Contin. Dyn. Syst.}, {\bf 34}(12), 5045--5059, 2014.

\bibitem{CCCGW12} \textsc{D.~Chae, P.~Constantin, D.~C\'{o}rdoba, F.~Gancedo, and J.~Wu}. Generalized surface quasi-geostrophic
    equations with singular velocities. {\it Comm. Pure Appl. Math.}, {\bf 65}(8), 1037--1066, 2012.

\bibitem{Che93} \textsc{J.~Y.~Chemin}. Persistence of geometric structures in two-dimensional incompressible fluids. \emph{Ann. Sci. Ecole. Norm. Sup.}, {\bf 26}(4), 517-542, 1993.

\bibitem{Che98} \textsc{J.~Y.~Chemin}. \emph{Perfect Incompressible Fluids}, Oxford University Press, New York, 1998.

\bibitem{CDL15} \textsc{W.~Chen, L.~D'Ambrosio and Y.~Li}. Some Liouville theorems for the fractional Laplacian. {\it Nonlinear Anal.}, {\bf 121}, 370--381, 2015.

\bibitem{CMT94a} \textsc{P.~Constantin, A.~J.~Majda, and E.~G.~Tabak}. Formation of strong fronts in the 2-D quasigeostrophic thermal active scalar. \emph{Nonlinearity}, \textbf{7}(6), 1495--1533, 1994.

\bibitem{CMT94b} \textsc{P.~Constantin, A.~J.~Majda, and E.~G.~Tabak}. Singular front formation in a model for quasigeostrophic flow. \emph{Phys. Fluids}, \textbf{6}, 9--11, 1994.

\bibitem{CCG18} \textsc{A. C\'{o}rdoba, D. C\'{o}rdoba and F. Gancedo}. Uniqueness for SQG patch solutions. {\it Trans. Amer. Math. Soc.}, {\bf Ser. B}.(5), 1--31, 2018.

\bibitem{CFR04}\textsc{D.~C\'{o}rdoba, C.~Fefferman and J.~L.~Rodrigo}. Almost sharp fronts for the surface quasi-geostrophic equation. {\it Proc.
Natl. Acad. Sci. USA}, {\bf 101}(9), 2687--2691, 2004.

\bibitem{CFMR05} \textsc{D.~C\'{o}rdoba, M.~A.~Fontelos, A.~M.~Mancho, and J.~L.~Rodrigo}. Evidence of singularities for a family of contour dynamics equations. {\it Proc. Natl. Acad. Sci.}, {\bf 102}(17), 5949--5952, 2005.

\bibitem{CGI19} \textsc{D.~C\'{o}rdoba, J.~G\'{o}mez-Serrano, and A.~D.~Ionescu}. Global solutions for the generalized SQG patch equation. \emph{Arch. Ration. Mech. Anal.}, {\bf 233}(3), 1211--1251, 2019.

\bibitem{dlHHH16} \textsc{F.~de~la~Hoz, Z.~Hassainia, and T.~Hmidi}. Doubly connected V-states for the generalized surface quasi-geostrophic equations. \emph{Arch. Ration. Mech. Anal.}, {\bf 220}(3), 1209--1281, 2016.

\bibitem{Dr1} \textsc{D.~G.~Dritschel}. The repeated filamentation of two-dimensional vorticity interfaces. \emph{J. Fluid
Mech.} \textbf{194}, 511-–547, 1988.

\bibitem{Dr2} \textsc{D.~G.~Dritschel}. Contour dynamics and contour surgery: Numerical algorithms for extended, high-resolution modelling of vortex dynamics in two-dimensional, inviscid, incompressible flows. \emph{Comput. Phys. Rep.} \textbf{10}, 77, 1989.

\bibitem{Fall16} \textsc{M.~M.~Fall}. Entire $s$-harmonic functions are affine. {\it Proc. Amer. Math. Soc.}, {\bf 144}(6), 2587--2592, 2016.

\bibitem{FLR12} \textsc{C.~Fefferman, G.~Luli, and J. Rodrigo}. The spine of an SQG almost-sharp front. {\it Nonlinearity}, {\bf 25}(2), 329--342, 2012.

\bibitem{FR11}\textsc{C.~Fefferman and J.~L.~Rodrigo}. Analytic sharp fronts for the surface quasi-geostrophic equation. {\it Comm. Math. Phys.},
{\bf 303} (1), 261--288, 2011.

\bibitem{FR12}\textsc{C.~Fefferman and J.~L.~Rodrigo}. Almost sharp fronts for SQG: the limit equations. {\it Comm. Math. Phys.}, {\bf 313}(1),
131--153, 2012.

\bibitem{FR15}\textsc{C.~Fefferman and J.~L.~Rodrigo}. Construction of almost-sharp fronts for the surface quasi-geostrophic equation. {\it Arch. Rational Mech. Anal.}, {\bf 218}, 123--162, 2015.

\bibitem{Gan08} \textsc{F.~Gancedo}. Existence for the $\alpha$-patch model and the QG sharp front in Sobolev spaces. {\it Adv. Math.}, {\bf 217}(6), 2569--2598, 2008.

\bibitem{GP18p} \textsc{F.~Gancedo and N.~Patel}. On the local existence and blow-up for generalized SQG patches. Preprint arXiv:1811.00530.


\bibitem{GS18} \textsc{J.~G\'{o}mez-Serrano}. On the existence of stationary patches. {\it Adv. Math.}, {\bf 343}, 110-140, 2019.

\bibitem{GSPSY19p} \textsc{J.~G\'{o}mez-Serrano, J.~Park, J.~Shi, and Y.~Yao}. Symmetry in stationary and uniformly-rotating solutions of active scalar equations. Preprint arXiv:1908.01722.

\bibitem{HMSZ} \textsc{J.~K.~Hunter, R.~C.~Moreno-Vasquez, J.~Shu, and Q.~Zhang}. On the approximation of vorticity fronts by the Burgers-Hilbert equation. In preparation.

\bibitem{HS18} \textsc{J.~K.~Hunter and J.~Shu}. Regularized and approximate equations for sharp fronts in the surface quasi-geostrophic equation and its generalization. {\it Nonlinearity}, {\bf 31}(6), 2480--2517, 2018.

\bibitem{HSZ18} \textsc{J.~K.~Hunter, J.~Shu, and Q.~Zhang}. Local well-posedness of an approximate equation for SQG fronts. {\it J. Math. Fluid Mech.}, {\bf 20}(4), 1967--1984, 2018.

\bibitem{HSZ18p} \textsc{J.~K.~Hunter, J.~Shu, and Q.~Zhang}. Global solutions of a surface quasi-geostrophic front equation. Preprint arXiv:1808.07631.

\bibitem{HSZ19pa} \textsc{J.~K.~Hunter, J.~Shu, and Q.~Zhang}.  Two-front solutions of the SQG equation and its generalizations. To appear in \emph{Commun. Math. Sci.}

\bibitem{HSZ19pb} \textsc{J.~K.~Hunter, J.~Shu, and Q.~Zhang}. Contour dynamics for surface quasi-geostrophic fronts. To appear in \emph{Nonlinearity}.

\bibitem{IP15} \textsc{A.~D.~Ionescu and F.~Pusateri}. Global solutions for the gravity water waves system in 2D. \emph{Invent. Math.}, {\bf 199}, 653--804, 2015.

\bibitem{IPu16} \textsc{A.~D.~Ionescu and F.~Pusateri}, Global analysis of a model for capillary water waves in two dimensions, \emph{Comm. Pure Appl. Math.}, {\bf 69}, 2016.

\bibitem{KR20pa} \textsc{C.~Khor and J.~L.~Rodrigo}. Local Existence of Analytic Sharp Fronts for Singular SQG. Preprint arXiv:2001.10412.

\bibitem{KR20pb} \textsc{C.~Khor and J.~L.~Rodrigo}. On Sharp Fronts and Almost-Sharp Fronts for singular SQG. Preprint arXiv:2001.10332.

\bibitem{KRYZ16} \textsc{A. Kiselev, L. Ryzhik, Y. Yao, and A. Zlato\v{s}}. Finite time singularity for the modified SQG patch equation. \emph{Annals of Mathematics}, \textbf{184}(3), 909--948, 2016.

\bibitem{KYZ17} \textsc{A. Kiselev, Y. Yao and A. Zlato\v{s}}. Local Regularity for the Modified SQG Patch Equation. {\it Comm. Pure Appl. Math}, {\bf 70}(7), 1253--1315, 2017.

\bibitem{Kwa17} \textsc{M.~Kwa\'snicki}. Ten equivalent definitions of the fractional laplace operator. \emph{Fractional Calculus and Applied Analysis}, {\bf 20}(1), 2017.

\bibitem{Lap17} \textsc{G.~Lapeyre}. Surface quasi-geostrophy, \emph{Fluids},  \textbf{2},  2017.

\bibitem{MB02} \textsc{A.~J.~Majda and A.~L.~Bertozzi}. \emph{Vorticity and Incompressible Flow}, Cambridge University Press, Cambridge, 2002.

\bibitem{Ped87} \textsc{J.~Pedlosky}. \emph{Geophysical Fluid Dynamics}, 2nd ed. Springer-Verlag, New York, N. Y., 1987.

\bibitem{Ro05} \textsc{J.~L.~Rodrigo}. On the evolution of sharp fronts for the quasi-geostrophic equation. {\it Comm. Pure and Appl. Math.}, {\bf 58}, 0821--0866, 2005.

\bibitem{SD14} \textsc{R.~K.~Scott and D.~G.~Dritschel}. Numerical simulation of a self-similar cascade of filament instabilities in the
Surface quasigeostrophic System. \emph{Phys. Rev. Lett.}, \textbf{112}, 144505, 2014.

\bibitem{SD19} \textsc{R.~K.~Scott and D.~G.~Dritschel}. Scale-invariant singularity of the surface quasigeostrophic patch. \emph{J. Fluid Mech.}, {\bf 863}(R2), 2019.

\bibitem{Ste70} \textsc{E.~M.~Stein}. {\it Singular integrals and differentiability properties of functions}. Princeton Mathematical Series, {\bf 30}, Princeton University Press, Princeton, N.J., 1970.

\bibitem{Ste93} \textsc{E. M. Stein}. \emph{Harmonic analysis: Real-variable Methods, Orthogonality, and Oscillatory Integrals}. Princeton Mathematical Series, {\bf 43}. Monographs in Harmonic Analysis, III. Princeton University, Princeton, N.J., 1993.

\bibitem{Zab} \textsc{N.~Zabusky, M.~H.~Hughes, and K.~V.~Roberts}. Contour dynamics for the Euler equations in two dimensions.  \emph{J.~Comput.~Phys.} \textbf{30}, 96--106, 1979.

\end{thebibliography}
\end{document}